\documentclass[11pt]{amsart}
\usepackage{graphicx}
\usepackage{amssymb}
\usepackage{amsthm}
\usepackage{amsmath}
\usepackage{amscd}
\numberwithin{equation}{section}

\newcommand{\lab}{\label}

\newcommand{\ben}{\begin{enumerate}}
\newcommand{\een}{\end{enumerate}}

\newcommand{\bea}{\begin{eqnarray}}
\newcommand{\ba}{\begin{array}}
\newcommand{\bean}{\begin{eqnarray*}}
\newcommand{\ea}{\end{array}}
\newcommand{\eea}{\end{eqnarray}}
\newcommand{\eean}{\end{eqnarray*}}
\newcommand{\beq}{\begin{equation}}
\newcommand{\eeq}{\end{equation}}
\newcommand{\bthm}{\begin{thm}}
\newcommand{\ethm}{\end{thm}}
\newcommand{\blem}{\begin{lem}}
\newcommand{\elem}{\end{lem}}
\newcommand{\bprop}{\begin{prop}}
\newcommand{\eprop}{\end{prop}}
\newcommand{\bcor}{\begin{cor}}
\newcommand{\ecor}{\end{cor}}
\newcommand{\bdfn}{\begin{dfn}}
\newcommand{\edfn}{\end{dfn}}
\newcommand{\brem}{\begin{rem}}
\newcommand{\erem}{\end{rem}}
\newcommand{\bpf}{\begin{proof}}
\newcommand{\epf}{\end{proof}}
\newcommand{\bfact}{\begin{fact}}
\newcommand{\efact}{\end{fact}}

\newcommand{\nl}{\newline}

\alph{enumii} \roman{enumiii}

\newtheorem{thm}{Theorem}[section]
\newtheorem{prop}[thm]{Proposition}
\newtheorem{lem}[thm]{Lemma}

\newtheorem{cor}[thm]{Corollary}
\newtheorem{dfn}[thm]{Definition}
\newtheorem{rem}[thm]{Remark}
\newtheorem{fact}[thm]{Fact}
\newtheorem{ex}[thm]{Example}
                    
             \def\cF{\mathcal F}

\def\tf{\tilde{f}}
\def\CV{\text{CV}}

\def\endpf{\qed}

\def\N{{\mathbb N}}                  \def\R{{\mathbb R}}
\def\C{{\mathbb C}}                  \def\oc{\hat \C}

\def\1{1\!\!1}
\def\and{\text{ and }}

\def\Comp{\text{{\rm Comp}}}        \def\diam{\text{\rm {diam}}}
\def\dist{\text{{\rm dist}}}  \def\Dist{\text{{\rm Dist}}}
\def\Crit{\text{{\rm Crit}}}
\def\Sing{\text{Sing}}        
      
\def\F{{\mathcal F}}

\def\H{\text{{\rm H}}}     \def\HD{\text{{\rm HD}}}   
\def\BD{\text{{\rm BD}}}         \def\PD{\text{{\rm PD}}}
     \def\CP{\text{CP}} \def\PCV{{\rm \text{PCV}}}
  \def\Mod{{\rm Mod}}      \def\const{\text{const}}
         \def\P{\text{{\rm P}}}     
 
             \def\Ba{\mathcal B}       

\def\L{{\mathcal L}}

\def\a{\alpha}                \def\b{\beta}             \def\d{\delta}
\def\De{\Delta}               \def\e{\varepsilon}          \def\f{\phi}
\def\g{\gamma}                \def\Ga{\Gamma}           \def\l{\lambda}
              \def\om{\omega}           
\def\Sg{\Sigma}               \def\sg{\sigma}
               \def\th{\theta}           
\def\ka{\kappa}
               \def\mh{\tilde{m}_h}

\def\abs{\prec}
\def\bi{\bigcap}              \def\bu{\bigcup}
\def\({\bigl(}                \def\){\bigr)}
\def\lt{\left}                \def\rt{\right}

\def\ld{\ldots}                        \def\^{\tilde}

\def\es{\emptyset}            \def\sms{\setminus}
\def\sbt{\subset}             \def\spt{\supset}
\def\gek{\succeq}             \def\lek{\preceq}
\def\eqv{\Leftrightarrow}     
      \def\imp{\Rightarrow}
\def\comp{\asymp}
           \def\downto{\searrow}
\def\sp{\medskip}             \def\fr{\noindent}        \def\nl{\newline}
\def\ov{\overline}            
           \def\rel{\sim}
\def\ni{\noindent}

\def\om{\omega}

\def\endpf{{\hfill $\square$}}
\def\Fa{\mathcal F}

\begin{document}

\title[]
{\bf\large {\Large M}easures and Dimensions of Julia sets of Semi-hyperbolic   
Rational Semigroups}
\date{February 15, 2011. Published in Discrete and Continuous Dynamical Systems Ser. A, 
Vol. 30, No. 1, 2011, 313--363.}
\author[\sc Hiroki SUMI]{\sc Hiroki SUMI}
%
\author[\sc Mariusz URBA\'NSKI]{\sc Mariusz URBA\'NSKI}
%
%
\thanks{Research of the first author supported in part by 
JSPS Grant-in-Aid for Scientific Research (C) 21540216. 
Research of the second author supported in part by the
NSF Grant DMS 0700831.}
\thanks{\ \newline 
\noindent Hiroki Sumi\newline 
Department of Mathematics,
Graduate School of Science,
Osaka University, 
1-1 Machikaneyama,
Toyonaka,
Osaka, 560-0043, 
Japan\newline 
E-mail: sumi@math.sci.osaka-u.ac.jp\newline 
Web: http://www.math.sci.osaka-u.ac.jp/\textasciitilde sumi/\newline
\ \newline 
Mariusz Urba\'nski\newline Department of Mathematics,
 University of North Texas, Denton, TX 76203-1430, USA\newline  
E-mail: urbanski@unt.edu\newline
Web: http://www.math.unt.edu/\textasciitilde urbanski/}


\begin{abstract}
We consider the dynamics of semi-hyperbolic semigroups generated by 
finitely many rational maps on the Riemann sphere. Assuming that the nice open set condition holds
it is proved that there exists a geometric measure on the Julia set with exponent
$h$ equal to the Hausdorff dimension of the Julia set. Both $h$-dimensional Hausdorff 
and packing measures are finite and positive on the Julia set and are mutually equivalent
with Radon-Nikodym derivatives uniformly separated from zero and infinity. All three fractal
dimensions, Hausdorff, packing and box counting are equal. It is also proved that for
the canonically associated skew-product map there exists a unique $h$-conformal measure.
Furthermore, it is shown that this conformal measure admits a unique Borel probability absolutely
continuous invariant (under the skew-product map) measure. In fact these two measures are 
equivalent, and the invariant measure is metrically exact, hence ergodic.
\end{abstract}
\maketitle

\section{Introduction}

In this paper, we frequently use the notation from \cite{hiroki1}. 
A ``rational semigroup" $G$ is a semigroup generated by a family of non-constant 
rational maps $g:\oc \rightarrow \oc$,\ where $\oc $ denotes the 
Riemann sphere,\ with the semigroup operation being functional 
composition. For a rational semigroup $G$,\ we set 
$$
F(G):=\{ z\in \oc \mid G \mbox{ is normal in a neighborhood of } z\} 
$$ 
and 
$$
J(G):=\oc \setminus F(G).
$$ 
$F(G)$ is called the Fatou set of $G$ and $J(G)$ is called the 
Julia set of $G.$ If $G$ is generated by a family $\{ f_{i}\} _{i}$,\ 
 then we write $G=\langle f_{1},f_{2},\ldots\! \rangle .$ 

The work on the dynamics of rational semigroups was initiated by 
Hinkkanen and Martin (\cite{HM}), who were interested 
in the role of the dynamics of polynomial 
semigroups while studying various one-complex-dimensional moduli spaces 
for discrete groups, and by 
F. Ren's group (\cite{ZR}), who studied such semigroups 
from the perspective of random complex dynamics.  
The theory of the dynamics of rational semigroups on $\oc $ 
has developed in many directions since the 1990s (\cite{HM, ZR, HM2,St1, St2, 
St3, SSS, sumihyp1, sumihyp2, 
hiroki1, hiroki2, hiroki3, hiroki4, hirokidc, sumikokyuroku, SU1, SdpbpI, SdpbpII, SdpbpIII, hiroki5, 
StSu, hiroki6}). 

Since the Julia set $J(G)$ of a rational semigroup $G$ 
generated by finitely many elements $f_{1},\ldots\!\!,f_{u}$ has   
{\bf backward self-similarity}, i.e., 
\begin{equation}
\label{eq:bss} 
J(G)=f_{1}^{-1}(J(G))\cup \cdots \cup f_{u}^{-1}(J(G))
\end{equation} 
 (see \cite{hiroki1}), 
%
it can be viewed as a significant generalization and extension of 
both, the theory of iteration of rational maps (see \cite{M}),  
and conformal iterated function systems (see \cite{mugdms}). 
For example, the Sierpi\'{n}ski gasket can be regarded as the Julia set 
of a rational semigroup. 
The theory of the dynamics of 
rational semigroups borrows and develops tools 
from both of these theories. It has also developed its own 
unique methods, notably the skew product approach 
(see \cite{hiroki1, hiroki2, hiroki3, hiroki4, SdpbpI, SdpbpII, SdpbpIII, hiroki5, SU1}, and \cite{SU2}). 
We remark that by (\ref{eq:bss}), 
the analysis of the Julia sets of rational semigroups somewhat
resembles 
 ``backward iterated functions systems'', however since each map 
$f_{j}$ is not in general injective (critical points), some 
qualitatively different extra effort in the cases of semigroups is needed.

The theory of the dynamics of rational semigroups is intimately 
related to that of the random dynamics of rational maps. 
For the study of random complex dynamics, the reader may 
consult \cite{FS,Bu1,Bu2,BBR,Br,GQL, MSU}. 
We remark that the complex dynamical systems 
 can be used to describe some mathematical models. For 
 example, the behavior of the population 
 of a certain species can be described as the 
 dynamical system of a polynomial 
 $f(z)= az(1-z)$ 
 such that $f$ preserves the unit interval and 
 the postcritical set in the plane is bounded 
 (cf. \cite{D}). From this point of view, 
 it is very important to consider the random 
 dynamics of polynomials.  
For the random dynamics of polynomials on the unit interval, 
see \cite{Steins}. 

The deep relation between these fields 
(rational semigroups, random complex dynamics, and (backward) IFS) 
is explained in detail in the subsequent papers
 (\cite{hirokidc, sumikokyuroku, SdpbpI, SdpbpII, SdpbpIII, hiroki5, hiroki6, hiroki7,hiroki2010prep}) of the first author.

In this paper, we investigate the Hausdorff, packing, and box dimension 
of the Julia sets of semi-hyperbolic rational semigroups 
$G=\langle f_{1},\ldots ,f_{u}\rangle $ satisfying the nice open set 
condition.  We will show that these dimensions coincide, 
that $0<\mbox{H}^{h}(J(G)),\mbox{P}^{h}(J(G))<\infty $, 
where $h$ is the Hausdorff dimension of $J(G)$ and 
$\mbox{H}^{h}$ (resp. $\mbox{P}^{h}$) denotes the 
$h$-dimensional Hausdorff (resp. packing) measure, that 
$h$ is equal to the critical exponent of the Poincar\'{e} series of 
the semigroup $G$, that there exists a unique $h$-conformal measure 
$\tilde{m}_{h}$ on the Julia set $J(\tilde{f})$ 
of the ``skew product map'' $\tilde{f}$, 
that there exists a unique Borel probability measure $\tilde{\mu }_{h}$ 
on $J(\tilde{f})$ which is absolutely continuous with respect to $\tilde{m}_{h}$, 
and that $\tilde{\mu }_{h}$ is metrically exact and equivalent with $\tilde{m}_{h}$. 
The precise statements of these results are given in 
Theorem ~\ref{Theorem A}. In order to prove these results, 
we develop and combine the idea of 
usual iteration of non-recurrent critical point maps (\cite{ncp1}), 
conformal iterated function systems (\cite{mugdms}), 
and the dynamics of expanding rational semigroups (\cite{hiroki3}). 
However, as we mentioned before, since the generators may have critical points in the Julia set, 
we need some careful treatment on the critical points in the Julia set and some 
observation on the overlapping of the backward images of the Julia set under the elements of the 
semigroup.    

Our approach develops the methods from \cite{hiroki3}, \cite{ncp1}, and \cite{ncp2}.
In order to prove that a conformal measure exists, is atomless, and, ultimately,
geometric, we expand the concepts of estimability of measures, which originally
appeared in \cite{ncp1}, we introduce a partial order in the set of critical 
points, and a stratification of invariant subsets of the Julia set. As an entirely
new tool to all \cite{hiroki3}, \cite{ncp1}, and \cite{ncp2}, we introduce the concept
of essential families of inverse branches. This concept, supported by the notion of
nice open set, is extremely useful in the realm of semi-hyperbolic rational semigroups,
at it would also (without nice open set) substantially simplified considerations in the 
expanding case. 

In the second part of the paper, devoted to proving the existence and uniqueness of an 
invariant (with respect to the canonical skew-product) probability measure equivalent
with the $h$-conformal measure, the most challenging task is to prove the uniqueness
of the latter. We do it by bringing up and elaborating the tool of Vitali relations due
to Federer (see \cite{federer}), the tool which has not come up in \cite{ncp1}, \cite{ncp2}
nor \cite{hiroki3}. We rely here heavily on deep results from \cite{federer}. The second tool,
already employed in \cite{ncp2} and subsequent papers of the second author, is the
Marco Martens method of producing $\sg$-finite invariant measures absolutely
continuous with respect to a given quasi invariant measure. We apply and develop this
method, proving in particular its validity for abstract measure spaces and not only for
$\sg$-compact measure spaces. This is possible because of our use of Banach limits 
rather than weak convergence of measures.

\sp We remark that as illustrated in \cite{sumikokyuroku, hirokidc, hiroki7}, 
 estimating the Hausdorff dimension of the Julia sets of 
 rational semigroups plays an important role when we 
 investigate random complex dynamics and its associated 
 Markov process on $\oc .$ For example, 
 when we consider the random dynamics of a compact 
 family $\Gamma $ of polynomials 
 of degree greater than or equal to two,  
 then the function $T_{\infty }:\oc \rightarrow 
 [0,1]$ of probability of tending to $\infty \in \oc $ 
 varies only on the Julia set of rational semigroup  
 generated by $\Gamma $, and under some condition, 
 this $T_{\infty }:\oc \rightarrow [0,1]$ is continuous on $\oc $ and varies precisely on $J(G).$ 
If the Hausdorff dimension of the Julia set is strictly less than two, 
then it means that $T_{\infty }:\oc \rightarrow [0,1]$ is a 
complex version of devil's staircase (Cantor function) 
(\cite{hirokidc, sumikokyuroku, hiroki7, hiroki2010prep}). 

In order to present the precise statements of the main result, 
we give some basic notations.  
For each meromorphic function $\varphi $, 
we denote by $|\varphi '(z)|_{s}$ the norm of the derivative 
with respect to the spherical metric.
Moreover, we denote by 
$CV(\varphi )$ the set of critical values of $\varphi .$ 

Given a set
$A\sbt {\mathbb C}$ and $r>0$, the symbol
$B(A,r)\index{B(A,r)@$B(A,r)$}$ denotes the Euclidean open
$r$-neighborhood of the set $A$ and  
$\diam (A)$ denotes the diameter of $A$ with respect to the Euclidean distance.  
Furthermore, given a subset $A$ of $\oc$, $B_{s}(A,r)$ denotes the 
spherical open $r$-neighborhood of the set $B$, and  
finally $\diam _{s}(A)$ denotes the diameter of $A$ with respect to the 
spherical distance. 

Let $u\in \N .$ In this paper, an element of $(\mbox{Rat})^{u}$ is called a multi-map. 

Let $f=(f_{1},\ldots ,f_{u})\in (\mbox{Rat})^{u}$ be a multi-map and let 
$G=\langle f_{1},\ldots\! ,f_{u}\rangle $ be the  
rational semigroup generated by $\{ f_{1},\ldots ,f_{u}\}  .$   
Then,\ we use the following notation. 
Let $\Sigma _{u}:=\{ 1,\ldots\! ,u\} ^{\Bbb{N}}$ be the 
space of one-sided sequences of $u$-symbols endowed with the 
product topology. This is a compact metric space. 
Let 
$\tf:\Sg_{u}\times \oc \rightarrow \Sg_{u}\times \oc $ 
be the skew product map associated with $f=(f_{1},\ldots\! ,f_{u}) $
given by the formula
$$
\tf(\om,z)=(\sg (\om ),\ f_{\om_{1}}(z)),
$$
where $(\om,z)\in \Sg _{u}\times \oc,\ \om=(\om_{1},\om_{2},\ldots\! ),$ and 
$\sg :\Sigma _{u}\rightarrow \Sg _{u}$ denotes the shift map.
We denote by $p_{1}:\Sigma _{u}\times \oc \rightarrow \Sigma _{u}$ 
the projection onto $\Sigma _{u}$ and 
$p_{2}:\Sigma _{u}\times \oc \rightarrow \oc $ the 
projection onto $\oc $. That is, 
$$
p_{1}(\om ,z)=\om \ \text{ and } \ p_{2}(\om ,z)=z.
$$  
Under the canonical identification $p_{1}^{-1}\{ \om \} \cong \oc $, 
each fiber $p_{1}^{-1}\{ \om \} $ is a Riemann surface which 
is isomorphic to $\oc .$

Let $\Sg_u^*:= \bigcup _{n\in \N }\{ 1,\ldots ,u\} ^{n} $ be the
family of finite words over the alphabet $\{1,2,\ld,u\}$. 
For every $\tau \in \Sg _{u}^{*}$,  
we denote by $|\tau | $ the only integer $n\ge 0$ such that $\tau \in
\{ 1,\ldots ,u\} ^{n}.$  
For every $\tau \in \Sg _{u}$ we set $|\tau |=\infty .$ 
In addition, for every $\tau =(\tau _{1},\tau _{2},\ldots )\in \Sg
_{u}^{\ast }\cup \Sg _{u}$ and  
$n\in \N $ with $n\leq |\tau |$, 
we set 
$$
\tau |_{n}:= (\tau _{1},\tau _{2},\ldots ,\tau _{n}) \in \Sg_{u}^{\ast}.
$$ 
For every $\tau \in \Sg _{u}^{\ast }$, we denote
$$
\hat{\tau }=\tau |_{|\tau |-1}, \  \tau _{\ast }:= \tau _{|\tau |}
$$
and 
\beq\label{referee2}
[\tau ] := \{ \om \in \Sg _{u}\mid \om |_{|\tau |}=\tau \} 
\eeq
Furthermore, for every $\om \in \Sg _{u}^{\ast }\cup \Sg _{u}$ and
all $a,b\in \N $ with $a<b\leq |\om |$, we set 
$$
\om _{a}^{b}:=(\om _{a},\ldots ,\om _{b})\in \Sg _{u}^{\ast }.
$$  
For all $\omega , \tau \in \Sigma _{u}^{\ast }$,  
we say that $\omega $ and $\tau $ are comparable 
if either (1) $|\tau |\leq |\omega |$ and $\omega \in [\tau ]$, 
or (2) $|\omega |\leq |\tau | $ and $\tau \in [\omega ].$ 
We say that $\omega ,\tau $ are incomparable if they are not
comparable. By $\tau\om\in \Sg _{u}^{\ast }$ we denoted the concatenation of the words
$\tau$ and $\om$.
For each $\om =(\om _{1},\ldots ,\om _{n})\in \Sigma _{u}^{\ast }$, let 
$$
f_{\om }:= f_{\om _{n}}\circ \cdots \circ f_{\om _{1}}.
$$ 
Fix $\tau \in \Sg _{u}^{\ast }, x\in \oc , $ and $n\in \N $. 
Suppose that $z=f_{\tau }(x)$ is not a critical value of $f_{\tau }.$ 
Then we denote by $f_{\tau ,x}^{-1}$ the inverse branch of $f_{\tau }$ 
mapping $z$ to $x.$ Furthermore, we denote by $\tf _{\tau, x}^{-|\tau |}$ 
the inverse branch of $\tf ^{|\tau |}$ such that 
$\tf _{\tau, x}^{-|\tau |}(\om ,y)=(\tau \om, f_{\tau ,x}^{-1}(y)).$
Let 
$$
\Crit(\tf):=
\bigcup_{\om \in \Sigma _{u}}\{ v\in p _{1}^{-1}\{ \om \} \mid 
v \mbox{ is a critical point of }\tf
|_{p _{1}^{-1}\{ \om \} }\rightarrow p _{1}^{-1}\{ \sigma (\om )\} \} 
\ (\subset \Sigma _{u}\times \oc )  
$$ 
be the set of critical points of $\tf.$  
For each $n\in \N $ and $(\om ,z)\in 
\Sigma _{u}\times \oc $, we set 
$$
(\tf^{n})'(\om ,z):= (f_{\om_{n}}\circ\cdots \circ f_{\om _{1}})'(z).
$$  
For each $\om\in \Sg_u$ we define 
$$
J_{\om }:=\{ z\in \oc \mid 
 \{ f_{\om _{n}}\circ \cdots \circ f_{\om _{1}}\} _{n\in \N} \mbox{ is 
 not normal in any neighborhood of } z\} 
$$
 and we then set 
$$
J(\tf):= \overline{\cup _{w\in \Sigma _{u}}\{ \om \} \times J_{\om } },
$$
 where the closure is taken in the product space 
 $\Sigma _{u}\times \oc .$ By definition,\ 
 $J(\tf)$ is compact. Furthermore,\ by Proposition 3.2 in \cite{hiroki1},\ 
 $J(\tf)$ is completely invariant under $\tf$,\ 
 $\tf$ is an open map on $J(\tf)$,\ 
 $(\tf,J(\tf))$ is topologically exact under a mild condition,\ and
 $J(\tf)$ is equal to the closure of  
 the set of repelling periodic points of 
 $\tf$ provided that $\sharp J(G)\geq 3$,\ where we say that a
 periodic point $(\om ,z)$ of $\tf$ with  
 period $n$ is repelling if the modulus of the multiplier of 
 $f_{\om _{n}}\circ \cdots \circ f_{\om _{1}}$ at $z$ is strictly larger than $1$. 
Furthermore, 
$$
p_{2}(J(\tf))=J(G).
$$ 
%
%
%
\bdfn
Let $G$ be a rational semigroup and let $F$ be a subset of $\oc .$ 
We set $G(F)= \bigcup _{g\in G}g(F)$ and 
$G^{-1}(F)= \bigcup _{g\in G}g^{-1}(F).$ 
Moreover, we set $G^{\ast }=G\cup \{ Id\} $, 
where $Id$ denotes the identity map on $\oc .$ 
Furthermore, 
let $E(G):= \{ z\in \oc \mid \# \bigcup _{g\in G}g^{-1}(\{ z\} )<\infty \} .$ 
\edfn
\bprop[Proposition 3.2(f) in \cite{hiroki1}] \label{p1h10} (topological exactness) 
Let $G=\langle f_{1},\ldots ,f_{u}\rangle $ be a finitely generated rational semigroup. 
Suppose $\# J(G)\geq 3$ and $E(G)\subset F(G).$ Then, 
the action of the semigroup $G$ on the Julia set $J(G)$ is topologically exact, 
meaning that for every non-empty open set $U\sbt J(G)$ there exist $g_1,g_2,\ld,g_n
\in G$ such that 
$$
g_1(U)\cup g_2(U)\cup \ld g_n(U)\spt J(G).
$$
\eprop

\bdfn\label{d1h3intro}
A rational semigroup $G$ is called semi-hyperbolic if and only if 
there exists an $N\in \N $ and a $\delta >0$ such that 
for each $x\in J(G)$ and $g\in G$, 
$$\deg (g:V\rightarrow B_{s}(x,\delta ))\leq N$$
for each connected component $V$ of $g^{-1}(B_{s}(x,\delta )).$ 
\edfn
\bdfn\label{d1h2intro} 
Let $f=(f_{1},\ldots ,f_{u})\in (\mbox{{\em Rat}})^{u}$ be a multi-map and 
let $G=\langle f_{1},\ldots ,f_{u}\rangle .$ We say that 
$G$ (or $f$) satisfies the open set condition if there exists a non-empty 
open subset $U$ of $\oc$ with the following two properties:
\begin{itemize}
\item[(osc1)] $f_1^{-1}(U)\cup f_2^{-1}(U)\cup\ld f_u^{-1}(U)\sbt U$,
\item[(osc2)] $f_i^{-1}(U)\cap f_2^{-1}(U)=\es$ whenever $i\ne j$.
\end{itemize}
Moreover, we say that $G$ (or $f$) satisfies the nice open set condition if in
addition the following condition is satisfied. 
\begin{itemize}
\item[(osc3)] $\exists(\a\in(0,1)) \, \forall(0< r\le 1)\, \forall(x\in\ov U) \  \
l_2(U\cap B_{s}(x,r))\ge \a l_2(B_{s}(x,r))$, where $l_{2}$ denotes the $2$-dimensional 
Lebesgue measure on $\oc .$ 
\end{itemize}
\edfn
\begin{rem}
\label{r:osc3}
Condition {\em (osc3)} is not needed if our semigroup $G$ is expanding (see \cite{hiroki3}
or note that our proofs would use only {\em (osc1)} and {\em (osc2)} under this assumption). 
Condition {\em (osc3)} is 
satisfied in the theory of conformal infinite iterated function systems (see \cite{mulms},
comp. \cite{mugdms}), where it follows from the open set condition and 
the cone condition. Moreover, condition {\em (osc3)} holds for example if the boundary of $U$ is
smooth enough; piecewise smooth with no exterior cusps suffices. 
Furthermore, {\em (osc3)} holds if $U$ is a John domain (see \cite{CJY}). 
\end{rem}
\begin{dfn}[\cite{hiroki3}]
Let $G$ be a countable rational semigroup. 
For any $t\geq 0$ and $z\in \oc $, we 
set $S_{G}(z,t):=\sum _{g\in G}\sum _{g(y)=z}| g'(y)| _{s}^{-t}$, 
counting multiplicities.  
We also set 
$S_{G}(z):= \inf \{ t\geq 0: S_{G}(z,t)<\infty \} $ 
(if no $t$ exists with $S_{G}(z,t) <\infty $, then we set 
$S_{G}(z):=\infty $). Furthermore, 
we set $s_{0}(G):= \inf \{ S_{G}(z): z\in \oc \} .$  
This $s_{0}(G)$ is called the {\bf critical exponent of the 
Poincar\'{e} series} of $G.$ 
\end{dfn}
\begin{dfn}[\cite{hiroki3}]
Let $f=(f_{1},\ldots ,f_{u})\in (\mbox{{\em Rat}})^{u}$, $t\geq 0$, and $z\in \oc .$
 We put
$T_{f}(z,t):=\sum _{\om \in \Sigma _{u}^{\ast}}
\sum _{f_{\om }(y)=z}| f_{\om }'(y)| _{s}^{-t}$, 
counting multiplicities.   
Moreover, we set 
$T_{f}(z):=\inf \{ t\geq 0:T_{f}(z,t)<\infty \} $ 
(if no $t$ exists with $T_{f}(z,t)<\infty $, then we set 
$T_{f}(z)=\infty $). 
Furthermore, we set 
$t_{0}(f):= \inf \{ T_{f}(z): z\in \oc \} .$  
This $t_{0}(f)$ is called the {\bf critical exponent of 
the Poincar\'{e} series} of $f=(f_{1},\ldots ,f_{u})\in 
(\mbox{{\em Rat}})^{u}.$ 
\end{dfn}
\begin{rem}
\label{strem}
Let $f=(f_{1},\ldots ,f_{u})\in (\mbox{{\em Rat}})^{u}$, $t\geq 0$ , 
$z\in \oc $ and let 
$G=\langle f_{1},\ldots ,f_{u}\rangle .$ 
Then, 
$S_{G}(t,z)\leq T_{f}(t,z), S_{G}(z)\leq T_{f}(z),$ and 
$s_{0}(G)\leq t_{0}(f).$ Note that 
 for almost every $f\in (\mbox{{\em Rat}})^{u}$ with 
 respect to the Lebesgue measure, 
 $G=\langle f_{1},\ldots ,f_{u}\rangle $ is a free 
 semigroup and so we have 
 $S_{G}(t,z)=T_{f}(t,z), S_{G}(z)=T_{f}(z), $ and 
$s_{0}(G)=t_{0}(f).$  
\end{rem}
\bdfn
Let $\varphi :J(\tf )\rightarrow \R $ be a function. 
Let $\nu $ be a Borel probability measure on $J(\tf ).$ 
We say that $\nu $ is a $\varphi$-conformal measure 
for the map $\tf :J(\tf )\rightarrow J(\tf )$ if 
for each Borel subset $A$ of $J(\tf )$ such that 
$\tf |_{A}:A\rightarrow J(\tf )$ is injective, we have 
$$\nu (\tf (A))=\int _{A}\varphi \ d\nu .$$ 
A $|\tf'|_{s} ^{t}$-conformal measure $\nu $ is sometimes called 
a $t$-conformal measure.  
When $J(G)\subset \C $,  
a $|\tf'|^{t}$-conformal measure is also sometimes called a $t$-conformal measure.   
\edfn 
\bdfn
Let $G=\langle f_{1},\ldots ,f_{u}\rangle $ and let $t\geq 0.$  
For all $z\in \oc \setminus G^{\ast }(\bigcup _{j=1}^{u}\CV (f_{j}))$, 
we set $$P_{z}(t):= \limsup _{n\rightarrow \infty }\frac{1}{n}\log \sum _{|\om |=n}
\sum _{x\in f_{\om }^{-1}(z)}|f_{\om }'(x)|_{s}^{-t}.$$  
\edfn
The main result of this paper is the following. 
\bthm[see Lemma~\ref{l1h63}, Lemma~\ref{l1h33}, Theorem~\ref{t1h61}, 
Corollary~\ref{t1h61cor1} and Theorem~\ref{t4h65}]
\label{Theorem A}
Let $f=(f_{1},\ldots ,f_{u})\in (\mbox{{\em Rat}})^{u}$ be a multi-map. 
Let $G=\langle f_{1},\ldots ,f_{u}\rangle $.
Suppose that there exists an element $g$ of $G$ such that $\deg (g)\geq 2$, 
that each element of {\em Aut}$(\oc)\cap G$ (if this is not empty) is loxodromic, 
that $G$ is semi-hyperbolic, and that $G$ satisfies the nice open set condition. 
Then, we have the following. 
\begin{itemize}
\item[(a)] $J(G)\cap \ov{G^{\ast }(\bigcup _{j=1}^{u}\CV(f_{j}))} $ is nowhere dense in 
$J(G)$ and, for each $t\geq 0$, the function $z\mapsto P_{z}(t)$ is constant 
throughout a neighborhood of $J(G)\setminus \ov{G^{\ast }(\bigcup _{j=1}^{u}\CV(f_{j}))}$ in $\oc .$  
Denote this constant by $P(t)$. 
\item[(b)] The function $t\mapsto P(t)$ has a unique zero. This zero is denoted by $h=h(f)$. 
\item[(c)] There exists a unique $|\tf '|_{s}^{h}$-conformal measure $\mh $ 
for the map 
$\tf :J(\tf)\rightarrow J(\tf ).$ 
\item[(d)] Let $m_{h}:= \mh\circ p_{2}^{-1}.$ 
Then there exists a constant $C\geq	1$ such that 
$$C^{-1}\leq \frac{m_{h}(B_{s}(z,r))}{r^{h}}\leq C$$
for all $z\in J(G)$ and all $r\in (0,1].$
\item[(e)] 
$h(f)=\HD(J(G))=\PD(J(G))=\BD(J(G))$, where $\HD ,\PD, \BD $ denote the 
Hausdorff dimension, packing dimension, and box dimension, respectively, 
with respect to the spherical distance in $\oc .$  
Moreover, for each $z\in J(G)\setminus \ov{G^{\ast }(\bigcup _{j=1}^{u}\CV(f_{j}))}$, 
we have  $h(f)=T_{f}(z)=t_{0}(f)=S_{G}(z)=s_{0}(G).$ 
\item[(f)] 
Let $\H^{h}$ and $\P^{h}$ be the $h$-dimensional Hausdorff dimension and 
$h$-dimensional packing measure respectively. Then, 
all the measures $\H^{h}, \P^{h}$, and $m_{h}$ are mutually equivalent 
with Radon-Nikodym derivatives uniformly separated away from zero and infinity.
\item[(g)] $0<\H^{h}(J(G)),\P^{h}(J(G))<\infty .$ 
\item[(h)] There exists a unique Borel probability $\tf$-invariant measure 
$\tilde{\mu} _{h}$ on 
$J(\tf )$ which is absolutely continuous with respect to $\mh .$ The measure 
$\tilde{\mu} _{h}$ is metrically exact and equivalent with $\mh .$   
\end{itemize}
\ethm 
The proof of Theorem~\ref{Theorem A} will be given in the following 
Sections~\ref{s:Pre}--\ref{s:Inv}. In Section~\ref{s:Ex}, we give some examples of 
semi-hyperbolic rational semigroups satisfying the nice open set
condition.  

\

\fr ACKNOWLEDGMENT: We wish to thank the anonymous referee  for
his/her valuable comments and suggestions which improved the final
exposition of our paper.

\section{Preliminaries}
\label{s:Pre}
\subsection{Distortion and Measures}

\fr All the points (numbers) appearing in this paper are complex
unless it is clear from the context that they are real. In
particular $x$ and $y$ are always assumed to be complex numbers and
not the real and imaginary parts of a complex number. 

\bthm\label{kdt1/4}
(Koebe's ${1\over 4}$-Theorem) 
If $z\in {\mathbb C}$, $r>0$ and $H:B(z,r)\to
{\mathbb C}$  is an arbitrary univalent analytic function, then
$H(B(z,r))\spt B(H(z),4^{-1}|H'(z)|r)$.
\ethm

\bthm\label{kdt1E} (Koebe's Distortion Theorem, I) There exists a function  $k:[0,1)\to [1,\infty)$ such that for
any $z\in {\mathbb C}$, $ r >0$, $ t\in [0,1)$ and any univalent
analytic function $H:B(z,r)\to {\mathbb C}$ we have that
$$
\sup\{|H'(w)|:w\in B(z,tr)\} \le k(t) \inf\{|H'(w)|:w\in
B(z,tr)\}.
$$
We put $K=k(1/2)$.
\ethm


 The following is a straightforward consequence of these two
distortion theorems.

\blem\lab{lncp12.9.} Suppose that $D\sbt {\mathbb C}$ is an open
set, $z\in D$ and $H:D\to {\mathbb C}$ is an analytic map which has
an analytic inverse $H_z^{-1}$ defined on $B(H(z),2R)$ for some
$R>0$. Then for every $0\le r\le R$
$$
B(z,K^{-1}r|H'(z)|^{-1})\sbt H_z^{-1}(B(H(z),r))\sbt
B(z,Kr|H'(z)|^{-1}).
$$
\elem


  We also use the following more geometric versions of Koebe's
Distortion Theorems involving moduli of annuli.

\bthm\label{kdt2E}(Koebe's Distortion Theorem, II) There exists a function $w:(0,+\infty)\to [1,\infty)$ such that
for any two open topological disks $Q_1\sbt Q_2\subset \C $ with
$\text{{\em Mod}}(Q_2\sms Q_1)\ge t$ and any univalent analytic function
$H:Q_2\to {\mathbb C}$ we have 
$$
\sup\{|H'(\xi)|:\xi\in Q_1\} \le w(t) \inf\{|H'(\xi)|:\xi\in Q_1\}.
$$
\ethm


\bdfn\lab{dcomp} If $H:D\to {\mathbb C}$ is an analytic map, $z\in
{\mathbb C}$, and $r>0$, then by
$$
\Comp(z,H,r)
$$
we denote the connected component of $H^{-1}(B(H(z),r))$ that
contains $z$. 
\edfn

\
\ni Given an  analytic function $H$ defined  throughout a region
$D\sbt {\mathbb C}$, we  put
$$\index{critical points@$\Crit(H)$}
\Crit(H)=\{z\in D:H'(z)=0\}.
$$

\ni Suppose now that $c$ is a critical point of an analytic map
$H:D\to {\mathbb C}$. Then there exists $R=R(H,c)>0$ and
$A=A(H,c)\ge 1$ such that
$$
A^{-1}|z-c|^{q} \le |H(z)-H(c)| \le A|z-c|^{q}
$$
and
$$
A^{-1}|z-c|^{q-1} \le |H'(z)| \le A|z-c|^{q-1}
$$
for every $z\in \Comp(c,H,R)$, and that
$$
H(\Comp(c,H,R))=B(H(c),R), 
$$ 
where $q=q(H,c)$ is the order of
$H$ at the critical point $c$. In particular
$$
\Comp(c,H,R)\sbt B(c,(AR)^{1/q}).
$$  Moreover, by taking
$R>0$ sufficiently small, we can ensure that the above two 
inequalities hold for every $z\in B(c,(AR)^{1/q})$ and the ball
$B(c,(AR)^{1/q})$ can be expressed as a union of $q$ closed
topological disks with piecewise smooth boundaries and mutually disjoint
interiors such that the map $H$ restricted to each of these
interiors, is injective.

In the sequel we require  the following technical lemma proven in
\cite{ncp1} as Lemma~2.11.

\blem\lab{lncp12.11} 
Let $H: D\rightarrow \C $ be an analytic function.  
Suppose that an analytic map $Q\circ H:D\to
\mathbb C$, a radius $R>0$ and a point $z\in D$ are such that
$$
\Comp(H(z),Q,2R)\cap \Crit(Q)=\es \  \text{{\rm and }} \
\Comp(z,Q\circ H,R)\cap \Crit(H)\ne\es.\ \ \mbox{{\em (a)}}
$$
If $c$ belongs to the last intersection, 
$A=A(H,c),$ and $q$ is the order of $H$ at $c$, 
 and
$$
\diam\(\Comp(z,Q\circ H,R)\)\le (AR(H,c))^{1/q}, \ \ \ \ \ \ \ \ \ 
\ \ \ \ \ \ \ \ \ \ \ \ \ \ \ \ \ \ \ \ \ \ \ \ \ \ \ \ \ \ \ \mbox{{\em (b)}}
$$
then
$$
|z-c|\le KA^2|(Q\circ H)'(z)|^{-1}R.
$$
\elem
\fr{\sl Proof.} In view of Lemma~\ref{lncp12.9.}
$$
\Comp(H(z),Q,R)\sbt B(H(z),KR|Q'(H(z))|^{-1}).
$$
So, since $H(c)\in \Comp(H(z),Q,R)$, we get $$H(c)\in
B(H(z),KR|Q'(H(z))|^{-1}).$$ Thus, using this  and  (b) we obtain
$$\begin{aligned}
A^{-1}|z-c|^{q} &\le |H(z)-H(c)| \\
              &\le KR|Q'(H(z))|^{-1}\\
              & =KR|(Q\circ H)'(z)|^{- 1}|H'(z)| \\
&\le KR |(Q\circ H)'(z)|^{-1} A|z-c|^{q-1}. \end{aligned}
$$
So, $|z-c|\le KA^2|(Q\circ H)'(z)|^{-1}R$. \endpf

\

\fr Developing the appropriate concepts from \cite{ncp1} we now shall define the 
notions of estimabilities (upper, lower and strongly lower) of measures, and 
we shall prove some of its properties and consequences.

\bdfn\label{dncp12.3.} 
Suppose $m$ is a finite Borel measure on Borel set $X\sbt\R^n$.
\begin{itemize}
\item[(1)] (Upper Estimability)
The measure $m$ is said to be upper $t$-estimable at a point $x\in X$
if there exist $L>0$ and $R>0$ such that
$$
m(B(x,r))\le Lr^t
$$
for all $0\le r\le R$. The number $L$ is referred to as the upper estimability
constant of the measure $m$ at $x$ and the number $R$ is referred to as the upper 
estimability radius of the measure $m$ at $x$. If there exists an $L>0$ and an $R>0$ such
that the measure $m$ is upper $t$-estimable at each point of $X$ with the upper 
estimability constant $L$ and the upper estimability radius $R$, the measure
$m$ is said to be uniformly upper $t$-estimable.
\item[(2)] (Lower Estimability)
The measure $m$ is said to be lower $t$-estimable at a point $x\in X$
if there exists an $L>0$ and an $R>0$ such that
$$
m(B(x,r))\ge Lr^t
$$
for all $0\le r\le R$. The number $L$ is referred to as the lower estimability
constant of the measure $m$ at $x$ and the number $R$ is referred to as the lower 
estimability radius of the measure $m$ at $x$. If there exists an $L>0$ and an $R>0$ such
that the measure $m$ is lower $t$-estimable at each point of $X$ with the lower  
estimability constant $L$ and the lower estimability radius $R$, then the measure
$m$ is said to be uniformly lower $t$-estimable.
\item[(3)] (Strongly Lower Estimability)
The measure $m$ is said to be strongly lower $t$-estimable at a point $x\in X$
if there exists an $L>0$, a $\lambda \in (0,\infty )$, and an $R>0$ such that
$$
m(B(y,\lambda  r))\ge Lr^t  
$$
for every $y\in B(x,R)$ for all $0\le r\le R$. The number $L$ is referred to as 
the lower estimability constant of the measure $m$ at $x$, the number $R$ is referred 
to as the lower estimability radius of the measure $m$ at $x$, and $\lambda $ is referred 
to as the lower estimability size of the measure $m$ at $x$. If there exists an $L>0$, 
a $\lambda $, and an $R>0$ such
that the measure $m$ is strongly lower $t$-estimable at each point of $X$ with the lower 
estimability constant $L$, the lower estimability radius $R$, and the lower estimability 
size $\lambda $, then the measure $m$ is said to be uniformly strongly
lower $t$-estimable.
\end{itemize}
\edfn

\ 

\ni Suppose $U$ and $V$ are open subsets of $\C $, $z$ is a point of
$U$,  and $H:U\to V$ is an analytic map. 
Fix $t\ge 0$. A pair $(m_1,m_2)$ of finite Borel measures respectively on $U$ and 
$V$ is called upper $t$-conformal for $H$ if and only if
$$
m_2(H(A))\ge \int_A|H'|^tdm_1
$$
for all Borel sets $A\sbt U$ such that the restriction $H|_A$ is injective. The pair
$(m_1,m_2)$ is called $t$-conformal if the above inequality sign can be replaced by 
equality. We will need the following lemmas.

\

\blem\label{lku4.8} 
Suppose $U$ and $V$ are open subsets of $\C $ and $H:U\to V$ is an analytic map
which has an analytic inverse $H_z^{-1}$ defined on $B(H(z),2R)$ for some
$R>0$. Suppose $(m_1,m_2)$ 
is a $t$-conformal pair of measures for $H$. Suppose $m_2$ is strongly lower $t$-estimable 
at $H(z)$ with estimability constant $L$, estimability radius $0<r_{0}\le R/2$, and
the lower estimability size $\lambda \le 1$. Then the measure $m_1$ is strongly lower 
$t$-estimable at $z$ with lower estimability constant $L$, lower estimability radius
$K^{-1}|H'(z)|^{-1}r_{0}$, and lower estimability size $K^2\lambda $.
\elem

{\sl Proof.} 
Let $0\leq r\leq r_{0}.$ 
Consider $x\in B(z,K^{-1}r|H'(z)|^{-1})$. Then 
by Lemma~\ref{lncp12.9.}
$H(x)\in B(H(z),r)$ and
therefore $m_2(B(H(x),\lambda  r))\ge Lr^t$. Since
$$
B(H(x),\lambda  r)\sbt B(H(z),2r)\sbt B(H(z),R)
$$ 
we have
$$
H_z^{-1}\(B(H(x),\lambda  r)\)\sbt B(x,K\lambda 
r|H'(z)|^{-1})=B(x,K^2\lambda ( K^{- 1}|H'(z)|^{-1}r)).
$$
Thus 
$$
m_1\(B(x,K^2\lambda (K^{-1}|H'(z)|^{-1}r))
\ge K^{-t}|H'(z)|^{-t}Lr^t= L(K^{- 1}|H'(z)|^{-1}r)^t.
$$ 
The proof is finished. \endpf

\

\blem\label{lku4.9} 
Suppose $U$ and $V$ are open subsets of $\C $ and $H:U\to V$ is an analytic map.
Let $c\in U$ be a critical point of $H$ of order $q$. Suppose $(m_1,m_2)$ 
is a $t$-conformal pair of measures for $H$. If $m_2$ is lower 
$t$-estimable at $H(c)$ with estimability constant $L$ and estimability radius 
$0<T\le R(H,c)$, then the measure $m_1$ is lower $t$-estimable at $c$ with 
estimability constant  $A(H,c)^{-2t}L$ and estimability radius $(A(H,c)T)^{1/q}$.
\elem

{\sl Proof.} Put $A=A(H,c)$. Let $0<r\leq T.$ Notice that $B(H(c),r)=H(\Comp(c,H,r))$. If $x\in
\Comp(c,H,r)$, then  $ A^{-1}|x-c|^{q} \leq  |H(x)-H(c)| <r$
which implies that $x\in B(c,(Ar)^{1/q})$. Thus $B(H(c),r)\sbt
H(B(c,(Ar)^{1/q})$ and therefore
$$
\begin{aligned}
Lr^t 
&\le m_2(B(H(c),r))\\
&\le m_2\(H(B(c,(Ar)^{1/q}))\)\\
&\le \int_{B(c,(Ar)^{1/q})}|H'(z)|^t\,dm_1(z) \\
&\le \int_{B(c,(Ar)^{1/q})}A^t(|z-c|^{q-1})^t\,dm_1(z) \\
&\le A^t(Ar)^{{q-1\over q}t}m_1(B(c,(Ar)^{1/q})).
\end{aligned}
$$
So, $m_1(B(c,(Ar)^{1/q})) \ge A^{-2t}L((Ar)^{1/q})^t$. \endpf

\blem\lab{lku4.10} 
Suppose $U$ and $V$ are open subsets of $\C $ and $H:U\to V$ is an analytic map.
Let $c\in U$ be a critical point of $H$ of order $q$. Suppose $(m_1,m_2)$ 
is an upper $t$-conformal pair of measures for $H$ such that $m_1(\{ c\} )=0$. If 
$m_2$ is upper $t$-estimable at $H(c)$ with 
estimability constant $L$ and estimability radius $0<T<R(H,c)$, then the measure
$m_1$ is upper $t$-estimable at $c$ with estimability constant 
$q(2A(H,c)^2)^t(2^{t/q}-1)^{-1}L$ and estimability radius $(A(H,c)^{-1}T)^{1/q}$.
\elem
{\sl Proof.} Put $A=A(H,c)$. Take any $0<s\le T$. then $H(B(c,(A^{-1}s)^{1/q}))\sbt 
B(H(c),s)$. Set $R(c,a,b)=\{z:a\le |z-c|<b\}$ and abbreviate
$
R(c,2^{-1/q}(A^{-1}s)^{1/q},(A^{-1}s)^{1/q}) 
$ to $R(c).$  
Using our assumptions and the fact that the map $H$ is $q$-to-$1$ on 
$B(c,(A^{-1}s)^{1/q})$, we obtain
$$
\begin{aligned}
Ls^t &\ge m_2(B(H(c),s))\\
     &\ge m_2\(H(B(c,(A^{-1}s)^{1/q}))\)\\
     &\ge q^{-1}\int_{B(c,(A^{-1}s)^{1/q})}|H'(z)|^t\,dm_1(z) \\
&\ge q^{-1}\int_{R(c)}|H'(z)|^t\,dm_1(z) \\
&\ge q^{-1}A^{-t}(2^{-1}A^{-1}s)^{{q-1\over q}t}m_1(R(c)).
\end{aligned}
$$
So, $m_1\(R(c,2^{-1/q}(A^{-1}s)^{1/q},(A^{-1}s)^{1/q})\)
\le q2^{t(1-{1\over q})}A^{2t}L((A^{-1}s)^{1/q})^t$ and
therefore for any $0<r\leq T$, 
$$
\begin{aligned}
m_1\(B(c, (A^{-1}r)^{1/q})\)
& =  m_1\(\bu_{n=0}^\infty R\(c,2^{-{n+1\over q}}(A^{-1}r)^{1/q},
         2^{-{n\over q}}(A^{-1}r)^{1/q})\) \\
&=  \sum_{n=0}^\infty m_{1} \(R(c,2^{-{1\over q}}(A^{-1}2^{-n}r)^{1/q},(A^{-1}2^{-
         n}r)^{1/q})\)\\
&\le q(2^{1-{1\over q}}A^2)^tL\sum_{n=0}^\infty (A^{-1}2^{-n}r)^{t/q} \\ 
&=q(2^{1-{1\over q}}A^2)^t{L\over 1-2^{-{t\over q}}}((A^{-1}r)^{1/q})^t \\
&=q(2A^2)^t(2^{t/q}-1)^{-1}L((A^{-1}r)^{1/q})^t.
\end{aligned}
$$
The proof is finished. \endpf

\

\blem\lab{lku4.11} 
Suppose $U$ and $V$ are open subsets of $\C $ and $H:U\to V$ is an analytic map.
Let $c\in U$ be a critical point of $H$ of order $q$. Suppose $(m_1,m_2)$ 
is a $t$-conformal pair of measures for $H$. If $m_2$ is strongly lower $t$-estimable 
at $H(c)$ with estimability constant $L$, estimability radius $0<T<R(H,c)/3$, and
lower estimability size $\lambda $. Then the measure $m_1$ is strongly lower 
$t$-estimable at $c$ with lower estimability constant 
$\^L=L\min\{K^{-t},(A(H,c)^2\lambda )^{{1-q\over q}t}\}$, lower estimability radius
$(A^{-1}T)^{1/q}$, and lower estimability size $\^\lambda =(2^{q+1}KA^2\lambda )^{1/q}$.
\elem

{\sl Proof.} As in the proof of the previous lemma put $A=A(H,c)$. 
Let $0<r\leq T$ and 
let 
$x\in B(c,(A^{-1}r)^{1/q})$. If $\^\lambda (A^{-1}r)^{1/q}\ge 2|x-c|$, then
$$\begin{aligned}
B(x,\^\lambda (A^{-1}r)^{1/q})&\spt B(c,\^\lambda (A^{-1}r)^{1/q}/2)\\
& =B(c,(2K)^{1/q}(A\lambda  r)^{1/q}) \\
   & \spt B(c,(A\lambda  r)^{1/q}).
\end{aligned}$$
It follows from the assumptions that $m_{2}$ is lower $t$-estimable at $H(c)$ with
lower estimability constant $\lambda ^{-t}L$ and lower estimability radius
$\lambda T$. Therefore, in view of Lemma~\ref{lku4.9} the critical 
point $c$ is lower $t$-estimable with lower estimability constant $A^{-2t}\lambda ^{-t}L$  
and lower estimability radius $(A\lambda  T)^{1/q}$. Thus
\begin{equation}\lab{1042606}
\begin{aligned}
m_1\(B(x,\^\lambda (A^{-1}r)^{1/q})\)&\ge A^{-2t}\lambda ^{-t}L(A\lambda 
r)^{t/q}\\
 & =(A^2\lambda )^{{1- q\over q}t}L((A^{-1}r)^{1/q})^t.
\end{aligned}
\end{equation}
So, suppose that
\begin{equation}\lab{2042606}
\^\lambda (A^{-1}r)^{1/q}< 2|x-c|.
\end{equation}
Since $c$ is a critical point we have $$|H'(x)|\ge
A^{-1}|x-c|^{q-1}\ge A^{-1}\^\lambda ^{q- 1}(A^{-1}r)^{q-1\over
q}2^{1-q},$$ which means that
\begin{equation}\lab{3042606}
\begin{aligned}
\^\lambda (A^{-1}r)^{1/q} & \ge
A^{-1}\^\lambda ^qA^{-1}r2^{1-q}|H'(x)|^{-1}\\
 & = 4K\lambda  r|H'(x)|^{-1}\ge
K\lambda  r|H'(x)|^{-1}.
\end{aligned}
\end{equation}
In view of (\ref{2042606})
$$
|H(x)-H(c)|\ge A^{-1}|x-c|^{q}\ge
A^{-1}2^{-q}\^\lambda ^{q}A^{-1}r=2K\lambda  r\ge 2\lambda  r
$$
which implies that
\begin{equation}\lab{4042606}
H(c)\notin B(H(x),2\lambda  r).
\end{equation}
Since $|H(x)-H(c)|\le A|x-c|^{q}\le R(H,c)/3$, we have $B(H(x),2\lambda 
r)\sbt B(H(c),R(H,c))$. So, (\ref{4042606}) implies the existence of a
holomorphic inverse branch $H_x^{-1}:B(H(x),2\lambda  r)\to \C $ of $H$ which sends $H(x)$ to $x$. Since, by the assumptions, the measure
$m_2$ is lower $t$-estimable at $H(x)$ with lower estimability constant
$\lambda ^{-t}L$ and lower estimability radius $\lambda  r$, it follows from 
the proof of Lemma~\ref{lku4.8} that the measure $m_1$ is lower $t$-estimable at $x$ 
with lower estimability constant $K^{-2t}\lambda ^{-t}L$ and lower estimability 
radius $K\lambda  r|H'(x)|^{-1}$. Thus, using (\ref{3042606}), we get
$$
\begin{aligned}
m_{1} \(B(x,\^\lambda (A^{-1}r)^{1/q})\) &\ge m_{1} \(B(x,K^{-1}\lambda  r|H'(x)|^{-1}))\\
& \ge K^{-2t}\lambda ^{-t}L(K\lambda  r|H'(x)|^{-1})^t \\
&\ge K^{-t}Lr^tA^{-t}|x-c|^{(1-q)t}\\
& \ge K^{-t}L(A^{-1}r)^t(A^{-1}r)^{{1-q\over q}t} \\
&=K^{-t}L((A^{-1}r)^{1/q})^t.
\end{aligned}
$$
In view of this and (\ref{1042606}) the proof is completed. \endpf

  By writing  $A \lek B$\index{\lek@$\lek$}  we mean that there exists a positive
  constant $C$  such that $  A \leq C B$ for all $A$ and $B$
  under consideration. Then $ A \gek B$\index{\gek@$\gek$} means that $B \lek A$, and
  $A  \comp B$\index{\comp@$\comp$}  says that $A\lek B$  and $B \lek A$.

\subsection{Open Set Condition and Essential Families}

\fr In this section, starting with
the open set condition,
we develop the machinery of essential families of inverse branches. We first prove 
the following two lemmas.

\

\blem
\label{l:osclower}
Let $G=\langle f_{1},\ldots ,f_{u}\rangle $ be a 
rational semigroup satisfying the nice open set condition 
with $U.$ Let $j\in \{ 1,\ldots ,u\} $ and 
let $c\in f_{j}^{-1}(\overline{U})$ be a critical point 
of $f_{j}.$ 
Then there exist constants $\zeta _{j,c}>0, \xi _{j,c}>0, $ and 
$T_{j,c}>0$ such that 
for each $x\in B_{s}(c,\zeta _{j,c})\cap f_{j}^{-1}(\overline{U})$ and 
for each $0<r<T_{j,c}$, 
$l_{2}(f_{j}^{-1}(U)\cap B_{s}(x,r))\geq \xi _{j,c}r^{2}.$ 
\elem
\begin{proof}
By conjugating $G$ by an element of Aut$(\oc)$, we may assume that 
$\infty \not\in  f_{j}^{-1}(\{ f_{j}(c)\} ).$ Let 
$W$ be an open neighborhood of $f_{j}(c)$ in $\C $ such that 
$f_{j}^{-1}(W)\subset \C .$   
Let $m_{1}:=l_{2,e}|_{f_{j}^{-1}(U\cap W)} $ 
and $m_{2}:= l_{2,e}|_{U\cap W}$, 
where $l_{2,e}$ denotes the Euclidian measure on $\C .$ 
Then $(m_{1},m_{2})$ is a $2$-conformal pair for 
$f_{j}.$ 
By the nice open set condition, 
there exist constants $C>0$ and $0<R<\infty $ such that 
for each $y\in \overline{U}\cap W$ and for each $0<r<R$, 
$m_{2}(B(y,r))\geq Cr^{2}.$ 
By using the method of the proof of Lemma~\ref{lku4.11}, 
it is easy to see that 
there exist constants $\zeta _{j,c}'>0$, $\xi _{j,c}>0$ and $T_{j,c}'>0$ such that 
for each $x\in B(c,\zeta _{j,c}')\cap f_{j}^{-1}(\overline{U})$ and for each 
$0<r<T_{j,c}'$, 
$m_{2}(B(x,r))\geq \xi _{j,c}'r^{2}.$ 
Thus, the statement of our lemma holds. We are done.     
\end{proof}
Combining Lemma~\ref{l:osclower} and Koebe's Distortion Theorem, 
we immediately obtain the following lemma.
\blem
\label{l:nosclc}
Let $G=\langle f_{1},\ldots ,f_{u}\rangle $ be a 
rational semigroup satisfying the nice open set condition with $U.$ 
Then, there exist constants $\xi >0$ and $T>0$ such that 
for each $j=1,\ldots ,u$ and for each $x\in f_{j}^{-1}(\overline{U})$, 
$l_{2}(f_{j}^{-1}(U)\cap B_{s}(x,r))\geq \xi r^{2}.$ 
\elem 

\

For every family $\Fa\sbt\Sg_u^*$ let
$$
\hat\Fa=\{{\hat\tau}:\tau\in\F\} \  \text{ and } \Fa_*=\{\tau_*:\tau\in\F\}.
$$

\bdfn\label{d1h2a.1} 
Let $G=\langle f_{1},\ldots, f_{u}\rangle $ be a rational semigroup 
satisfying the nice open set condition. Suppose that $J(G)\subset \C. $ 
Fix a number $M>0$, a number $a>0$, and $V$, an open subset of $\Sg_u$. Suppose $x\in J(G)$ and $r\in (0,1]$. 
A family $\Fa\sbt\Sg_u^*$ is called $(M,a,V)$-essential for the
pair $(x,r)$ provided that the following conditions are satisfied.
\begin{itemize}
\item[(ess0)] For every $\tau \in \Fa$, $f_{\tau }(x)\in J(G).$  
\item[(ess1)] For every $\tau\in \Fa$ there exists a number $R_\tau$ with $0<R_{\tau }<a$ and an $f_{\hat{\tau},x}^{-1}:
B(f_{\hat{\tau}}(x),2R_\tau)\to\C$, an analytic inverse branch of $f_{\hat{\tau}}^{-1}$ sending $f_{\hat{\tau}}(x)$
to $x$, such that
$$
M^{-1}R_\tau\le |f_{\hat{\tau}}'(x)|r\le {1\over 4}R_\tau.
$$
\item[(ess2)] The family $\Fa$ consists of mutually incomparable words.
\item[(ess3)] $\bu_{\tau\in\Fa}[\tau]=V$.
\end{itemize}
If $V=\Sg_u$, the family $\Fa$ is simply called $(M,a)$-essential for the pair $(x,r)$.
\edfn

\

%
%

\fr We shall prove the following.

\

\bprop\label{p1h2a.1}
Let $G=\langle f_{1},\ldots ,f_{u}\rangle $ be a rational semigroup satisfying the nice 
open set condition with $U.$ Suppose that $J(G)\subset \C .$ 
Then, for every number $M>0$ and for every $a>0$ there exists an integer $\#_{(M,a)}\ge 1$ with the following
properties. If $V$ is an open subset of $\Sg_u$, $x\in J(G)$, $r\in (0,1]$, and $\Fa\sbt\Sg_u^*$ 
is an $(M,a,V)$-essential family for $(x,r)$, then we have the following.
\begin{itemize}
\item[(a)] 
$$
B(x,r)
\sbt f_{{\hat\tau},x}^{-1}\(B(f_{\hat\tau}(x),R_\tau)\)
\sbt\bu_{\g\in\Fa}f_{\hat\g,x}^{-1}\(B(f_{\hat\g}(x),R_\g)\)
\sbt B(x,KMr)
$$
for all $\tau\in\Fa$.
\item[(b)] 
$$
J(\tf)\cap(V\times B(x,r))
\sbt \bu_{\tau\in \Fa}\^f_{{\hat\tau},x}^{-|{\hat\tau}|}\(p_2^{-1}(B(f_{\hat\tau}(x),R_\tau))\)
=    \bu_{\tau\in \Fa}[\tau]\times f_{{\hat\tau},x}^{-1}\(B(f_{\hat\tau}(x),R_\tau)\). 
$$
\item[(c)] $\#\Fa\le\#_{(M,a)}$.
\end{itemize}
\eprop
{\sl Proof.} Item (a) follows immediately from Theorem~\ref{kdt1/4} (${1\over 4}$-Koebe's 
Distortion Theorem), and Theorem~\ref{kdt1E}. The equality part in item (b) is obvious. In order to prove the
inclusion take $(\om,z)\in J(\tf)\cap (V\times B(x,r))$. By item (ess3) of 
Definition~\ref{d1h2a.1} there exists $\tau\in\Fa$ such that $\om\in[\tau]$.
But then, by the first in item (a), $(\om,z)\in 
[\tau]\times f_{{\hat\tau},x}^{-1}\(B(f_{\hat\tau}(x),R_\tau)\)$ and item (b) is entirely 
proved. Let us deal with item (c). 
By item (osc2) of Definition~\ref{d1h2intro}, 
$$ \{ f_{\hat{\tau},x}^{-1}((f_{\tau _{\ast }}|_{B(f_{\hat{\tau } }(x),R_{\tau })})^{-1}(U))\} _{\tau \in \mathcal{F}}$$ 
is a family of mutually disjoint sets.
Hence, using also (a), we get
\beq\label{1h2c.1}
\Sg _{\tau \in \mathcal{F}}l_{2}
(f_{\hat{\tau},x}^{-1}((f_{\tau _{\ast }}|_{B(f_{\hat{\tau }}(x),R_{\tau })})^{-1}(U)))
\le l_2(B(x,KMr))
=C\pi(KM)^2r^2, 
\eeq
where $C>0$ is a constant independent of $r,M$, and $a.$ 
Let $L_{a}:= \xi \min \{ (T/a)^{2},1\} $, 
where $\xi $ and $T$ come from Lemma~\ref{l:nosclc}. 
By Lemma~\ref{l:nosclc}, we obtain that 
for each $j=1,\ldots ,u$ , for each $y\in \overline{f_{j}^{-1}(U)}$, and 
for each $0<b\leq a$,  
\beq
\label{eq:l2byb}
l_{2}(B(y,b)\cap f_{j}^{-1}(U))\geq L_{a}b^{2}. 
\eeq
It follows from Theorem~\ref{kdt1E}, (\ref{eq:l2byb}), and (ess1) that 
for all $\tau \in \mathcal{F}$, 
we have
$$
\aligned
l_{2}\left(f_{\hat{\tau},x}^{-1}((f_{\tau _{\ast }}|_{B(f_{\hat{\tau }}(x),R_{\tau })})^{-1}
(U))\right) 
& \geq K^{-2}|f_{\hat{\tau }}'(x)|^{-2}
l_{2}((f_{\tau _{\ast }}|_{B(f_{\hat{\tau }}(x),R_{\tau })})^{-1}(U)) \\ 
& \geq K^{-2}|f_{\tau }'(x)|^{-2}l_{2}(B(f_{\hat{\tau }}(x),R_{\tau })\cap f_{\tau _{\ast }}^{-1}(U))\\ 
& \geq K^{-2}|f_{\hat{\tau }}(x)|^{-2}L_{a}R_{\tau }^{2}\\ 
& \geq K^{-2}16 L_{a}r^{2}.
\endaligned 
$$
Combining this with (\ref{1h2c.1}) we get that $\#\Fa\le (16L_{a})^{-1}CK^{4}\pi M^2$. 
We are done. 
\endpf

\section{Basic Properties of semi-hyperbolic Rational Semigroups}

\fr In this section we define semi-hyperbolic rational semigroups and 
collect their dynamical properties, with proofs, which will be needed in the sequel.
\bdfn\label{d1h3}
A rational semigroup $G$ is called semi-hyperbolic if and only if 
there exists an $N\in \N $ and a $\delta >0$ such that 
for each $x\in J(G)$ and $g\in G$, 
$$\deg (g:V\rightarrow B_{s}(x,\delta ))\leq N$$
for each connected component $V$ of $g^{-1}(B_{s}(x,\delta )).$ 
\edfn

\fr The crucial tool, which makes all further considerations possible, is given by the
following semigroup version of Mane's Theorem proved in \cite{hiroki2}.

\bthm\label{t2.9h3}
Let $G=\langle f_{1},\ldots ,f_{u}\rangle $ be a finitely generated 
rational semigroup. Assume that there exists an element of $G$ with 
degree at least two, that each element of {\em Aut}$(\oc )\cap G$ (if this is not empty) is 
loxodromic, and that $F(G)\neq \emptyset .$ 
Then, $G$ is semi-hyperbolic if and only if all of the following conditions are satisfied.  

\begin{itemize}
\item[(a)] For each $z\in J(G)$ there exists a neighborhood $U$ of $z$ in $\oc $ 
such that for any sequence $\{ g_{n}\} _{n=1}^{\infty }$ in $G$, 
any domain $V$ in $\oc $ and any point 
$\zeta \in U$, the sequence $\{ g_{n}\} _{n=1}^{\infty }$ does not converge 
to $\zeta $ locally uniformly on $V.$ 
\item[(b)] For each $j=1,\ldots ,u$, each $c\in \Crit(f_{j})\cap J(G)$ 
satisfies $\dist(c, G^{\ast }(f_{j}(c)))>0.$ 
\end{itemize}
\ethm

%
%
%
%
\ni The first author proved in \cite{hiroki2} the following.

\bthm\label{t1h4}
Let $G=\langle f_{1},\ldots ,f_{u}\rangle $ be a semi-hyperbolic finitely generated 
rational semigroup. Assume that there exists an element of $G$ with  
degree at least two, that each element of {\em Aut}$(\oc )\cap G$ (if
this is not empty) is  
loxodromic, and that $F(G)\neq \emptyset .$ 
Then there exist 
$R>0$, $C>0$, and $\a>0$ such that if $x\in J(G)$, $\om \in \Sg
_{u}^{\ast }$ and $V$ is a connected component  
of $f_{\om }^{-1}(B_{s}(x,R))$, then $V$ is simply connected and
$\diam_{s}(V)\le Ce^{-\a|\om |}$. 
\ethm

\

Throughout the rest of the paper, we assume the following: 

\noindent {\bf Assumption ($\ast $):} 
\begin{itemize}
\item 
Let $f=(f_{1},\ldots ,f_{u})\in (\mbox{Rat})^{u}$ be a multi-map and 
let $G=\langle f_{1},\ldots ,f_{u}\rangle .$ 
\item There exists an element $g$ of $G$ such that $\deg (g)\geq 2.$
\item Each element of Aut$(\oc )\cap G$ (if this is not empty) is loxodromic.  
\item 
$G$ is semi-hyperbolic.
\item 
$G$ satisfies the nice open set condition. 
\end{itemize}

In order to prove the main results (Theorem~\ref{Theorem A} etc.), in
virtue of \cite{ncp1} and \cite{ncp2},
we may assume that $u\geq 2.$ If $u\geq 2$, then the open set condition 
implies that $F(G)\neq \emptyset .$ 
Hence, conjugating $G$ by some element of Aut$(\oc )$ if necessary, we may assume that  
$J(G)\subset \C .$ 
Thus, throughout the rest of the paper, in addition to the above assumption, we also assume that 
\begin{itemize}
\item $u\geq 2$ and $J(G)\subset \C .$ 
\end{itemize} 
Note that in Theorem~\ref{Theorem A}, we work with the spherical distance. However, 
throughout the rest of the paper, we will work with the Euclidian distance. If 
we want to get the results on the spherical distance (and this would include the case $u=1$), 
then we have only to consider some minor modifications in our argument.  

We now give further notation. 
A pair $(c,j)\in\oc\times\{1,2,\ld,u\}$ is called critical if $f_j'(c)=0$. The
set of all critical pairs of $f$ will be denoted by $\CP(f)$. 
Let $\Crit(f)$ be the union of $\bigcup _{j=1}^{u}\Crit (f_{j}).$ 
For every $c\in\Crit(f)$
put 
$$
c_+=\{f_j(c):(c,j)\in\CP(f)\}.
$$
The set $c_+$ is called the set of critical values of $c$. 
For any subset $A$ of $\Crit(f)$ put
$$
A_+=\{c_+:c\in A\}.
$$
For each $(c,j)\in \CP(f)$ let $q(c,j)$ be the local order of $f_{j}$ at $c.$ 
For any set $F\sbt\oc$, set
$$
\om_G(F)=\bi_{n=0}^\infty\ov{\bigcup _{|\om |\geq n}f_{\om}(F)}.
$$
The latter is called the $\om$-limit set of $F$ with respect to the semigroup $G$. 
Similarly, for every set $B\sbt\Sg_u\times\oc$,
\beq\label{referee1}
\om(B)=\bi_{N=0}^\infty\ov{\bigcup _{n\geq N}\tilde{f}^n(B)},
\eeq
and this set is called the $\om$-limit set of $F$ with respect to the skew product map
$\tf:\Sg_u\times\oc\to\Sg_u\times\oc$.

\

\

\ni Given $\om\in\Sg_u^*$, $j\in\{1,2,\ld,u\}$, $z\in \tf_\om^{-1}(J(\tf))$ and $r>0$, we say
that a critical pair $(c,j)$ sticks to $\Comp(z,f_\om,r)$ if $c\in\Comp(z,f_\om,r)$ and
$j=\om_1$. We then write
$$
(c,j)\rel\Comp(z,f_\om,r).
$$ 
Set
$$
A=A_f:=\max\{A(f_j,c):(c,j)\in\CP(f)\} \  \text{ and }  \  R_f:=\min\{R(f_j,c):(c,j)\in\CP(f)\} .
$$
\ni For $A$, $B$, any two subsets of a metric space put
$$
\dist(A,B)=\inf\{\dist(a,b): a\in A, b\in B\}
$$\index{dist(A,B)@$\dist(A,B)$}
and
$$
\Dist(A,B)=\sup\{\dist(a,b): a\in A, b\in B\}.
$$
Fix a positive $\b$ smaller than the following four positive numbers (a)--(d).
$$
\mbox{(a)}\ \min\{\dist(c,G^{\ast }(c_+)):c\in\Crit(f)\cap J(G)\}, \ \mbox{(b)}\ R_f, \ 
\mbox{(c)}\ \min\{|c'-c|:c,c'\in\Crit(f)\cap J(G),c\ne c'\},$$ 
and $$\mbox{(d)}\ \dist(\bigcup _{j=1}^{u}CV(f_{j})\cap F(G),\ J(G)), $$
where, (a) is positive because of semi-hyperbolicity (Theorem~\ref{t2.9h3}). It immediately follows 
from Theorem~\ref{t1h4} that there exists $\g\in(0,1/4)$ so
small that if $g\in G^{\ast }$ and $g(x)\in J(G)$, then  
\beq\label{2.25ku1h7}
\text{Comp}(x,g,2\gamma )\subset \C \ \ \mbox{\ and \ } \diam(\Comp(x,g,2\g))<\b.
\eeq
We shall prove the following.

\

\blem\label{l2.16ku1h7}  
Fix $\eta\in (0, \beta )$, an integer $n\ge 0$ and $(\om,z)\in J(\tf)$. Suppose that for every 
$0\le k\le n-1$,
$$
\diam\(\Comp\(f_{\om|_k}(z),f_{\om|_{k+1}^n},\eta\)\)\le\b.
$$
Then each connected component $\Comp\(f_{\om|_k}(z),f_{\om|_{k+1}^n},\eta\)$ is 
stuck to by at most one critical pair $(c,j)$ of $f$; and if a critical pair $(c,j)$ sticks to 
a component $\Comp\( f_{\om |_{k}}(z), f_{\om |^{n}_{k+1}}, \eta \)$, then $f_{j}(c)\in J(G).$  
Furthermore, each critical pair of $f$ sticks to at most one of all these components 
$\Comp\( f_{\om |_{k}}(z), f_{\om |^{n}_{k+1}}, \eta )$. 
\elem
{\sl Proof.} The first part is obvious by the choice of $\b$. In order to prove 
the second part suppose that
$$
(c,\om_{k+1})\rel \Comp\(f_{\om|_k}(z),f_{\om|_{k+1}^n},\eta\) 
\  \text{ and } \ 
(c,\om_{l+1})\rel \Comp\(f_{\om|_l}(z),f_{\om|_{l+1}^n},\eta\)
$$
with some $0\le k<l\le n-1$ and $\om_{k+1}=\om_{l+1}$. 
Then both $c$ and $f_{\om|_{k+1}^{l}}(c)$
belong to $\Comp\(f_{\om|_l}(z),f_{\om|_{l+1}^n},\eta\)$, and therefore, 
$$
|f_{\om|_{k+1}^l}(c)-c|\le \diam\(\Comp\(f_{\om|_l}(z),f_{\om|_{l+1}^n},\eta\)\)\le\b,
$$
contrary to the choice of $\b$. \endpf

\

\ni Let 
$$
\ka=\Pi_{(c,j)\in\CP(f)}q(c,j)^{-1}.
$$
\blem\label{lku12.17h7}
If $g\in G$ and $z\in g^{-1}(J(G))$, then
$$
\Mod\(\Comp(z,g,2\g)\sms \overline{\Comp(z,g,\g)}\)\ge {\ka\log 2\over \#\CP(f)}.
$$
\elem

{\sl Proof.} By Lemma~\ref{l2.16ku1h7} there exists a geometric annulus $R\sbt
B(g(z),2\g)\sms B(g(z),\g)$ centered at $g(z)$ and with modulus $\ge \log2/\#\CP(f)$
and such that $g^{-1}(R)\cap\Comp(z,g,2\g)\cap\Crit(f)=\es$. Since covering maps
increase moduli of annuli by factors at most equal to their degrees, we conclude that
$$
\Mod\(\Comp(z,g,2\g)\sms \overline{\Comp(z,g,\g)}\)
\ge \Mod(R_g)
\ge {\log2/\#\CP(f)\over \Pi_{(c,j)\in\CP(f)}q(c,j)}
={\ka\log 2\over \#\CP(f)},
$$
where $R_g\sbt \Comp(z,g,2\g)$ is the connected component of $R$ enclosing 
$\Comp(z,g,\g)$. \endpf

\

\ni As an immediate consequence of this lemma and Theorem~\ref{kdt2E} we get the following.

\

\blem\label{lku22.18h9}
Let $\om \in \Sg _{u}^{\ast }$ and 
suppose $f_\om(z)\in J(G)$. If $0\le k\le|\om|$ and the map $f_{\om|_k}:\Comp(z,f_\om,2\g)
\to \Comp(f_{\om|_k}(z),f_{\om|_{k+1}^{|\om|}},2\g)$ is univalent, then
$$
{|f_{\om|_k}'(y)|\over |f_{\om|_k}'(x)|}\le \const
$$
for all $x,y\in \Comp(z,f_\om,\g)$, where const is a number depending only on $\#\CP(f)$
and $\ka$. 
\elem

\

\blem\label{lku2.19h9}
Suppose that $g\in G$ and $g(z)\in J(\tf)$. Suppose also that 
$Q^{(1)}\sbt Q^{(2)}\sbt
B(g(z),\g)$ are connected sets. If $Q_g^{(2)}$ is a connected
component of $g^{-1}(Q^{(2)})$ contained in $\Comp(z,g,\g)$
and $Q_g^{(1)}$ is a connected component of $g^{-1}(Q^{(1)})$
contained in $Q_g^{(2)}$, then
$$
{\diam\(Q_g^{(1)}\)\over \diam\(Q_g^{(2)}\)} 
\ge \Ga{\diam\(Q^{(1)}\)\over \diam\(Q^{(2)}\)}
$$
with some universal constant $\Ga>0$ depending on $f$ only. 
\elem
{\sl Proof.} Write $g=f_\om$, where $\om\in\Sg_u^*$ and put $n=|\om|$. For every
$0\le j\le n$, set
$$
Q_j^{(1)}=f_{\om|_{n-j}}(Q_g^{(1)}) \  
\text{ and } \
Q_j^{(2)}=f_{\om|_{n-j}}(Q_g^{(2)}).
$$
Let $1\leq n_1\leq \ldots \leq n_v\leq n$  be all
the integers $k$ between $1$ and $n$  such that
$$ 
\Crit(f_{\om_{n-k+1}}) \cap \Comp(f_{\om|_{n-k}}(z),f_{\om_{n-k+1}^n},2\g)\neq \es.
$$
Fix $1\leq i \leq v $. If $ j \in [ n_i, n_{i+1}-1]$ (we set
$n_{v+1}=n-1$), then  by Lemma~\ref{lku22.18h9} there exists a
universal constant $T>0$ such that
\begin{equation}\label{1020907}
\frac{\diam(Q_{j}^{(1)})}{\diam(Q_{j}^{(2)})} 
\geq T \frac{\diam(Q_{n_i}^{(1)})}{\diam(Q_{n_i}^{(2)})}.
\end{equation}
Since, in view of Lemma~\ref{l2.16ku1h7}, $v \leq \#\CP(f)$, in order   
to conclude  the proof  it is enough  to show
the existence  of a universal   constant $E >0$ such that  for every
$ 1  \leq i \leq u$
$$\frac{\diam(Q^{(1)}_{n_i})}{\diam(Q^{(2)}_{n_i})}
\geq E \frac{\diam(Q^{(1)}_{n_i-1})}{\diam(Q^{(2)}_{n_i-1})}.
$$
Indeed, let $c$   be the  critical  point   in $\Comp(f_{\om|_{n-n_i}}(z),
f_{\om|_{n_i+1}^n}, 2\g)$ and let $q$  be its  order. Since   both
sets $Q^{(1)}_{n_i}$ and  $Q^{(1)}_{n_i}$  are connected, we get for
$j=1,2$ that
$$\begin{aligned}
\diam(Q_{n_i-1}^{(j)})
&\comp \diam(Q_{n_i}^{(j) })\sup\{|f_{\om_{n-n_i+1}}'(x)|:x\in Q_{n_i}^{(j)}\} \\
&\comp \diam(Q_{n_i}^{(j)})\Dist(c, Q_{n_i}^{(i)}).
\end{aligned}
$$
Hence
$$\begin{aligned}
\frac{\diam(Q_{n_i}^{(1)})}{\diam(Q_{n_i}^{(2)})}& \comp
\frac{\diam(Q_{n_i-1}^{(1)})}{\Dist(c, Q_{n_i}^{(1)})} \cdot
\frac{\Dist(c, Q_{n_i}^{(2)})}{\diam(Q_{n_i-1}^{(2)})} \\
& \ge \frac{\diam(Q_{n_i-1}^{(1)})}{ \diam(Q_{n_i-1}^{(2)})}.
\end{aligned}
$$
The proof is finished. \endpf

\section{Partial Order in $\Crit(f)\cap J(G)$ and Stratification of $J(G)$}
\ni In this section we introduce a partial order in the critical set $\Crit(f)\cap J(G)$ 
and stratification of $J(G)$. They will be used to do the inductive steps in
the proofs of the main theorems of our paper. We start with the following.

\

\blem\label{lku22.20h11}
The set $\om_G((\Crit(f)\cap J(G))_+)$ is nowhere dense in $J(G)$.
\elem

{\sl Proof.} Suppose on the contrary that the interior (relative to $J(G)$) of 
$\om_G((\Crit(f)\cap J(G))_+)$ is not empty. Then, there exists a
critical point $c\in \Crit(f)\cap J(G)$ 
such that $\om_G(c_+)$ has non-empty interior. But then, in virtue of
Proposition~\ref{p1h10} 
there would exist finitely many elements $g_1,g_2,\ld,g_n\in G$ such that 
$\om_G(c_+)\spt g_1(\om_G(c_+))\cup g_2(\om_G(c_+))\cup\ld\cup
g_n(\om_G(c_+))\spt J(G)$. Hence 
$c\in \om_G(c_+)$, contrary to the non-recurrence condition (Theorem~\ref{t2.9h3}). \endpf

\

\ni Now we introduce in $\Crit(f)\cap J(G)$ a relation $<$ which, in
view of Lemma~\ref{lku22.21h11} below, is an ordering relation. Put
$$
c_1<c_2 \eqv c_1\in\om_G({c_2}_+).
$$
\blem\label{lku22.21h11}
If $c_1<c_2$ and $c_2<c_3$, then $c_1<c_3$.
\elem

{\sl Proof.} Since $c_2\in\om_G({c_3}_+)$, we have
$\om_G({c_2}_+)\sbt\om_G({c_3}_+)$. Along with 
$c_1\in\om_G({c_2}_+)$ this implies that $c_1\in\om_G({c_3}_+)$,
meaning that $c_1<c_3$. \endpf 

\

\blem\label{l1h10}
There exists no $c\in \Crit(f)\cap J(G)$ such that $c<c$.
\elem 

{\sl Proof.} Indeed, $c<c$ means that $c\in\om_G(c_+)$, contrary to the 
non-recurrence condition. \endpf

\

\ni Since the set $\Crit(f)\cap J(G)$ is finite, as an immediate
consequence of this lemma and Lemma~\ref{lku22.21h11} we get the following.

\

\blem\label{lku22.22h11}
There is no infinite linear subset of the partially ordered set
$(\Crit(f)\cap \nolinebreak J(G),<\nolinebreak ).$ 
\elem

\

\ni Now define inductively a sequence $\{Cr_i(f)\}$ of subsets of $\Crit(f)\cap J(G)$ 
by setting $Cr_0(f)=\es$ and
\begin{equation}\label{ku22.29}
Cr_{i+1}(f)=
\left\{c\in (\Crit(f)\cap J(G))\sms\bu_{j=0}^iCr_j(f):c'<c, \imp c'\in
  \bu_{j=0}^iCr_j(f)\right\}. 
\end{equation}

\

\blem\lab{lu5.5h13} The following four statements hold.
\begin{itemize}
\item[(a)] The sets $\{Cr_i(f)\}$ are mutually disjoint.
\item[(b)] $\exists_{p\ge 1} \  \forall_{i\ge p+1} \  \  Cr_i(f)=\es$.
\item[(c)] $Cr_0(f)\cup\ld\cup Cr_p(f)=\Crit(f)\cap J(G)$.
\item[(d)] $Cr_1(f)\ne\es$.
\end{itemize}
\elem

\fr{\sl Proof.} By definition $Cr_{i+1}(f)\cap\bu_{j=0}^iCr_j(f)=\es$, whence disjointness 
in (a) is clear. As the set $\Crit(f)\cap J(G)$ is finite, (b) follows from (a). Take $p$
to be the minimal number satisfying (b) and suppose that
$(\Crit(f)\cap J(G))\sms\bu_{j=0}^pCr_j(f) 
\ne\es$. Take $c\in (\Crit(f)\cap J(G))\sms\bu_{j=0}^pCr_j(f)$. Since
$Cr_{p+1}=\es$, there would 
thus exist $c'\in (\Crit(f)\cap J(G))\sms\bu_{j=0}^pCr_j(f)$ such that
$c'<c$. Iterating this  
procedure we would obtain an infinite sequence $c_1=c>c'=c_2>c_3>\ld$, contrary to 
Lemma~\ref{lku22.22h11}. Now, part (d) follows from (c) and (\ref{ku22.29}). \endpf

\

\ni For every $(\tau,z)\in J(\tf)$ put
$$
\Crit(\tau,z)=\Crit(\tf)\cap\om(\tau,z) 
\  \text{ and } \
\Crit(\tau,z)_+=p_2(\Crit(\tf)\cap\om(\tau,z))_+.
$$
\blem\label{lu5.7h13}
If $(\tau,z)\in J(\tf)$, then 
$
p_2(\om(\tau,z))\not\sbt\ov{G^{\ast }(\Crit(\tau,z)_+)}.
$
\elem

{\sl Proof.} Suppose on the contrary that 
\beq\label{u2.30h14}
p_2(\om(\tau,z))\sbt\ov{G^{\ast }(\Crit(\tau,z)_+)}.
\eeq
Consequently, $\Crit(\tau,z)\ne\es$. Let $(\tau^1,c_1)\in\Crit(\tau,z)$. This
means that $(\tau^1,c_1)\in \om(\tau,z)$, and it follows from (\ref{u2.30h14})
that there exists $(\tau^2,c_2)\in\Crit(\tau,z)$ such that either $c_1\in\om_G({c_2}_+)$
or $c_1=g_1(c_2)$ for some $g_1\in G$ of the form $f_\om$ with $f_{\om_1}'(c_2)=0$.
Iterating this procedure we obtain an infinite sequence $((\tau^j,c_j))_{j=1}^\infty$
of points in $\Crit(\tau,z)$ such that for every $j\ge 1$ either $c_j\in\om_G({c_{j+1}}_+)$
or $c_j=g_j(c_{j+1})$ for some $g_j\in G$ of the form $f_\rho$ with $f_{\rho_1}'(c_{j+1})=0$.
Consider an arbitrary block $c_k,c_{k+1},\ld,c_l$ such that $c_j=g_j(c_{j+1})$ for 
every $k\le j\le l-1$, and suppose that $l-(k-1)\ge \#(\Crit(f)\cap J(G))$. Then there are 
two indices $k\le a<b\le l$ such that $c_a=c_b$. Hence $g_a\circ g_{a+1}\circ\ld\circ
g_{b-1}(c_b)=c_a=c_b$ and $(g_a\circ g_{a+1}\circ\ld\circ g_{b-1})'(c_b)=0$.
This however contradicts our assumption that the Julia set of $G$ contains no superstable
fixed points. In consequence, the length of the block $c_k,c_{k+1},\ld,c_l$ is
bounded above by $\#(\Crit(f)\cap J(G))$. Therefore, there exists an infinite sequence 
$(j_n)_{n=1}^\infty$ such that $c_{j_n}\in\om_G({c_{j_n+1}}_+)$ for all $n\ge 1$. 
This however
contradicts Lemma~\ref{lku22.22h11} and finishes the proof. \endpf

\

\ni Now, for every $i=0,1,\ld,p$, set
$$
S_i(f)=Cr_0(f)\cup Cr_1(f)\cup\ld\cup Cr_i(f).
$$
Fix $i\in\{0,1,\ld,p-1\}$ consider an arbitrary point $c'\in\bu_{c\in Cr_{i+1}(f)}\om_G(c_+)\cap 
\Crit(f)\cap J(G)$. Then there exists $c\in Cr_{i+1}(f)$ such that $c'\in\om_G(c_+)$ which 
equivalently means that $c'<c$. Thus, by (\ref{ku22.29}) we get $c'\in S_i(f)$. So, 
\beq\label{u5.8h15}
\bu_{c\in Cr_{i+1}(f)}\om_G(c_+)\cap \( (\Crit(f)\cap J(G))\sms S_i(f)\)=\es.
\eeq
Since the set $\bu_{c\in Cr_{i+1}(f)}\om_G(c_+)$ is compact and 
$(\Crit(f)\cap J(G))\sms S_i(f)$ is finite, we therefore get
\beq\label{u5.9h15}
\d_i=\dist\lt(\bu_{c\in Cr_{i+1}(f)}\om_G(c_+), (\Crit(f)\cap J(G))\sms S_i(f)\rt)>0.
\eeq
Set
$$
\rho=\min\{\d_i/2:i=0,1,\ld,p-1\}
$$
and for every $i=0,1,\ld,p$ define
\beq\label{1h17}
J_i(G)=\{z\in J(G):\dist\(G^{\ast }(z), (\Crit(f)\cap J(G))\sms S_i(f)\)\ge\rho\}.
\eeq
We end this section with the following two lemmas concerning the sets $J_i(G)$.

\

\blem\label{lu5.8h17}
$\es\ne J_0(G)\sbt J_1(G)\sbt\ld\sbt J_p(G)=J(G)$.
\elem

{\sl Proof.} Since $S_{i+1}(f)\spt S_i(f)$, the inclusions $J_i(G)\sbt 
J_{i+1}(G)$ are obvious. Since $S_p(f)=\Crit(f)\cap J(G)$ (see Lemma~\ref{lu5.5h13}), it 
holds $J_p(G)=J(G)$. 
We get from (\ref{u5.9h15}) that
$$
\dist\lt(\bu_{c\in Cr_1(f)}\om_G(c_+), (\Crit(f)\cap J(G))\sms S_0(f)\rt)=\d_0\ge 2\rho\ge\rho.
$$
Thus, 
$\bu_{c\in Cr_1(f)}\om_G(c_+)\sbt J_0(G)$, and since $Cr_1(f)\ne\es$ (see 
Lemma~\ref{lu5.5h13}), we conclude that $J_0(G)\ne\es$. The proof is complete. \endpf

\

\blem\label{lu5.9h17}
There exists $l=l(f)\ge 0$ so large that for all $i=0,1,\ld,p-1$ we have
$$
\bu_{c\in Cr_{i+1}(f)}\om_G(c_+)\cap J(G)\sbt\ov{\bu_{|\tau|\ge
    l}f_{\tau }(Cr_{i+1}(f)_+)} \cap J(G) 
=\ov{G^{\ast }\(\bu_{|\tau|=l}f_{\tau }(Cr_{i+1}(f)_+)\)} \cap J(G)
\sbt J_i(G).
$$
\elem
{\bf Proof.} The left-hand inclusion is obvious regardless of what $l(f)$ is. 
The equality part of the assertion is obvious. In order to prove the right-hand 
inclusion fix $i\in\{0,1,\ld,p-1\}$. By the definition of the $\om$-limit sets of 
$G$ there exists $l_i\ge 0$ such that for every $c\in Cr_{i+1}(f)$ we have 
$\dist\(\bu_{|\tau |\ge l_i}f_{\tau }(c_+),\bu_{c'\in Cr_{i+1}(f)}\om_G(c_+')\)<\d_i/2$. Thus, 
by (\ref{u5.9h15}), $\dist\(\bu_{|\tau |\ge l_i}f_{\tau }(c_+),
(\Crit(f)\cap J(G))\sms S_i(f)\)>\d_i/2 
\ge \rho $. Hence, for every $\tau \in \Sg _{u }^{\ast }$ with $|\tau |\ge l_i$, we have 
$\dist\(f_{\tau }(c_+),
(\Crit(f)\cap J(G))\sms S_i(f)\)>\rho$. Thus $f_{\tau }(c_+)\sbt J_i(G)$. Therefore, 
$\bu_{|\tau |\ge l_i}f_{\tau }(c_+)\sbt J_i(G)$, and consequently, 
$\bu_{|\tau |\ge l_i}f_{\tau }(Cr_{i+1}(f)_+)\sbt J_i(G)$. Since $J_i(G)$ is a closed set, 
this yields that $\ov{\bu_{|\tau |\ge l_i}f_{\tau }(Cr_{i+1}(f)_+)}\sbt J_i(G)$. Setting
$l(f)=\max\{l_i:i=0,1,\ld,p-1\}$ completes the proof. \endpf

\section{Holomorphic Inverse Branches.}  

\ni In this section we prove the existence of suitable holomorphic inverse branches, 
our basic tools throughout the paper. Set 
$$
\Sing(\tf)=\bu_{n\ge 0}\tf^{-n}(\Crit(\tf))
$$
and
$$
\Sing(f)=\bu_{g\in G^{\ast }}g^{-1}(\Crit(f)).
$$
Recall that according to the formula \eqref{referee1}, given a point
$(\tau,z)\in\Sg_u\times\oc$, the set $\om(\tau,z)$ is the $\om$-limit set of
$(\tau,z)$ with respect to the skew product map 
$\tf:\Sg_u\times\oc\to\Sg_u\times\oc$.

\

\bprop\label{pu6.1h23}
For each $(\tau,z)\in J(\tf)\sms \Sing(\tf)$, there exists a number
$\eta(\tau,z)$ with $0<\eta (\tau ,z)<\gamma $, an increasing 
sequence $(n_j)_{j=1}^\infty$ of positive integers and a point $(\hat\tau,\hat z)\in
\om(\tau,z)\sms p_2^{-1}\(\ov{G^{\ast }(\Crit(\tau,z)_+)}\)$ with the
following two properties. 
\begin{itemize}
\item[(a)] $\lim_{j\to\infty}\tf^{n_j}(\tau,z)=(\hat\tau,\hat z)$,
\item[(b)] $\Comp\(z,f_{\tau|_{n_j}},\eta(\tau,z)\)\cap \Crit(f_{\tau|_{n_j}})=\es$. 
\end{itemize}
\eprop

{\sl Proof.} In view of Lemma~\ref{lu5.7h13} there exists a point $(\hat\tau,\hat z)\in
\om(\tau,z)$ such that $\hat z\notin\ov{G^{\ast }(\Crit(\tau,z)_+)}$. Let
$$
\eta={1\over 2}\dist\(\hat z, \ov{G^{\ast }(\Crit(\tau,z)_+)} \).
$$
Then there exists an infinite increasing sequence $(n_j)_{j=1}^\infty$ of positive integers 
such that
\beq\label{u2.39h23}
\lim_{j\to\infty}\tf^{n_j}(\tau,z)=(\hat\tau,\hat z)
\eeq
and 
\beq\label{u2.40h23}
f_{\tau|_{n_j}}(z)\notin B\( \ov{G^{\ast }( \Crit (\tau,z)_+)},\eta \)
\eeq
for all $j\ge 1$. We claim that there exists $\eta(\tau,z)>0$ such that for all
$j\ge 1$ large enough 
$$
\Comp\(z,f_{\tau|_{n_j}},\eta(\tau,z)\)\cap \Crit(f_{\tau|_{n_j}})=\es.
$$
Indeed, otherwise we find an increasing subsequence $(j_i)_{i=1}^\infty$ and a 
decreasing to zero sequence of positive numbers $\eta_i<\eta$ such that
$$
\Comp\(z,f_{\tau|_{n_{j_i}}},\eta_i\)\cap \Crit(f_{\tau|_{n_{j_i}}})\ne\es.
$$
Let $\^c_i\in\Comp\(z,f_{\tau|_{n_{j_i}}},\eta_i\)\cap \Crit(f_{\tau|_{n_{j_i}}})$. Then
there exist $0\le p_i\le n_{j_i}-1$ and 
\beq\label{1h23}
c_i\in\Crit(f_{\tau_{p_{i}+1}})
\eeq
such that $c_i=f_{\tau |_{p_i}}(\^c_i)$. Since $\lim_{i\to\infty}\eta_i=0$, it follows 
from Theorem~\ref{t1h4} that $\lim_{i\to\infty}\^c_i=z$. Since $(\tau,z)
\notin \bu_{n\ge 0}\tf^{-n}(\Crit(\tf))$, this implies that $\lim_{i\to\infty}p_i=+\infty$.
But then, making use of Theorem~\ref{t1h4} again and of the formula 
$(\sg^{p_i}(\tau),c_i)=\tilde{f}^{p_i}(\tau,\^c_i)$, we conclude that the set of accumulation
points of the sequence $((\sg^{p_i}(\tau),c_i))_1^\infty$ is contained in $\om(\tau,z)$.
Fix $(\tau^\infty,c)$ to be one of these accumulation points. Since $\Crit(\tf)$ is closed
we conclude that 
\beq\label{1h24}
(\tau^\infty,c)\in \Crit(\tau,z).
\eeq
Since that set $\Crit(f)$ is finite, passing to a subsequence, we may assume without
loss of generality that $(c_i)_1^\infty$ is a constant sequence, so equal to $c$. 
Since $c=f_{\tau |_{p_i}}(\^c_i)$, we get
$$
\lt|f_{\tau|_{n_{j_i}}}(z)-f_{\tau|_{p_i+1}^{n_{j_i}}}(c)\rt|
=\lt|f_{\tau|_{n_{j_i}}}(z)-f_{\tau|_{n_{j_i}}}(\^c_i)\rt|
<\eta_i
<\eta.
$$
But, looking at (\ref{1h23}) and (\ref{1h24}), we conclude that $f_{\tau|_{p_i+1}^{n_{j_i}}}(c)
\in G^{\ast }(\Crit(\tau,z)_+)$. We thus arrived at a contradiction with (\ref{u2.40h23}),
and the proof is finished. \endpf

\

\bcor\label{cu6.3h25}
If $(\tau,z)\in J(\tf)\sms \Sing(\tf)$, then $\limsup_{n\to\infty}|f_{\tau|_n}'(z)|=+\infty$.
\ecor

{\sl Proof.} Let $(n_j)_{j=1}^\infty$ and $\eta(\tau,z)$ be produced by 
Proposition~\ref{pu6.1h23}. Then, by this proposition and Theorem~\ref{t1h4}, the family 
$\lt\{f_{\tau|_{n_{j}},z}^{-1}:B\(f_{\tau|_{n_{j}}}(z),\eta(\tau,z)\)\to\C\rt\}_{j=1}^\infty$
of holomorphic inverse branches of $f_{\tau|_{n_{j}}}$ sending $f_{\tau|_{n_{j}}}(z)$
to $z$ is well defined and normal. As a matter of fact we mean here this family
restricted to the disk $B(\hat z,\eta(\tau,z)/2)$ and $j\ge 1$ large enough.
Therefore by Theorem~\ref{t1h4} again,
$\lim_{j\to\infty}|f_{\tau|_{n_j}}'(z)|^{-1}=0$ and we are
done. \endpf 

\

\ni We end this section with the following.
Let $\| \tf'\| =\sup _{w\in J(\tf)} |\tf '(w)|.$ 
\

\bprop\label{pu6.4h27}
Fix $\th\in(0,\min\{1,\g\})$. For all $(\tau,z)\in J(\tf)$ and  $r>0$ 
there exists a minimal integer $s=s(\th,(\tau,z),r)\ge 0$ with the following 
properties (a) and (b).
\begin{itemize}
\item[(a)] $|(\tf^s)'(\tau,z)|\ne 0$.
\item[(b)] Either $r|(\tf^s)'(\tau,z)|>||\tf'||^{-1}$ or there exists 
$c\in\Crit(f_{\tau_{s+1}})$ such that 
$f_{\tau _{s+1}}(c)\in J(G)$ 
and
$$
|f_{\tau|_s}(z)-c|\le \th r|f_{\tau|_s}'(z)|.
$$
In addition, for this $s$, we have  
\item[(c)] 
$\theta r|f_{\tau |_{s}}'(z)|\leq \theta <\gamma $ and 
$$
\Comp\(z, f_{\tau|_s},(KA_{f}^2)^{-1}2^{-\#\Crit(f)}\th r|f_{\tau|_s}'(z)|\)
     \cap\Crit(f_{\tau|_s})=\es.
$$
\end{itemize}
\eprop

{\sl Proof.} First note that the set of integers $(\ge 0)$ 
satisfying conditions (a) and (b) is not empty. Indeed, if
$(\tau,z)\in J(\tf)\sms \Sing(\tf)$, then this follows directly from
Corollary~\ref{cu6.3h25}. If, on the other hand, $(\tau,z)\in
\Sing(\tf)$, then just consider the least integer non-negative, call it $l$, such
that $\tf^l(\tau,z)\in\Crit(\tf)$. Let $s$ be the minimum of those
numbers. Then conditions (a) and 
(b) are satisfied. If $s=0$ then (c) is also satisfied since the identity map 
has no critical points. So, we may assume that $s\ge 1$. By the definition of $s$
we have $r|(\tf^{s-1})'(\tau,z)|\le ||\tf'||^{-1}$, whence
$$
\aligned
\th r|f_{\tau|_s}'(z)|
&    =\th r|(\tf^s)'(\tau,z)|
     =\th r|(\tf^{s-1})'(\tau,z)|\cdot|\tf'(\tf^{s-1}(\tau,z))| \\
&\le  \th ||\tf'||^{-1}||\tf'||
     =\th
     <\g.
\endaligned
$$
Thus, (\ref{2.25ku1h7}) yields for all $0\le j\le s$ that
$$
\diam\(\Comp\(f_{\tau|_{s-j}}(z),f_{\tau|_{s-j+1}^s},\th r|f_{\tau|_s}'(z)|\)\)\le \b.
$$
It therefore follows from Lemma~\ref{l2.16ku1h7} that there exist $0\le p\le \#\Crit(f)$,
an increasing sequence of integers $1\le k_1<k_2<\ld<k_p\le s$ and mutually 
distinct critical pairs $(c_1,\tau_{s-k_1+1}),(c_2, \tau_{s-k_2+1}),$ $\ld,
(c_p, \tau_{s-k_p+1})$ such that 
$f_{\tau _{s-k_{l}+1}}(c_{l})\in J(G)$ 
and
$$
\Comp\(f_{\tau|_{s-k_l}}(z),f_{\tau|_{s-k_l+1}^s},\th r|f_{\tau|_s}'(z)|\)
\cap\Crit(f_{\tau_{s-k_l+1}})
=\{c_l\}
$$
for every $l=1,2,\ld,p$, and, in addition, if $j\notin\{k_1,k_2,\ld,k_p\}$, then
\beq\label{u6.12h29}
\Comp\(f_{\tau|_{s-j}}(z), f_{\tau|_{s-j+1}^s},\th r|f_{\tau|_s}'(z)|\)
\cap\Crit(f_{\tau _{s-j+1}})
=\es.
\eeq
Setting $k_0=0$, we shall prove by induction that for every $0\le l\le  p$, we have
\beq\label{u6.13h29}
\Comp\(f_{\tau|_{s-k_l}}(z), f_{\tau|_{s-k_l+1}^s},(KA_{f}^2)^{-1}2^{-l}\th r|f_{\tau|_s}'(z)|\)
\cap\Crit(f_{\tau|_{s-k_l+1}^s})
=\es,
\eeq
where $f_{\tau _{v}^{s}}=Id $ if $s<v.$ 
Indeed, for $l=0$ there is nothing to prove. So, suppose that (\ref{u6.13h29}) is
true for some $0\le l\le  p-1$. Then using (\ref{u6.12h29}) we get
$$
\Comp\(f_{\tau|_{s-k_{l+1}+1}}(z),f_{\tau|_{s-k_{l+1}+2}^s},(KA_{f}^2)^{-1}2^{-l}\th r|f_{\tau|_s}'(z)|\)
\cap\Crit(f_{\tau|_{s-k_{l+1}+2}^s})
=\es.
$$
So, if 
$$
c_{l+1}\in 
\Comp\(f_{\tau|_{s-k_{l+1}}}(z),f_{\tau|_{s-k_{l+1}+1}^s},(KA_{f}^2)^{-1}2^{-(l+1)}\th r|f_{\tau|_s}'(z)|\),
$$
then, by Lemma~\ref{lncp12.11} applied to holomorphic maps $H=f_{\tau_{s-k_{l+1}+1}}$
and $Q=f_{\tau|_{s-k_{l+1}+2}^{s}}$, $z$ being $f_{\tau|_{s-k_{l+1}}}(z)$ and the radius $R=
(KA_{f}^2)^{-1}2^{-(l+1)}\th r|f_{\tau|_s}'(z)|<\g$, we get
$$
\aligned
\lt|f_{\tau|_{s-k_{l+1}}}(z)-c_{l+1}\rt|
& \le KA_{f}^2\lt|f_{\tau|_{s-k_{l+1}+1}^s}'\(f_{\tau|_{s-k_{l+1}}}(z)\)\rt| ^{-1}
      (KA_{f}^2)^{-1}2^{-(l+1)}\th r|f_{\tau|_s}'(z)| \\
& =   2^{-(l+1)}\th r|f_{\tau|_{s-k_{l+1}}}'(z)| \\
& \le           \th r|f_{\tau|_{s-k_{l+1}}}'(z)|,
\endaligned
$$
which along with the facts that 
$c_{l+1}\in\Crit\(f_{\tau|_{s-k_{l+1}+1}}\)$ 
and 
$f_{\tau _{s-k_{l+1}+1}}(c_{l+1})\in J(G)$
contradicts 
the definition of $s$ and proves (\ref{u6.13h29}) for $l+1$. In particular, it
follows from (\ref{u6.13h29}) with $l=p$ and (\ref{u6.12h29}) with $j=k_p+1,k_p+2,\ld,
s$, that
$$
\Comp\(z, f_{\tau|_s},(KA_{f}^2)^{-1}2^{-\#\Crit(f)}\th r|f_{\tau|_s}'(z)|\)
     \cap\Crit(f_{\tau|_s})=\es.
$$
We are done. \endpf

\section{Geometric Measures Theory and Conformal Measures; Preliminaries}

\fr In this section we deal in detail with Hausdorff and packing measures
and we also establish some geometrical properties of conformal measures.

\

\subsection{Preliminaries from Geometric Measure Theory; Hausdorff and 
Packing Measures}

\ni Given a subset $A$ of a metric space $(X,d)$, a  
countable family $\{B(x_i,r_i)\}_{i=1}^\infty$ of open balls centered at points
of $A$ is said to be a packing of $A$ if and only if for any pair $i\ne j$
$$
d(x_i,x_j) > r_i + r_j.
$$
Given $t\ge 0$, the $t$-dimensional outer Hausdorff measure 
$\H^t(A)$ of the set $A$ is defined as
$$
\H^t(A)=\sup_{\e>0}\inf\bigl\{\sum_{i=1}^\infty \diam^t(A_i)\bigr\}
$$                      
where infimum is taken over all countable covers $\{A_i\}_{i=1}^\infty$ of the 
set $A$ by sets whose diameters do not exceed $\e$. 

\sp\fr The $t$-dimensional outer packing measure $\Pi^t(A)$ of the set 
$A$ is defined as
$$
\Pi^t(A)= \inf_{\cup A_i=A}\bigl\{\sum_i \Pi^t_*(A_i)\bigr\}
$$
($A_i$ are arbitrary subsets of $A$), where 
$$
\Pi^t_*(A)= \sup_{\e>0}\sup\bigl\{\sum_{i=1}^\infty r_i^t\bigr\}.
$$                      
Here the second supremum is taken over all 
packings $\{B(x_i,r_i)\}_{i=1}^\infty$ of the  
set $A$ consisting of open balls whose radii do not exceed $\e$. 
These two outer measures define countable additive measures on the Borel 
$\sg$-algebra of $X$. 

\sp\fr The definition of the Hausdorff dimension $\HD(A)$ of 
the set $A$ is the following 
$$
\HD(A)=\inf\{t:\H^t(A)=0\}=\sup\{t:\H^t(A)=\infty\}.
$$
Let $\nu$ be a Borel probability measure on $X$. Define the function
$\rho=\rho_{t,\nu}:X\times (0,\infty)\to(0,\infty)$ by
$$
\rho(x,r)={\nu(B(x,r))\over r^t}
$$
The following two theorems (see \cite{PUbook, Fal}, and \cite{MS}) are for 
our aims the key facts from geometric measure theory. Their proofs are an 
easy consequence of Besicovi\v c covering theorem (see \cite{PUbook}) or a 
more elementary $4r$-covering theorem (see \cite{MS}). 

\

\bthm\lab{tncp12.1.} 
Let $\nu$ be a Borel probability measure on $\R^n$ with some $n\ge 1$.
Then there exists a constant $b(n)$ depending only on $n$ 
with the following properties. If $A$ is a Borel subset of $\R^n$ and
$C>0$ is a positive constant such that 
\begin{itemize}
\item[(1)] for all 
$x\in A$
$$
\limsup_{r\to 0} \rho_{t,\nu}(x,r) \ge C^{-1},
$$
then for every Borel subset $E\sbt A$ we have $\H^t(E)\le b(n)C\nu (E)$ and, in 
particular, $\H^t(A)<\infty$, \nl 
or
\item[(2)] for all $x\in A$
$$
\limsup_{r\to 0}\rho_{t,\nu}(x,r) \le C^{-1},
$$
then for every Borel subset $E\sbt A$ we have $\H^t(E)\ge C\nu (E)$.
\item[(1)'] If $t>0$ then 
(1) holds under the weaker assumption that the hypothesis of part (1) is satisfied on the 
complement of a countable set. 
\end{itemize}
\ethm
                      
\

\bthm\lab{tncp12.2.}
Let $\nu$ be a Borel probability measure on $\R^n$ with some $n\ge
1$. Then there exists a constant $b(n)$ depending only on $n$  
with the following properties. If $A$ is a Borel subset of $\R^n$ and
$C>0$ is a positive constant such that 
\begin{itemize}  
\item[(1)] for all $x\in A$
$$
\liminf_{r\to 0}\rho_{t,\nu}(x,r)\le C^{-1},
$$
then for every Borel subset $E\sbt A$ we have $\Pi^t(E)\ge Cb(n)^{-1}\nu(E)$, 
\nl
or
\item[(2)] for all $x\in A$
$$
\liminf_{r\to 0} \rho_{t,\nu}(x,r) \ge C^{-1},
$$
then $\Pi^t(E)\le C\nu (E)$ and, consequently, $\Pi^t(A)<\infty$.  
\item[(1')] If $\rho $ is non--atomic then (1) holds under the weaker 
assumption that the hypothesis of part (1) is satisfied on the
complement of a countable set. 
\end{itemize}
\ethm

\

\section{Conformal Measures; Existence, Uniqueness, and Continuity}

\ni For every $t\ge 0$ and every function $\phi:J(\tf)\to\C$ 
let $\L_t\phi:
J(\tf)\to\C$ be defined by the following formula:
$$
\L_t\phi(y)=\sum_{x\in \tf^{-1}(y)}|\tf'(x)|^{-t}\phi(x).
$$
$\L_t\phi(y)$ is finite if and only if $y\notin\Crit(\tf)$. Otherwise $\L_t\phi(y)$ is
declared to be $\infty$. Iterating this formula we get for all $n\ge 1$ that
$$
\L_t^n\phi(y)=\sum_{x\in \tf^{-n}(y)}|(\tf^n)'(x)|^{-t}\phi(x).
$$
If $y\in J(\tf)\sms p_2^{-1}(G^{\ast }(\Crit(f)_+))$, then
$\L_t^n\1(y)$ is finite for all $n\ge 0$. 
If $\psi:\oc\to\C$, then define $\L_t\psi:\oc\to\C$ by the formula
$$
\L_t\psi(z)=\sum_{i=1}^u\sum_{x\in f_i^{-1}(z)}|f_i'(x)|^{-t}\psi(x).
$$
It will be always clear from the context whether $\L_t$ is applied to a function 
defined on $J(\tf)$ or on a compact neighborhood $\mathcal{A}$ of
$J(G)$. Iterating this formula we get for all $n\ge 1$ that 
\beq\label{1h32}
\L_t^n\psi(z)=\sum_{|\om |=n}\sum_{x\in f_{\om }^{-1}(z)}|f_{\om }'(x)|^{-t}\psi(x).
\eeq
Note that if $\^\psi:J(\tf)\to\C$ is defined by the formula $\^\psi(\tau,z)=\psi(z)$,
then 
$$
\L_t^n\^\psi(\tau,z)=\L_t^n\psi(z)
$$
for all $(\tau,z)\in J(\tf)$. Without confusion we put $\^\1=\1$. Note that $\L_t^n\psi(z)$
is finite for all $z\in \mathcal{A} \setminus  G^{\ast }(\Crit(f)_+)$. For all $z\in \mathcal{A} \sms G^{\ast }(\Crit(f)_+)$ set
$$
\P_z(t)=\limsup_{n\to\infty}{1\over n}\log\L_t^n\1(z)\in(-\infty,+\infty].
$$

\bdfn
 Denote by 
$\mbox{{\em PCV}}(\tf)$ the closure of the postcritical set of $\tf$, i.e.
$$
\mbox{{\em PCV}}(\tf)=\ov{\bu_{n=1}^\infty \tf^n(\Crit(\tf))}.
$$
\edfn 
\blem\label{l1h63}
$\ov{G^{\ast }(\Crit(f)_+)}\cap J(G)$ is a nowhere dense subset of $J(G)$ and $\mbox{{\em PCV}}(\tf)$ is a nowhere dense 
subset of $J(\tf)$. 
\elem

{\sl Proof.} Since, by Lemma~\ref{lku22.20h11}, $\om_G(\Crit(f)_+)\cap J(G)$ is nowhere dense in
$J(G)$ and since the set $G^{\ast }(\Crit(f)_+)$ is countable, it follows from the Baire
Category Theorem the set $\ov{G^{\ast }(\Crit(f)_+)}\cap J(G)$ is nowhere dense. In order to prove 
the second part of our lemma, suppose that $\PCV(\tf)$ is not nowhere dense in $J(\tf)$. 
This means that $\PCV(\tf)$ has non-empty interior, and therefore, because of it forward
invariance and topological exactness of the map $\tf:J(\tf)\to J(\tf)$, we have $\PCV(\tf)
=J(\tf)$. Hence $J(G)=p_2(J(\tf))=p_2(\PCV(\tf))\sbt \ov{G^{\ast }(\Crit(f)_+)}\cap J(G)$, contrary to, the
already proved, first part of the lemma. \endpf

\

We shall prove the following.

\

\blem\label{l1h33}
The function $z\mapsto\P_z(t)$ is constant throughout a neighborhood
of $J(G)\sms\ov{G^{\ast }(\Crit(f)_+)}$. 
\elem

{\sl Proof.} For every $z\in J(G)\sms\ov{G^{\ast }(\Crit(f)_+)}$ fix
$U_z=\{ w\mid |w-z|<r\} $, an open round disk centered at  
$z$ and such that $\{ w\mid |w-z|<2r\} $ is disjoint from $G^{\ast
}(\Crit(f)_+)$. It then directly follows from  
Koebe's Distortion Theorem that the function $w\mapsto\P_w(t)$ is constant on $U_z$.
Now, fix $z_1,z_2\in J(G)\sms\ov{G^{\ast }(\Crit(f)_+)}$. By
\cite[Lemma 3.2]{HM}, there exists $g=f_{\om }\in G$ such that  
$g(U_{z_1})\cap U_{z_2}\cap J(G)\ne\es$. Fix $x\in U_{z_1}$ such that
$g(x)\in U_{z_2}\cap J(G)$. 
Then $x\in J(G)$ and for every $n\ge 1$, $\L_t^{n+|\om |}\1(g(x))\ge
|g'(x)|^{-t}\L_t^n\1(x)$. 
Therefore, $\P_{g(x)}(t)\ge \P_x(t)$. Hence $\P_{z_2}(t)\ge
\P_{z_1}(t)$. Exchanging the roles 
of $z_1$ and $z_2$, we get $\P_{z_1}(t)\ge \P_{z_2}(t)$, and we are done. \endpf

\

\ni By Lemma~\ref{l1h63} the set $J(G)\sms \ov{G^{\ast }(\Crit(f)_+)}$
is not empty. Denote by  
$\P(t)$ the constant common value of the function $z\mapsto \P_z(t)$
on $J(G)\sms \ov{G^{\ast }(\Crit(f)_+)}$. 
$\P(t)$ is called the topological pressure of $t$. Its basic properties are contained
in the following.

\

\blem\label{l1h35}
The function $t\mapsto\P(t)$, $t\ge 0$, has the following properties.
\begin{itemize}
\item[(a)] $\P(t)$ is non-increasing. In particular $\P(t)<+\infty$ as clearly $\P(0)<+\infty$.
\item[(b)] $\P(t)$ is convex and, hence, continuous.
\item[(c)] $\P(0)\ge \log 2>0$.
\item[(d)] $\P(2)\le 0$.
\end{itemize}
\elem

{\sl Proof.} Fix $z\in J(G)\sms \ov{G^{\ast }(\Crit(f)_+)}$. Since the family of all analytic 
inverse branches of all elements of $G$ is normal in some neighborhood of $z$ (see \cite[Lemma 4.5]{hiroki1}) 
and all
its limit functions are constant (see Theorem~\ref{t1h4}), 
$\lim_{n\to\infty}\max\{|f_{\om }'(x)|:|\om |=n, x\in f_{\om }^{-1}(z)\}=\infty $.
So, item (a) follows directly from (\ref{1h32}). Item (b), that is convexity of $\P(t)$
follows directly from (\ref{1h32}) and H\"older inequality. Item (c) follows from the
fact that $\max\{u,\max\{\deg(f_j):1\le j\le u\}\}\ge 2$. For the proof of item (d) let
$U\sbt\oc$ be the set coming from the nice open set condition. Fix $z\in J(G)\sms \ov{G^{\ast }(\Crit(f)_+)}$. 
Let $U_z=B\(z,{1\over 2}\dist(z,G^{\ast }(\Crit(f)_+))\)$. It follows from Koebe's Distortion Theorem
that
$$
|(g_*^{-1})'(z)|^{2}\le CK^2{l_2(g_*^{-1}(U_z\cap U))\over l_2(U_z\cap U)}
$$
for all $g\in G$ and all analytic inverse branches $g_*^{-1}$ of $g$ defined on 
$B\left(z, \dist (z, G^{\ast }(\Crit(f)_{+}))\right)$, 
where $C>0$ is a constant independent of $g.$  
Since, by the open set condition, all the sets $g_*^{-1}(U_z\cap U)$ are mutually
disjoint, we thus get
$$
\L_2^n\1(z)
\le CK^2{l_2(\bu g_*^{-1}(U_z\cap U))\over l_2(U_z\cap U)}
\le {CK^2l_{2}(U)\over l_2(U_z\cap U)}.
$$
Hence $\P(2)=\P_z(2)\le 0$ and we are done. \endpf

\

\ni We say that a measure $\^m_t$ on $J(\tf)$ is $e^{\P(t)}|\tf'|^t$-conformal provided that
$$
\^m_t(\tf(A))=\int_Ae^{\P(t)}|\tf'|^td\^m_t
$$
for all Borel sets $A\sbt J(\tf)$ such that $\tf|_A$ is injective. If $\P(t)=0$, the 
measure $\^m_t$ is simply referred to as $t$-conformal. Fix $z\in
J(G)\sms \ov{G^{\ast }(\Crit(f)_+)}$.  
Observe that the critical parameter for the series 
$$
S_s(z)=\sum_{n=1}^\infty e^{-sn}\L_t^n\1(z)
$$
is equal to the topological pressure $\P(t)$, i.e. $S_s(z)=+\infty$ if $s<P(t)$ and 
$S_s(z)<+\infty$ if $s>\P(t)$. For every $\sg$-finite Borel measure
$m$ on $J(\tf)$ let
$\L_t^{*n}m$ be given by the formula
$$
\L_t^{*n}m(A)=m(\L_t^n\1_A), \  A\sbt J(\tf),
$$
where $m(g):=\int gdm$.
Notice that if $(\tau,\xi)\in J(\tf)\setminus \bigcup _{n=1}^{\infty }\tf ^{n}(\Crit(\tf))$, 
then for all Borel sets $A\sbt J(\tf)$ we have
$$
\L_t^{*n}\d_{(\tau,\xi)}(A)
=\d_{(\tau,\xi)}(\L_t^n\1_A)
=\L_t^n\1_A(\tau,\xi)
=\sum_{|\om |=n}\sum_{x\in A\cap f_{\om }^{-1}(\xi)}|f_{\om }'(x)|^{-t}
\le \L_t^n\1(\xi)
<\infty.
$$
In particular, $L_t^{*n}\d_{(\tau,\xi)}(J(\tf))\le \L_t^n\1(\xi)<\infty$. Hence, if
$s>\P(t)$, then
\beq\label{1h37}
\^\nu_s=S_s^{-1}(\xi)\sum_{n=1}^\infty e^{-sn}\L_t^{*n}\d_{(\tau,\xi)}
\eeq
is a Borel probability measure on $J(\tf)$. Now, for every Borel set $A\sbt J(\tf)$ we have
$$
\L_t^{*n}\d_{(\tau,\xi)}(A)
=\d_{(\tau,\xi)}(\L_t^n\1_A)
=\L_t^n\1_A(\tau,\xi)
=\sum_{(\g,z)\in \tf^{-n}(\tau,\xi)\cap A}|(\tf^n)'(\g,z)|^{-t}.
$$
So, $\L_t^{*n}\d_{(\tau,\xi)}(\tf^{-n}(\tau,\xi))=1$. Hence, denoting
$$
\nu_s=\^\nu_s\circ p_2^{-1},
$$
we get the following.

\

\blem\label{l1h36}
We have $\^\nu_s\lt(\bu_{n=0}^\infty \tilde{f}^{-n}(\tau,\xi)\rt)=1$ and 
$\nu_s(G^{-1}(\xi))=1$.
\elem

\

\ni In what follows that we are in the divergence type, i.e. $S_{\P(t)}(\xi)
=+\infty$. For the convergence type situation the usual modifications 
involving slowly varying functions have to be done, the details can be found
in \cite{dusullivan}. The following lemma is proved by a direct straightforward 
calculations.

\

\blem\label{l1h37}
For every $s>\P(t)$ the following hold.
\begin{itemize}
\item[(a)] $\^\nu_s$ is a Borel probability measure.
\item[(b)] For every continuous function $g:J(\tf)\to\R$, we have
$$
\int gd\^\nu_s
=S_s^{-1}(\xi)\sum_{n=1}^\infty e^{-sn}\L_t^ngd\d_{(\tau,\xi)}
=S_s^{-1}(\xi)\sum_{n=1}^\infty e^{-sn}\L_t^ng(\tau,\xi).
$$
\item[(c)]
$$
e^{-s}\L_t^*\^\nu_s
=S_s^{-1}(\xi)\sum_{n=1}^\infty e^{-s(n+1)}\L_t^{*(n+1)}\d_{(\tau,\xi)}
=\^\nu_s-S_s^{-1}(\xi)(e^{-s}\L_t^*\d_{(\tau,\xi)}).
$$
\end{itemize}
\elem

\

\ni Now we can easily prove the following.

\

\bprop\label{p1h39}
For every $t\ge 0$ there exists an $e^{\P(t)}|\tf'|^t$-conformal measure $\^m_t$
for the map $\tf:J(\tf)\to J(\tf)$.
\eprop

{\sl Proof.} Since $\lim_{s\downto\P(t)}S_s(\xi)=+\infty$, it suffices to take
as $\^m_t$ any weak limit of $\^\nu_s$ when $s\downto\P(t)$, and to apply
Lemma~\ref{l1h37}(c). \endpf

\

\ni Consider now a Borel set $A\sbt J(\tf)$ such that $\tf|_A$ is injective. It then
follows from Lemma~\ref{l1h37}(c) that
\beq\label{1h39}
\aligned
\^\nu_s(A)
&= e^{-s}\L_t^*\^\nu_s(\1_A) + S_s^{-1}(\xi)e^{-s}\L_t^*\d_{(\tau,\xi)}(\1_A) \\
&= e^{-s}\int\L_t(\1|_A)d\^\nu_s + e^{-s}S_s^{-1}(\xi)\int\L_t(\1|_A)d\d_{(\tau,\xi)} \\
&= e^{-s}\int\sum_{y\in \tf^{-1}(x)}|\tf'(y)|^{-t}\1_A(y)d\^\nu_s(x)  
    + e^{-s}S_s^{-1}(\xi)\L_t(\1|_A)(\tau,\xi) \\
&= e^{-s}\int_{\tf(A)}|(\tf|_A^{-1})'(x)|^td\^\nu_s(x)
    + \begin{cases} 0 \                  &\text{ if }  \  A\cap \tf^{-1}(\tau,\xi)=\es \\
        e^{-s}S_s^{-1}(\xi)|\tf'(y)|^{-t} &\text{ if }  \  A\cap \tf^{-1}(\tau,\xi)=\{y\}.
      \end{cases}
\endaligned
\eeq
Suppose now that $(\om,x)\in J(\tf)$ and there exists a (unique) continuous inverse branch
 $\f_{(\om,x)}^{-1}:\Sg_u\times B(f_{\om_1}(x),2R)\to \Sg_u\times \C $ of $\tilde{f}$ sending 
$(\sg(\om),f_{\om_1}(x))$ to $(\om,x)$. It then follows from (\ref{1h39}) and Lemma~\ref{l1h37}(c) that for 
every set $A\sbt \Sg_u\times B(f_{\om_1}(x),2R)$, we have
that
\beq\label{3h39}
\aligned
\^\nu_s(\phi _{(\om ,x)}^{-1}(A))
&= e^{-s}\int_A|(\f_{(\om,x)}^{-1})'|^td\^\nu_s
     +
       e^{-s}S_s^{-1}(\xi)|(\f_{(\om,x)}^{-1})'(\tau,\xi)|^t\d_{(\tau,\xi)}(A)  
\endaligned
\eeq
From now on throughout the paper we assume that
\beq\label{2h39}
\P(t)\ge 0.
\eeq
We also require that 
\beq\label{1h41}
\xi\notin \overline{G^{\ast }(\Crit(f)_{+})}.
\eeq
Our goal now is to show that the measure 
$$
m_t=\^m_t\circ p_2^{-1}
$$ 
is uniformly upper $t$-estimable. For every critical point $c\in\Crit(f)$ let
$$
I(c)=\{1\le i\le u:f_i'(c)=0\}
$$
and let
$$
\Sg(c)=\bu_{i\in I(c)}[i]\sbt \Sg_u,
$$
where $[u]$ is defined by formula \eqref{referee2}.
Now suppose that $\Ga$ is a closed subset of $J(G)$ 
such that $g(\Gamma )\cap J(G)\subset \Gamma $ for each $g\in \Gamma $, 
and that $\^m$ is a Borel probability measure on $J(\tf)$. 

\

\bdfn\label{d1h40}
The measure $\^m$ is said to be nearly upper $t$-conformal respective to $\Ga$ 
provided that there exists an $S>0$ such that the following conditions are satisfied.
\begin{itemize}
\item[(a)] For every $z\in \Ga$
$$
\^m(\tf(A))\ge\int_A|\tf'|^td\^m
$$
for every Borel sets $A\sbt J(\tf)\cap p_{2}^{-1}(B(z,S))$ such that $\tf|_A$ is injective.
\item[(b)] For every $c\in\Crit(f)$ such that $\bigcup _{|\tau |=l}
  f_{\tau }(c_+)\cap J(G)\sbt\Ga$  
(the integer $l=l(f)\ge 0$ 
coming from Lemma~\ref{lu5.9h17}) and every $1\le j\le l+1$,
$$
\^m(\tf^j(A))\ge\int_A|(\tf^j)'|^td\^m
$$
for every Borel sets $A\sbt J(\tf)\cap p_{2}^{-1}(B(c,S))$ such that $\tf^j|_A$ is injective.
\item[(c)] $\^m(\Sg(c)\times\{c\})=0$ for every point $c\in\Ga\cap\Crit(f)$.
\end{itemize}
The constant $S$ is said to be the nearly upper conformality radius. 
If $\Ga=J(G)$, we simply say that $\^m$ is nearly upper $t$-conformal. In any case put
$$
m=\^m\circ p_2^{-1}.
$$
\edfn

\

\ni Let us prove the following.

\

\blem\label{l1h43}
Suppose that $\Ga$ is a closed subset of $J(G)$ such that $g(\Gamma )\cap J(G)\subset \Gamma $ for each $g\in G$, and that 
$\^m$ is a Borel probability 
nearly upper $t$-conformal measure on $J(\tf)$ respective to $\Ga$. Fix $i\in\{0,1,\ld,p\}$ and 
suppose that for every critical point $c\in S_i(f)\cap\Ga$ the measure 
$\^m|_{\Sg(c)\times\oc}\circ p_2^{-1}$ is upper $t$-estimable at $c$. Then the measure $m$
is uniformly upper $t$-estimable at all points $z\in J_i(G)\cap\Ga$.
\elem

{\sl Proof.} Since $\Ga$ is a closed set and $\Crit(f)$ is finite, the number 
$\De=\dist(\Ga,\Crit(f)\sms\Ga)$ is positive 
(if $\Crit(f)\setminus \Ga =\emptyset $ then we put $\De =\infty $). 
Fix $\th\in(0,\min\{1,\g\})$ so small that
\beq\label{1h42}
\th||\tilde{f}'||^{-1}<\min \{ \De, \rho \} .
\eeq
Put 
$$
\a=\th(KA_{f}^2)^{-1}2^{-\#\Crit(f)}.
$$
Let $z\in J_{i}(G)\cap \Ga .$ 
Fix $\tau\in\Sg_u$ such that $(\tau,z)\in J(\tf)$, i.e. $\tau\in p_1(J(\tf)\cap p_2^{-1}(z))$. 
Assume $r\in(0,R_{f}]$ to be sufficiently small.
Let $s(\tau,r)=s(\th,(\tau,z),8\a^{-1}r)\ge 0$ be the integer produced in Proposition~\ref{pu6.4h27}.
Set $R_{\tau|_{s(\tau ,r)+1}}=4r|f_{\tau|_{s(\tau ,r)}}'(z)|$. It then follows from Proposition~\ref{pu6.4h27}
that the family 
$$
\Fa(z,r)=\{\tau|_{s(\tau,r)+1}:\tau\in p_1(J(\tf)\cap p_2^{-1}(z))\}
$$ 
is $(4, \gamma , V)$-essential for the pair $(z,r)$, where $V=\bu\{[\tau|_{s+1}]:\tau\in 
p_1(J(\tf)\cap p_2^{-1}(z))\}$. Keep $\tau\in p_1(J(\tf)\cap p_2^{-1}(z))$ and $s=s(\tau,r)$.
Suppose that the first alternative of (b) in Proposition~\ref{pu6.4h27} holds. Then
$8\a^{-1}r|f_{\tau|_s}'(z)|>||\tilde{f}'||^{-1}$. So, using Koebe's Distortion Theorem, 
and assuming $\theta $ is small enough, we
get from nearly upper $t$-conformality of $\^m$ respective to $\Ga$ that
\beq\label{1h45}
\aligned
\^m\(f_{\tau|_s,z}^{-s}([\tau_{s+1}]\times B(f_{\tau|_s}(z),R_{\tau|_{s+1}}))\)
&\le \^m\(f_{\tau|_s,z}^{-s}(p_2^{-1}(B(f_{\tau|_s}(z),R_{\tau|_{s+1}})))\) \\
&\le K^t|f_{\tau|_s}'(z)|^{-t}\^mp_2^{-1}\(B(f_{\tau|_s}(z),R_{\tau|_{s+1}}))\)  \\
&\le K^t|f_{\tau|_s}'(z)|^{-t} \\
&\le (8K\a ^{-1}\| \tf '\| )^tr^t.
\endaligned
\eeq
Now suppose $8\alpha ^{-1}r| f_{\tau |_{s}}'(z)| \leq \| \tf '\| ^{-1}$ which particular implies 
that the second alternative of (b) in Proposition~\ref{pu6.4h27} holds. 
Let $c\in\Crit(f_{\tau_{s+1}})$ such that $f_{\tau |_{s+1}}(c)\in J(G)$ come from item (b)
of this proposition. Since $z\in J_i(G)$ (and $\theta \| \tf'\| ^{-1}<\rho $), it
follows from (\ref{1h17}) and Proposition~\ref{pu6.4h27} that $c\in S_i(f)$. Since
$8\a^{-1}r|f_{\tau|_s}'(z)|\le ||\tilde{f}'||^{-1}$, it follows from Proposition~\ref{pu6.4h27}(b)
and (\ref{1h42}) that $|f_{\tau|_s}(z)-c|\le \th||\tilde{f}'||^{-1}<\De$. Thus $c\in\Ga$. Hence,
making use of Proposition~\ref{pu6.4h27}(b), (c), as well as Koebe's Distortion Theorem,
nearly upper $t$-conformality of $\^m$, and our $t$-upper estimability assumption, 
and assuming $\theta $ is small enough, 
we get with some universal constant $C_1$ that
$$
\aligned
\^m\(f_{\tau|_s,z}^{-s}([\tau_{s+1}] &\times B(f_{\tau|_s}(z),R_{\tau|_{s+1}}))\) \\
&\le K^t|f_{\tau|_s}'(z)|^{-t}\^m\([\tau_{s+1}]\times B(f_{\tau|_s}(z),R_{\tau|_{s+1}})\) \\
&\le K^t|f_{\tau|_s}'(z)|^{-t}\^m|_{\Sg(c)\times\oc}\circ p_2^{-1}\(B(f_{\tau|_s}(z),R_{\tau|_{s+1}})\) \\
&\le K^t|f_{\tau|_s}'(z)|^{-t}\^m|_{\Sg(c)\times\oc}\circ p_2^{-1}
        \(B(c,R_{\tau|_{s+1}}+8\th\a^{-1}r|f_{\tau|_s}'(z)|)\) \\
&\le K^t|f_{\tau|_s}'(z)|^{-t}\^m|_{\Sg(c)\times\oc}\circ p_2^{-1}
        \(B(c,4(1+2\th\a^{-1})r|f_{\tau|_s}'(z)|)\) \\
&\le K^t|f_{\tau|_s}'(z)|^{-t}C_1\(4(1+2\th\a^{-1})r|f_{\tau|_s}'(z)|\)^t \\
&=   C_1(4K(1+2\th\a^{-1}))^tr^t.
\endaligned
$$
Combining this with (\ref{1h45}) and applying Proposition~\ref{p1h2a.1}, we get that
\begin{equation}
\label{eq:mbzrast1}
m(B(z,r))\le \#_{4,\gamma }C_{1}\max\{8K\a ^{-1}\|\tf'\| , 4K(1+2\th\a^{-1})\} ^{t}r^t.
\end{equation}
We are done. \endpf

\

\blem\label{l1h47}
There are two functions $(R,S)\mapsto R^{\ast }$ and $L\mapsto \hat{L}$ with the following 
property.
\begin{itemize}
\item 
Suppose that $\Ga$ is a closed subset of $J(G)$ such that 
$g(\Gamma )\cap J(G)\subset \Gamma $ for each $g\in G$, and that $\^m$ is a Borel probability 
nearly upper $t$-conformal measure on $J(\tf)$ respective to $\Ga$ 
with nearly upper conformality radius $S.$ Fix $i\in\{0,1,\ld,p\}$ 
and suppose that the measure $m$ is uniformly upper $t$-estimable at all points 
$z\in J_i(G)\cap\Ga$ with corresponding estimability constant $L$ and estimability radius $R.$ 
 Then the measure $\^m|_{\Sg(c)\times\oc}\circ p_2^{-1}$ is $t$-upper 
estimable, with upper estimability constant $\hat{L}$ 
and radius $R^{\ast }$ at every point 
$c\in Cr_{i+1}(f)$
such that $\bigcup _{|\om |=l}f_{\om }(c_+)\cap J(G)\sbt\Ga$.
\end{itemize}
\elem 

{\sl Proof.} Fix $c\in Cr_{i+1}(f)$ such that $\bigcup _{|\om |=l}f_{\om }(c_+)\sbt\Ga$ and also $j\in\{0,1,\ld,u\}$
such that $f_j'(c)=0$. Consider an arbitrary $\tau\in\Sg_u$ such that $\tau_1=j$ and 
$(\tau,c)\in J(\tf)$. In view of Lemma~\ref{lu5.9h17}
$$
f_{\tau|_{l+1}}(c)\in J_i(G)\cap\Ga.
$$
Let $R>0$ (sufficiently small) be the radius resulting from uniform $t$-upper estimability 
at all points of $J_i(G)\cap\Ga$. Let $D_{\tau|_{l+1}}(c)$ be the connected component of
$f_{\tau|_{l+1}}^{-1}(B(f_{\tau|_{l+1}}(c),R))$ containing $c$. Set 
$$
\nu_{\tau|_{l+1}}=\^m|_{[\tau|_{l+1}]\times\oc}\circ p_2^{-1}|_{D_{\tau|_{l+1}}(c)}.
$$
Applying nearly upper $t$-conformality of $\^m$ we get for every Borel set $A\sbt D_{\tau|_{l+1}}(c)
\sms \{c\}$ such that $f_{\tau|_{l+1}}|_A$ is injective, the following.
$$
m(f_{\tau|_{l+1}}(A))
=\^m(\Sg_u\times f_{\tau|_{l+1}}(A)))
=\^m(\tilde{f}^{l+1}([\tau|_{l+1}]\times A))
\ge \int_A|f_{\tau|_{l+1}}'(x)|^td\nu_{\tau|_{l+1}}(x).
$$
It therefore follows from Lemma~\ref{lku4.10} and item (c) of Definition~\ref{d1h40} that
the measure $\nu_{\tau|_{l+1}}$ is upper $t$-estimable at $c$ with upper estimability 
constant $L_{0}$ and radius $R_{0}$ independent of $\^m$ (but possibly $R_{0}$ depends on $(R,S)$ and 
$L_{0}$ depends on $L$). Let
$$
\Fa=\{\tau|_{l+1}:(\tau,c)\in J(\tf) \  \text{ and } \ f_{\tau_1}'(c)=0\}.
$$
Let $D_c=\bi_{\om\in\Fa}D_\om(c)$. Since $\#\Fa\le u^{l+1}$ and since 
$$
\^m|_{\Sg(c)\times\oc}\circ p_2^{-1}|_{D_{c}}=\sum_{\om\in\Fa}\nu_\om |_{D_{c}},
$$
we conclude that the measure $\^m|_{\Sg(c)\times\oc}\circ p_2^{-1}$ is $t$-upper 
estimable at the point $c$ with upper estimability constant $\hat{L}$ 
and radius $R^{\ast }$ independent of $\^m$.
We are done. \endpf

\

\ni Now, a straightforward inductive reasoning based on Lemma~\ref{l1h43} and (\ref{eq:mbzrast1}), (which also
give the base of induction since $S_0(f)=\es$), and Lemma~\ref{l1h47} yields the
following.

\

\blem\label{l1h49}
Suppose that $\Ga$ is a closed subset of $J(G)$ such that 
$g(\Gamma )\cap J(G)\subset \Gamma $ for each $g\in G$, 
and that $\^m$ is a Borel probability 
nearly upper $t$-conformal measure on $J(\tf)$ respective to $\Ga$ with nearly upper conformality radius $S.$ 
Then the measure 
$m=\^m\circ p_2^{-1}$ is uniformly upper $t$-estimable at every point of $\Ga$ and 
$\^m|_{\Sg(c)\times{\oc}}\circ p_2^{-1}$ is upper $t$-estimable, with upper estimability 
constants and radii independent of the measure $\^m$ (but possibly dependent on $S$), 
at every point $c\in\Ga\cap\Crit(f)$.
\elem

\ni Now we are in the position to prove the following.

\blem\label{l1h51}
If $\P(t)\ge 0$, then the measure $m_t=\^m_t\circ p_2^{-1}$ is uniformly upper $t$-estimable.
\elem

{\sl Proof.} Fix $s>\P(t)\ge 0$ and consider the measure $\^\nu_s$ defined in (\ref{1h37}).
We want to apply Lemma~\ref{l1h49} with $\Ga=\ov{G^{\ast }(\Crit(f)_{+}\cap J(G))}\cap J(G)$ and $\^m=\^\nu_s$.
For this we have to check that $\^\nu_s$ is nearly upper $t$-conformal respective to $\Ga$. Condition
(c) of Definition~\ref{d1h40} follows directly from Lemma~\ref{l1h36} and the fact that
$\xi\notin G(\Crit(f))$ (see (\ref{1h41})). Since $\xi\notin\Ga$ and $G(\Ga)\cap J(G)\subset \Ga$, there 
exists an $S_{0}>0$ such that $\xi\notin\bu_{j=1}^uf_j(B(\Ga,S_{0}))\cap J(G)$. Formula (\ref{3h39}) then
yields that for every $z\in \Gamma $, 
$$
\^\nu_s(\tf(A))
=e^s\int_A|\tf'|^td\^\nu_s
\ge \int_A|\tf'|^td\^\nu_s
$$
for every Borel set $A\sbt\Sg_u\times B(z,S_{0})$ such that $\tilde{f}|_A$ is injective. Thus,
condition (a) of Definition~\ref{d1h40} is also verified. Condition (b) of this definition
follows by iterating the above argument $l+1$ times and keeping in mind that 
$\xi\notin \overline{G^{\ast }(\Crit(f)_{+})}$. 
Hence, there exists a constant $S$ such that for each $s>P(t)$, $\tilde{\nu }_{s}$ is 
nearly upper $t$-conformal respective to $\Gamma $ with nearly upper conformality radius $S.$ 
Therefore, Lemma~\ref{l1h49} applies and we conclude that all measures
$\^\nu_s|_{\Sg(c)\times\oc}\circ p_2^{-1}$ are upper $t$-estimable at respective points
$c\in\Crit(f)\cap J(G)$ with estimability constants and radii independent of $s>\P(t)$.
Therefore, $\^m_t$, a weak limit of measures $\^\nu_s$, $s>\P(t)$, (see the proof of 
Proposition~\ref{p1h39})) also enjoys the property that $\^m_t|_{\Sg(c)\times\oc}\circ p_2^{-1}$
is  upper $t$-estimable at respective points $c\in\Crit(f)\cap J(G)$. Consequently
$\^m_t(\Sg(c)\times\{c\})=0$. Having this we immediately see from Proposition~\ref{p1h39}
that the measure $\^m_t$ is nearly upper $t$-conformal, i.e. respective to $\Ga=J(G)$. 
So, applying Lemma~\ref{l1h49}, we conclude that the measure $m_t=\^m_t\circ p_2^{-1}$ is 
uniformly upper $t$-estimable at every point of $\Ga=J(G)$. We are done. \endpf

\

\ni  
Now we assume that $t=h$, i.e. $\P(t)=0$ and we deal with the problem of lower 
estimability. It is easier than the upper one. We start with the following.

\

\blem\label{l1h53}
Fix $i\in\{0,1,\ld,p\}$ and suppose that for every critical point $c\in S_i(f)$
and every $j\in I(c)$ the measure $\^m_h|_{[j]\times\oc}\circ p_2^{-1}$ is strongly
lower $h$-estimable at $c$ with sufficiently small lower estimability size. Then $m_h$ is
uniformly strongly lower $h$-estimable at all points of $J_i(G)$. 
\elem

{\sl Proof.} 
Let $\theta \in (0,\min\{ 1,\gamma \} )$ be such that $\theta \| \tilde{f}'\| ^{-1}<\rho .$ 
Put $\alpha := \theta ^{-1}KA_{f}^{2}2^{\# \text{Crit}(f)+5}.$ Let 
$$
\lambda =\max\{\lambda (c):c\in S_i(f)\},
$$
where all $\lambda (c)$ are lower estimability sizes at respective critical points $c$. Fix
$z\in J_{i}(G)\setminus S_{i}(f)$ and take $\tau\in\Sg_u$ such that
$(\tau,z)\in J(\tf)$. Assume $r>0$ to be  
sufficiently small. Let $s=s(\th,(\tau,z),\alpha r)\ge 0$ be the integer produced in
Proposition~\ref{pu6.4h27} for the point $z$ and radius $r$. A straightforward calculation
based on Proposition~\ref{p1h39} shows that
$$
\nu_1=\^m_h|_{[\tau|_s]\times f_{(\tau|_s,z)}^{-1}(B(f_{\tau|_s}(z),32r|f_{\tau|_s}'(z)|))}
           \circ p_2^{-1} 
\  \text{ and }  \
\nu_2=m_h|_{B(f_{\tau|_s}(z),32r|f_{\tau|_s}'(z)|)}
$$
form an $h$-conformal pair of measures with respect to the map 
$$
f_{\tau|_s}:f_{\tau|_s,z}^{-1}(B(f_{\tau|_s}(z),32r|f_{\tau|_s}'(z)|))\to 
B(f_{\tau|_s}(z),32r|f_{\tau|_s}'(z)|).
$$
By Koebe's ${1\over 4}$-Theorem (Theorem~\ref{kdt1/4})
for every $x\in B(z,r)$ we have 
\beq\label{1h52}
B(x,r)\sbt f_{\tau|_s,z}^{-1}\(B(f_{\tau|_s}(z),8r|f_{\tau|_s}'(z)|)\).
\eeq
So,
$$
B(f_{\tau|_s}(x),r|f_{\tau|_s}'(z)|)\sbt B(f_{\tau|_s}(z),9r|f_{\tau|_s}'(z)|).
$$
By Koebe's Distortion Theorem we also get (with small enough $\lambda $) 
$$
B(x,\lambda  r)\spt f_{\tau|_s,z}^{-1}\(B\(f_{\tau|_s}(x),K^{-1}\lambda  r|f_{\tau|_s}'(z)|\)\).
$$
In virtue of Koebe's Distortion Theorem and $t$-conformality of the pair $(\nu_1,\nu_2)$,
we get as a consequence of all of this that
$$
\aligned
m_h(B(x,\lambda  r))
&\ge \nu_1(B(x,\lambda  r))
 \ge \nu_1\(f_{\tau|_s,z}^{-1}\(B\(f_{\tau|_s}(x),K^{-1}\lambda  r|f_{\tau|_s}'(z)|\)\)\) \\
&=   \int_{B\(f_{\tau|_s}(x),K^{-1}\lambda  r|f_{\tau|_s}'(z)|\)}|(f_{\tau|_s,z}^{-1})'|^hd\nu_2 \\
&\ge K^{-h}|f_{\tau|_s}'(z)|^{-h}\nu_2\(B\(f_{\tau|_s}(x),K^{-1}\lambda  r|f_{\tau|_s}'(z)|\)\).
\endaligned
$$
Suppose now that the first alternative in Proposition~\ref{pu6.4h27}(b) holds. We then can continue
the above estimate as follows.
\beq\label{1h55}
\aligned
m_h(B(x,\lambda  r))
&\ge K^{-h}|f_{\tau|_s}'(z)|^{-h}\nu_2\(B\(f_{\tau|_s}(x),K^{-1}\lambda ||\tilde{f}'||^{-1}\)\) \\
&=   K^{-h}|f_{\tau|_s}'(z)|^{-h}  m_h\(B\(f_{\tau|_s}(x),K^{-1}\lambda ||\tilde{f}'||^{-1}\)\)
\endaligned
\eeq
By conformality the measure $\^m_h$ is positive on open subsets of $J(\tf)$, and so, the measure 
$m_h$ is positive on open subsets of $J(G)$. Therefore, for every $R>0$,
$$
M_R=\inf\{m_h(B(w,R):w\in J(G)\}>0.
$$
Hence, (\ref{1h55}) gives that
$$
m_h(B(x,\lambda  r))\ge K^{-h}M_{K^{-1}\lambda ||\tilde{f}'||^{-1}}|f_{\tau|_s}'(z)|^{-h}. 
$$
By minimality of $s=s(\th,(\tau,z),\alpha r)$ we have $\alpha r|f_{\tau|_{s-1}}'(z)|\le ||\tilde{f}'||^{-1}$ 
($s\ge 1$ assuming $r>0$ to be sufficiently small). Hence $|f_{\tau|_s}'(z)|\le (\alpha r)^{-1}$, 
and therefore
$$
m_h(B(x,\lambda  r))\ge K^{-h}M_{K^{-1}\lambda ||\tilde{f}'||^{-1}}\alpha ^{h}r^h.
$$
So suppose that $\alpha r\| (\tilde{f}^{s})'(\tau ,z)\| \leq \| \tilde{f}'\| ^{-1}$ 
and the second alternative in Proposition~\ref{pu6.4h27}(b) holds. 
Let $c\in \text{Crit}(f_{\tau _{s+1}})$ be such that $f_{\tau _{s+1}}(c)\in J(G)$ and 
$|f_{\tau |_{s}}(z)-c|\leq \theta \alpha r|f_{\tau |_{s}}'(z)|\leq \theta \| \tilde{f}'\| ^{-1}<\rho .$ 
Since $z\in J_{i}(G)$, we obtain $c\in S_{i}(f).$ 
Then, 
using (\ref{1h52}), we get
$$
f_{\tau |_{s}}(x)\in   B(f_{\tau|_s}(z),8r|f_{\tau|_s}'(z)|)
    \sbt  B(c,(\th \alpha +8)r|f_{\tau|_s}'(z)|)
$$
and
$$
     B\(f_{\tau|_s}(x),\lambda (\th \alpha +8)r|f_{\tau|_s}'(z)|\)
\sbt B\(f_{\tau|_s}(z),(8+\lambda (\th \alpha +8))r|f_{\tau|_s}'(z)|\)
\sbt B\(f_{\tau|_s}(z),9r|f_{\tau|_s}'(z)|\)
$$
if $\lambda >0$ is small enough. Hence, using conformality of the pair $(\nu_1,\nu_2)$,
Koebe's Distortion Theorem, the fact that $\tau_{s+1}\in I(c)$, and the lower 
$h$-estimability $\^m_h|_{[\tau_{s+1}]\times\oc}\circ p_2^{-1}$ at the point $c$, we 
get that
$$
\aligned
m_h(B(x,K\lambda (\th \alpha +8)r)
&\ge \nu_1(B(x,K\lambda (\th \alpha +8)r)
 \ge \nu_1\(f_{(\tau|_s,z)}^{-1}\(B\(f_{\tau|_s}(x),\lambda (\th \alpha +8)r|f_{\tau|_s}'(z)|\)\)\) \\
&\ge K^{-h}|f_{\tau|_s}'(z)|^{-h}\nu_2\(B\(f_{\tau|_s}(x),\lambda (\th \alpha +8)r|f_{\tau|_s}'(z)|\)\) \\
&\ge K^{-h}|f_{\tau|_s}'(z)|^{-h}\^m_h|_{[\tau_{s+1}]\times\oc}\circ p_2^{-1}
           \(B\(f_{\tau|_s}(x),\lambda (\th \alpha +8)r|f_{\tau|_s}'(z)|\)\) \\
&\ge K^{-h}|f_{\tau|_s}'(z)|^{-h}L_{0}((\th \alpha +8)r|f_{\tau|_s}'(z)|)^h \\
&=   L_{0}((\th \alpha +8)K^{-1})^hr^h, 
\endaligned
$$
where $L_{0}$ is a constant independent of $x$ and $r.$ 
So, we are done with the lower estimability size $K\lambda (\th \alpha +8)$. \endpf

\

\ni Now we shall prove the following.

\blem\label{l1h57}
Fix $i\in\{0,1,\ld,p\}$ and suppose that the measure $m_h$ is uniformly strongly lower $h$-estimable 
at all points of $J_i(G)$. Then the measure $\^m_h|_{[j]\times\oc}\circ p_2^{-1}$ is strongly
lower $h$-estimable at every critical point $c\in Cr_{i+1}(f)$ and every $j\in I(c)$.
\elem

{\sl Proof.} Fix $c\in Cr_{i+1}(f)$ and then an arbitrary $j\in I(c)$. Next consider an arbitrary
$\tau\in\Sg_u$ such that $\tau_1=j$ and $(\tau,c)\in J(\tf)$. Now, ignoring $\Ga$, follow 
the proof of Lemma~\ref{l1h47} up to the definition of the measure $\nu_{\tau|_{l+1}}$. It 
follows from conformality of $\tilde{m}_h$ that the measure $\nu_{\tau|_{l+1}}$ on $D_{\tau|_{l+1}}(c)$
and $\^m_h|_{\Sg_u\times B(f_{\tau|_{l+1}}(c),R)}\circ p_2^{-1}=m_h|_{B(f_{\tau|_{l+1}}(c),R)}$
form an $h$-conformal pair of measures for the map $f_{\tau|_{l+1}}:D_{\tau|_{l+1}}(c)\to
B(f_{\tau|_{l+1}}(c),R)$. So the measure $\nu_{\tau|_{l+1}}$ is strongly lower $h$-estimable at 
$c$ in virtue of our assumption and Lemma~\ref{lku4.11}. Since $D_{\tau|_{l+1}}(c)$ is an
open neighborhood of $c$ and $[\tau|_{l+1}]\times D_{\tau|_{l+1}}(c)\sbt [j]\times\oc$, we thus
see that the measure $\^m_h|_{[j]\times\oc}\circ p_2^{-1}$ is strongly lower $h$-estimable at $c$.
We are done. \endpf

\ 

\ni The second main result of this section is this.

\blem\label{l1h59}
The measure $m_h=\^m_h\circ p_2^{-1}$ is uniformly strongly lower $h$-estimable.
\elem

{\sl Proof.} Having $J_p(G)=J(G)$ (Lemma~\ref{lu5.8h17}) the proof of this lemma is the 
obvious mathematical induction based on
Lemma~\ref{l1h53} and Lemma~\ref{l1h57} as inductive steps and Lemma~\ref{l1h53} with
$i=0$ (then $S_i(G)=\es$ and its hypothesis are vacuously fulfilled) serving as the base 
of induction. \endpf

\

\ni Recall that two measures are said to be equivalent if they are
absolutely continuous one with the other. Since every uniformly
strongly lower $h$-estimable measure is uniformly lower $h$-estimable, 
as an immediate consequence of Lemma~\ref{l1h51}, Lemma~\ref{l1h59},
and \cite{Fal, mattila, PUbook}, we obtain the 
following main result of this section and one of the two main results of the entire paper.

\bthm\label{t1h61}
Under Assumption {\em (}$\ast ${\em )} (formulated just after
Theorem~\ref{t1h4}), we have the following.  
\begin{itemize}
\item[(a)] The measure $m_h=\^m_h\circ p_2^{-1}$ is geometric meaning
  that there exists a constant 
$C\ge 1$ such that 
$$
C^{-1}\le {m_h(B(z,r))\over r^h}\le C
$$
for all $z\in J(G)$ and all $r\in(0,1]$.
\end{itemize}
Consequently,
\begin{itemize}
\item[(b)] $h=\HD(J(G))=\PD(J(G))=\BD(J(G))$.
\item[(c)] $\HD(J(G))$ is the unique zero of $t\mapsto P(t).$ 
\item[(d)] All the measures $\H^h$, $\P^h$, and $m_h$ are 
  equivalent one with each other with Radon-Nikodym 
derivatives uniformly separated away from zero and infinity.
\end{itemize}
In particular
\begin{itemize}
\item[(e)] $0<\H^h(J(G))),\P^h(J(G)))<+\infty$.
\end{itemize}
\ethm

\bdfn
The unique zero of $t\mapsto P(t)$ is denoted by $h=h(f).$ 
Note that $h(f)=\HD(J(G))=\PD(J(G))=\BD(J(G)).$ 
\edfn

\bcor
\label{t1h61cor1}
Under Assumption \mbox{(}$\ast $\mbox{)}, 
for each $z\in J(G)\setminus \ov{G^{\ast }(\Crit(f)_{+})}, $ 
we have $h(f)=T_{f}(z)=t_{0}(f)=S_{G}(z)=s_{0}(G)=\HD(J(G))=\PD(J(G))=\BD(J(G)).$ 
\ecor
{\sl Proof.}
Let $z\in J(G)\setminus \ov{G^{\ast }(\Crit(f)_{+})}. $ 
Since $G$ satisfies the open set condition, $G$ is a free semigroup. 
Hence 
$T_{f}(z)=S_{G}(z)$ and $t_{0}(f)=s_{0}(G).$ Moreover, 
by \cite[Theorem 5.7]{hiroki2}, 
we have $\HD(J(G))\leq s_{0}(G)\leq S_{G}(z).$ 
We now let $a>h(f).$ Since 
$h(f)$ is the unique zero of $P(t)$ and since $t\mapsto P(t)$ is 
non-increasing function, we have $P(a)<0.$ 
Hence there exists a number $v<0$ such that 
for each $n\in \N $, $\sum _{|\om |=n}\sum _{x\in f_{\om }^{-1}(z)}|f_{\om }'(x)|^{-a}\leq e^{nv}.$ 
Therefore $T_{f}(z)\leq a.$ Thus $T_{f}(z)\leq h(f).$ Since $h(f)=\HD(J(G))$, 
it follows that $h(f)=T_{f}(z)=t_{0}(f)=S_{G}(z)=s_{0}(G)=\HD(J(G))=\PD(J(G))=\BD(J(G)).$ 
We are done. 
\endpf

\ 

\ni It follows from Theorem~\ref{t1h61} that the measure $m_h$ is atomless. We thus
get the following.

\

\bcor\label{c3h61} 
Under Assumption {\em (}$\ast ${\em )}, 
we have $\^m_h(\Sing(\tf))=0$.
\ecor

{\sl Proof.} Indeed, the set $\Crit(f)$ is finite and so, $G^{-1}(\Crit(f))$ is 
countable. For all $n\ge 0$ we have
$$
\tf^{-n}(\Crit(\tf))
\sbt p_2^{-1}(p_2(\tf^{-n}(\Crit(\tf))))
\sbt p_2^{-1}(G^{-1}(\Crit(f))). 
$$
Hence, $\^m_h(\tf^{-n}(\Crit(\tf)))\le m_h(G^{-1}(\Crit(f)))=0$. Since
$\Sing(\tf)=\bu_{n=0}^\infty 
\tf^{-n}(\Crit(\tf))$, we are thus done. \endpf

\

\section{Invariant Measures}
\label{s:Inv}

\ni In this section we prove that there exists a unique Borel
probability $\tilde{f}$-invariant  
measure on $J(\tf)$ which is absolutely continuous with respect to
$\^m_h$. This measure is proved to be metrically exact, in particular ergodic.

Frequently in order to denote that a Borel measure $\mu$ is
absolutely  continuous\index{absolutely  continuous measure} with
respect to $\nu$ we  write $\mu \abs \nu \index{\abs@$\abs$}$. We do
not use any special symbol however to record equivalence\index{equivalent
measures} of measures. We use some notations from \cite{Aa}. 
Let $(X,\mathcal{F}, \mu)$ be a $\sg$-finite measure space and let  
$T:X\rightarrow X$ be a
measurable almost everywhere defined transformation. 
$T$ is said to be nonsingular if and only if for any $A\in \mathcal{F}$,  
$\mu (T^{-1}(A))\Leftrightarrow \mu (A)=0.$ 
$T$ is said to be ergodic\index{ergodic measure} with respect to $\mu$, or
$\mu$ is said to be ergodic with respect to $T$, if and only if
$\mu(A)=0$ or $\mu(X \sms A)=0$  whenever the measurable set $A$ is
$T$-invariant, meaning that  $T^{-1}(A)=A$. 
For a nonsingular transformation $T:X\rightarrow X$, 
the measure $\mu$  is
said to  be conservative\index{conservative measure} with respect to
$T$  or $T$ conservative with respect to $\mu$ if and only if for
every measurable set $A$ with $\mu(A)>0$,
$$\mu(\{z\in A: \sum_{n=0}^\infty  1_A\circ T^n(z)< +\infty\})=0.
$$
Note that by \cite[Proposition 1.2.2]{Aa}, 
for a nonsingular transformation $T:X\rightarrow X$, 
$\mu $ is ergodic and conservative with respect to $T$ if and only if 
for any $A\in \mathcal{F}$ with $\mu (A)>0,$  
$$
\mu (\{ z\in X\mid \sum _{n=0}^{\infty }1_{A}\circ T^{n}(z)<+\infty \} )=0.
$$ 
Finally, the measure $\mu$ is said to be
$T$-invariant\index{invariant measure}, or $T$ is said to preserve
the measure $\mu$ if and only if $\mu\circ T^{-1}=\mu$. It follows
from Birkhoff's Ergodic Theorem that every finite ergodic
$T$-invariant measure $\mu$ is conservative,  for infinite measures
this is no longer true. Finally, two  ergodic invariant measures
defined on the same $\sigma$-algebra are  either singular or they
coincide up  to a multiplicative constant.

\

\bdfn\label{d1h65}
Suppose that $(X,\Fa,\nu)$ is a probability space and $T:X\to X$ is a measurable map
such that $T(A)\in \Fa$ whenever $A\in\Fa$. The map $T:X\to X$ is said to be weakly
metrically exact provided that $\limsup_{n\to\infty}\mu(T^n(A))=1$ whenever $A\in\Fa$ and 
$\mu(A)>0$.
\edfn

\

\ni We need the following two facts about weak metrical exactness, 
the first being straightforward (see the argument in \cite[page 15]{Aa}),
the latter more involved (see \cite{PUbook}).

\

\bfact\label{f2h65}
If a nonsingular measurable transformation $T:X\to X$ of a probability space $(X,\cF,\nu)$ is
weakly metrically exact, then it is ergodic and conservative.
\efact

\

\bfact\label{f3h65}
A measure-preserving transformation $T:X\to X$ of a probability space $(X,\Fa,\mu)$ is
weakly metrically exact if and only if it is exact, which means that $\lim_{n\to\infty}
\mu(T^n(A))=1$ whenever $A\in\cF$ and $\mu(A)>0$, or equivalently, the $\sg$-algebra
$\bi_{n\ge 0}T^{-n}(\cF)$ consists of sets of measure $0$ and $1$ only. 
Note that if $T:X\rightarrow X$ is exact, then the
Rokhlin's natural extension $(\^T,\^X,\^\mu)$ of $(T,X,\mu)$ is K-mixing.
\efact

\

\ni The precise formulation of our main result in this section is the following.

\

\bthm\label{t4h65}
$\^m_h$ is a unique $h$-conformal measure for the map $\tf:J(\tf)\to J(\tf)$. There exists 
a unique Borel probability $\tf$-invariant measure $\^\mu_h$ on $J(\tf)$ which is absolutely 
continuous with respect to $\^m_h$. The measure $\^\mu_h$ is metrically exact and 
equivalent with $\^m_h$. 
\ethm

\

\ni The proof of this theorem will consist of several steps. We start with the 
following.

\

\blem\label{l1h64}
Every $h$-conformal measure $\nu$ for $\tf:J(\tf)\to J(\tf)$ is equivalent to $\^m_h$.
\elem

{\sl Proof.} Fix an integer $v\ge 1$ and let 
$$
I_v=\{(\tau,z)\in J(\tf)\sms \Sing(\tf):\eta(\tau,z)\ge 1/v\},
$$
where $\eta(\tau,z)>0$ is the number produced in Proposition~\ref{pu6.1h23}. 
We may assume that $\eta (\tau ,z)\leq 1.$ Let
also $(\hat\tau,\hat z)$ and $(n_j)_1^\infty$ be the objects produced in this
proposition. Fix $(\tau,z)\in I_v$. Disregarding finitely many values
of $j$, we may assume without loss of generality that
$$
|f_{\tau|_{n_j}}(z)-\hat z|<{1\over 4}\eta(\tau,z).
$$
Let
\beq\label{2h67}
\aligned
&B_j(\tau,z)=
 [\tau|_{n_j}]\times f_{\tau|_{n_j},z}^{-1}\(B\(f_{\tau|_{n_j}}(z),
 {1\over 2}\eta(\tau,z)\)\) \\ 
&\text{and} \\
&r_j(\tau,z)={1\over 2}K\eta(\tau,z)|f_{\tau|_{n_j}}'(z)|^{-1}.
\endaligned
\eeq
By Koebe's Distortion Theorem and Proposition~\ref{pu6.1h23} we get that,
\beq\label{1h67}
\aligned
\nu (B_j(\tau,z))
&=   \nu \(\tf_{\tau |_{n_{j}},z}^{-n_j}\(\Sg_u\times B\(f_{\tau|_{n_j}}(z), {1\over 2}\eta(\tau,z)\)\)\) \\
&\ge K^{-h}|f_{\tau|_{n_j}}'(z)|^{-h}\nu\(\Sg_u\times B\(f_{\tau|_{n_j}}(z), {1\over 2}\eta(\tau,z)\)\) \\
&=   K^{-h}|f_{\tau|_{n_j}}'(z)|^{-h}\nu\circ p_2^{-1}\(B\(f_{\tau|_{n_j}}(z), {1\over 2}\eta(\tau,z)\)\) \\
&\ge M_{(2v)^{-1},\nu }K^{-h}|f_{\tau|_{n_j}}'(z)|^{-h} \\
&\ge M_{(2v)^{-1},\nu}(2K^{-1}\eta^{-1}(\tau,z))^hr_j^h(\tau,z) \\
&\ge 2^hM_{(2v)^{-1},\nu}K^{-h}r_j^h(\tau,z), 
\endaligned
\eeq
where $M_{R,v}:= \inf\{ \nu \circ p_{2}^{-1}(B(w,R))\mid w\in J(G)\} >0.$ 
Now fix $E$, an arbitrary Borel set contained in $I_v$. Fix also $\varepsilon>0$. Since the 
measure $\nu$ is regular, by Theorem~\ref{t1h4} there exists $j(\tau,z)\ge 1$ such that, with $B(\tau,z)
=B_{j(\tau,z)}(\tau,z)$ and $r(\tau,z)=r_{j(\tau,z)}(\tau,z)$, we have
\beq\label{1h69}
\nu\lt(\bu_{(\tau,z)\in E}B(\tau,z)\rt)\le \nu(E)+\e.
\eeq
By the $4r$-Covering Theorem (\cite{MS}), there exists a countable set $\hat E\sbt E$ such that
the balls $\{B(z,r(\tau,z)):(\tau,z)\in\hat E\}$ are mutually disjoint and
$$
\bu_{(\tau,z)\in\hat E}B(z,4r(\tau,z))\spt \bu_{(\tau,z)\in E}B(z,r(\tau,z))\spt p_2(E).
$$
Hence, by Theorem~\ref{t1h61} and (\ref{1h67}), we get
\beq\label{2h69}
\aligned
\^m_h(E)
&\le \^m_h(p_2^{-1}(p_2(E)))
 \le \sum_{(\tau,z)\in\hat E}\^m_h\circ p_2^{-1}\(B(z,4r(\tau,z))\) \\
&=        \sum_{(\tau,z)\in\hat E}m_h\(B(z,4r(\tau,z))\) \\
&\le C4^h \sum_{(\tau,z)\in\hat E}r^h(\tau,z) \\
&\le C(2K)^hM_{(2v)^{-1},\nu}^{-1}\sum_{(\tau,z)\in\hat E}\nu(B(\tau,z)).
\endaligned
\eeq
Now, since the sets $\{B(z,r(\tau,z)):(\tau,z)\in\hat E\}$ are mutually disjoint and
since $$B(\tau,z)\sbt p_2^{-1}(B(z,r(\tau,z))),$$ so are disjoint the sets
$\{B(\tau,z):(\tau,z)\in\hat E\}$. Thus, using (\ref{1h69}), we get
\beq\label{3h69}
\^m_h(E)
\le C(2K)^hM_{(2v)^{-1},\nu}^{-1}\nu\lt(\bu_{(\tau,z)\in\hat E}B(\tau,z)\rt)
\le C(2K)^hM_{(2v)^{-1},\nu}^{-1}(\nu(E)+\e). 
\eeq
Letting $\e\downto 0$ we thus get
$$
\^m_h(E)\le C(2K)^hM_{(2v)^{-1},\nu}^{-1}\nu(E).
$$
Consequently $\^m_h|_{I_v}\abs \nu|_{I_v}$. Since, in virtue of Proposition~\ref{pu6.1h23},
$J(\tf)\sms\Sing(\tf)=\bu_{v=1}^{\infty }I_v$, we get that
\beq\label{1absh69}
\^m_h|_{J(\tf)\sms\Sing(\tf)} \abs  \nu|_{J(\tf)\sms\Sing(\tf)}.
\eeq
Now, suppose that $\nu(\Sing(\tf))>0$. Since $\tf'$ vanishes on $\Crit(\tf)$, the measure 
$$\nu_0=(\nu(\Sing(\tf)))^{-1}\nu|_{\Sing(\tf)},$$ is $h$-conformal for $\tf:J(\tf)\to J(\tf)$. But
then (\ref{1absh69}) would be true with $\nu$ replaced by $\nu_0$. We would thus have 
$\^m_h(J(\tf)\sms\Sing(\tf))=0$. Since, by Corollary~\ref{c3h61}, $\^m_h(\Sing(\tf))=0$,
we would get $\^m_h(J(\tf))=0$. This contradiction shows that $\nu(\Sing(\tf))=0$. 
Consequently,
\beq\label{2absh69}
\^m_h\abs \nu.
\eeq
Seeking contradiction, suppose that $\nu$ is not absolutely continuous with respect 
to $\^m_h$. Then, there exists a Borel set $X\sbt J(\tf)\sms\cup_{n=0}^\infty \tf^n(\Sing(\tf))$ 
such that $\^m_h(X)=0$ but $\nu(X)>0$. But then the measure $\nu$ restricted to the forward 
and backward invariant set $\bu_{n,m\in\N}\tf^{-m}(\tf^n(X))$ and multiplied by the reciprocal of 
$\nu\(\bu_{n,m\in\N}\tf^{-m}(\tf^n(X))\)$ is $h$-conformal for $\tf:J(\tf)\to J(\tf)$. But, by conformality of
$\^m_h$, and as $X\sbt J(\tf)\sms\cup_{n=0}^\infty \tf^n(\Sing(\tf))$), we conclude from
$\^m_h(X)=0$ that $\^m_h\(\bu_{n,m\in\N}\tf^{-m}(\tf^n(X))\)=0$. Since, by (\ref{2absh69}), the
measure $\^m_h$ is absolutely continuous with respect to $\nu $ restricted to 
$\bu_{n,m\in\N}\tf^{-m}(\tf^n(X))$, we finally get that $\^m_h(J(\tf))=0$. This contradiction
show that $\nu\abs \^m_h$. Together with (\ref{2absh69}) this gives that $\nu$ and 
$\^m_h$ are equivalent. We are done. \endpf

\

\fr Combining inequalities (\ref{2h69}) and (\ref{3h69}) (with $\nu=\^m_h$) from 
the proof of Lemma~\ref{l1h64}, and letting $\e\downto 0$ in (\ref{3h69}), we get 
for every Borel set $E\sbt I_v$, $v\ge 1$, such that $p_{2} (E)$ is measurable, that
$$
m_h(p_2(E))\le C(2K)^hM_{(2v)^{-1}}^{-1}\^m_h(E).
$$
Consequently, as $J(\tf)\sms\Sing(\tf)=\bu_{v=1}^{\infty }I_v$ and
$\^m_h(\Sing(\tf))=0$, we get the following.

\

\blem\label{l1h68}
If $E$ is a Borel subset of $J(\tf)$ such that $p_{2}(E)$ is
measurable and $\^m_h(E)=0$, then $m_h(p_2(E))=0$. So, by  
Lemma~\ref{l1h64}, for any $h$-conformal measure $\nu$ for $\tf:J(\tf)\to J(\tf)$, we have 
that $\nu\circ p_2^{-1}(p_2(E))=0$ whenever $\nu(E)=0$.
\elem

\

\fr We now shall recall the concept of Vitali relations defined on the
page 151 of Federer's book \cite{federer}. Let $X$ be an arbitrary
set. By a covering relation on $X$ one means a subset of 
$$
\{(x,S):x\in S\sbt X\}.
$$
If $C$ is a covering relation on $X$ and $Z\sbt X$, one puts
$$
C(Z)=\{S\sbt X:(x,S)\in C \  \text{ for some } \ x\in Z\}.
$$
One then says that $C$ is fine at $x$ if 
$$
\inf\{\diam(S):(x,S)\in C\} = 0.
$$
If in addition $X$ is a metric space and a Borel measure $\mu$ is given
on $X$, then a covering relation $V$ on $X$ is called a Vitali relation
if 
\begin{itemize}
\item[(a)] All elements of $V(X)$ are Borel sets.
\item[(b)] $V$ is fine at each point of $X$
\item[(c)] If $C\sbt V$, $Z\sbt X$ and $C$ is fine at each point of
  $Z$, then there exists a countable disjoint subfamily $\Fa$ of $C(Z)$
  such that $\mu(Z\sms \cup\,\Fa)=0$.
\end{itemize}

\

\fr Now, given $(\tau,z)\in J(\tf)\sms \Sing(\tf)$, let 
$$
\Ba_{(\tau,z)}=\lt\{\((\tau,z),\ov B_j(\tau,z)\)\rt\}_{j=1}^\infty,
$$
where the sets $B_j(\tau,z)$ are defined by formula (\ref{2h67}). Let
$$
\Ba=\bu_{(\tau,z)\in J(\tf)\sms \Sing(\tf)}\Ba_{(\tau,z)}
$$
and, following notation from Federer's book \cite{federer}, let
$$
\Ba_2:=\Ba(J(\tf)\sms \Sing(\tf))=\{\ov B_j(\tau,z):(\tau,z)\in
J(\tf)\sms \Sing(\tf),\, j\ge1\}. 
$$
We shall prove the following.

\

\blem\label{l1h71}
The family $\Ba_2$ is a Vitali relation for the measure $\^m_h$ on the
set $J(\tf)\sms \Sing(\tf)$. 
\elem

{\sl Proof.} Fix $(\tau,z)\in J(\tf)\sms \Sing(\tf)$. Since $p_{2}(\ov B_j(\tau,z))\sbt 
\ov B(z,r_j(\tau,z))$ and since
\beq\label{1h71}
\lim_{j\to\infty}r_j(\tau,z)=0,
\eeq
we have
$$
\lim_{j\to\infty}\diam(\ov B_j(\tau,z))=0.
$$
This means that the relation $\Ba$ is fine at the point $(\tau,z)$. Aiming to apply 
Theorem~2.8.17 from \cite{federer}, we set
$$
\d((\ov B_j(\om,x)))=r_j(\om,x)
$$
for every $\ov B_j(\om,x)\in \Ba_2$. Fix $1<\ka<+\infty$ (a different notation for 
$1<\tau<+\infty$ appearing in Theorem~2.8.17 from \cite{federer}). With the notation 
from page 144 in \cite{federer} we have
$$
\hat{\ov B_j}(\tau,z)
=    \bu\{B:B\in\Ba_2,\, B\cap\ov B_j(\tau,z)\ne\es,\,\d(B)\le\ka\d(\ov B_j(\tau,z))\}
\sbt p_{2}^{-1}\left( (\ov B\(z,(1+2\ka)r_j(\tau,z)\)\right).
$$
So, in virtue of Theorem~\ref{t1h61} and (\ref{1h67}), we obtain
$$
\d(\ov B_j(\tau,z))+{\^m_{h}(\hat{\ov B_j}(\tau,z))\over \^m_{h}(\ov B_j(\tau,z))}
\le r_j(\tau,z)+{C\((1+2\ka)r_j(\tau,z)\)^h\over C^{-1}r_j^h(\tau,z)}
=C^2(1+2\ka)^h+r_j(\tau,z),
$$
where $C>0$ is a constant independent of $j.$ 
Hence, using (\ref{1h71}), we get
$$
\lim_{j\to\infty}\lt(\d(\ov B_j(\tau,z))+{\^m_{h}(\hat{\ov B_j}(\tau,z))\over \^m_{h}(\ov B_j(\tau,z))}\rt)
\le C^2(1+2\ka)^h<+\infty.
$$
Thus, all the hypothesis of Theorem~2.8.17 in \cite{federer}, p. 151 are verified 
and the proof of our lemma is complete. \endpf

\

\fr As an immediate consequence of this lemma and Theorem~2.9.11, p. 158 in \cite{federer}
we get the following.

\

\bprop\label{p1h73}
For every Borel set $A\sbt J(\tf)\sms \Sing(\tf)$ let
$$
A_h=\lt\{(\tau,z)\in A:\lim_{j\to\infty}\frac{\^m_{h}(A\cap \ov B_j(\tau,z))}{\^m_{h}(\ov B_j(\tau,z))}=1\rt\}.
$$
Then $\^m_{h}(A_h)=\^m_{h}(A)$.
\eprop

\

\fr Now, we shall prove the following. 

\

\blem\label{l2h73}
The measure $\^m_{h}$ is weakly metrically exact for the map $\tf:J(\tf)\to J(\tf)$. In particular
it is ergodic and conservative.
\elem

{\sl Proof.} Fix a Borel set $F\sbt J(\tf)\sms \Sing(\tf)$ with $\^m_{h}(F)>0$. By Proposition~\ref{p1h73}
there exists at least one point $(\tau,z)\in F_h$. Our first goal is to show that
\beq\label{1h73}
\lim_{j\to\infty}{\^m_{h}(\tf^{n_j}(F)\cap p_2^{-1}(\ov B(f_{\tau|_{n_j}}(z),\eta/2))\) \over
     \^m_{h}\(p_2^{-1}(\ov B(f_{\tau|_{n_j}}(z),\eta/2))\)}
     =1,
\eeq
where, we recall $\eta=\eta(\tau,z)>0$ is the number produced in Proposition~\ref{pu6.1h23} and
$(n_j)_1^\infty$ is the corresponding sequence produced there. Indeed, suppose for the contrary
that
$$
\ka={1\over 2}\liminf_{j\to\infty}{\^m_{h}\(p_2^{-1}(\ov B(f_{\tau|_{n_j}}(z),\eta/2))\sms \tf^{n_j}(F)\) \over
    \mh\(p_2^{-1}(\ov B(f_{\tau|_{n_j}}(z),\eta/2))\)}
    >0.
$$
Then, disregarding finitely many $n$s we may assume that
$$
{\mh\(p_2^{-1}(\ov B(f_{\tau|_{n_j}}(z),\eta/2))\sms \tf^{n_j}(F)\) \over
    \mh\(p_2^{-1}(\ov B(f_{\tau|_{n_j}}(z),\eta/2))\)}
\ge\ka>0
$$
for all $j\ge 1$. But
$$
\tf_{\tau |_{n_{j}},z}^{-n_j}\(p_2^{-1}(\ov B(f_{\tau|_{n_j}}(z),\eta/2))\sms \tf^{n_j}(F)\)
\sbt \([\tau|_{n_j}]\times \ov B(z,{1\over 2}K\eta\lt|f_{\tau|_{n_j}}'(z)\rt|^{-1})\)\sms F
=    \ov B_j(\tau,z)\sms F
$$
and
$$
\aligned
\mh\(\tf_{\tau |_{n_{j}},z}^{-n_j}\(p_2^{-1}(\ov B(f_{\tau|_{n_j}}(z)&,\eta/2))\sms \tilde{f}^{n_j}(F)\)\) \ge \\
&\ge K^{-h}\lt|f_{\tau|_{n_j}}'(z)\rt|^{-h}\mh\(p_2^{-1}(\ov B(f_{\tau|_{n_j}}(z),\eta/2))\sms \tilde{f}^{n_j}(F)\) \\
&\ge \ka K^{-h}\lt|f_{\tau|_{n_j}}'(z)\rt|^{-h}\mh\(p_2^{-1}(\ov B(f_{\tau|_{n_j}}(z),\eta/2))\) \\
&=   \ka K^{-h}\lt|f_{\tau|_{n_j}}'(z)\rt|^{-h}m_h\(\ov B(f_{\tau|_{n_j}}(z),\eta/2)\) \\
&\ge \ka K^{-h}M_{\eta/2}\lt|f_{\tau|_{n_j}}'(z)\rt|^{-h}.
\endaligned
$$
Hence, making use of Theorem~\ref{t1h61}, we obtain
$$
\aligned
\mh(\ov B_j(\tau,z)\sms F)
&\ge \ka K^{-h}M_{\eta/2}\lt|f_{\tau|_{n_j}}'(z)\rt|^{-h} \\
&=   \ka(K^2\eta/2)^{-h}M_{\eta/2}r_j^h(\tau,z)              \\        
&\ge C^{-1}(K^2\eta/2)^{-h}M_{\eta/2}\mh(\ov B_j(\tau,z)).
\endaligned
$$
Thus,
$$
{\mh(\ov B_j(\tau,z)\sms F)\over \mh(\ov B_j(\tau,z))}
\ge C^{-1}(K^2\eta/2)^{-h}M_{\eta/2}>0.
$$
Letting $j\to\infty$ this contradicts the fact that $(\tau,z)\in F_h$ and finishes the proof
of (\ref{1h73}). Now since $\tf:J(\tf)\to J(\tf)$ is topologically exact, there exists $q\ge 0$
such that $\tf^q(p_2^{-1}(B(w,\eta/2)))\spt J(\tf)$ for all $w\in J(G)$. It then easily
follows from (\ref{1h73}) and conformality of $\mh$ that
$$
\limsup_{k\to\infty} \mh(\tf^k(F))
\ge \limsup_{j\to\infty} \mh(\tf^{q+n_j})(F))
=1.
$$
Noting also that $\mh(\Sing(\tf))=0$ (by Corollary~\ref{c3h61}), the weak metric 
exactness of $\mh$ is proved. Ergodicity and conservativity follow then from 
Fact~\ref{f2h65}. We are done. \endpf

\

\bcor\label{c1h77}
$\mh$ is the only $h$-conformal measure on $J(\tf)$ for the map $\tf:J(\tf)\to J(\tf)$. 
\ecor

{\sl Proof.} Let $\nu$ be an arbitrary $h$-conformal measure on $J(\tf)$ for the 
map $\tf:J(\tf)\to J(\tf)$. Since, by Lemma~\ref{l1h64} the measure $\nu$ is absolutely 
continuous with respect $\mh$, it follows from Theorem~2.9.7 in \cite{federer}, p. 155
and Lemma~\ref{l1h71} that for $\mh$-a.e. $(\tau,z)\in J(\tf)\sms \Sing(\tf)$,
$$
\aligned
{d\nu\over d\mh}(\tf(\tau,z))
& =  \lim_{j\to\infty}{\nu(\ov B_j(\tf(\tau,z))) \over \mh(\ov B_j(\tf(\tau,z)))}
=\lim_{j\rightarrow \infty }\frac{\nu (\tf (\overline{B}_{j}(\tau ,z)))}
{\mh (\tf (\overline{B}_{j}(\tau ,z)))}\\ 
& = \lim _{j\rightarrow \infty }\frac{\int _{\ov{B}_{j}(\tau ,z)}|\tf'|^{h}\ d\nu }
{\int _{\ov{B}_{j}(\tau ,z)} |\tf'|^{h}\ d\mh} 
=\lim _{j\rightarrow \infty }\frac{\nu (\ov{B}_{j}(\tau, z))}{\mh (\ov{B}_{j}(\tau ,z))}
={d\nu\over d\mh}(\tau,z).
\endaligned
$$
Since, by Lemma~\ref{l2h73}, the measure $\mh$ is ergodic, it follows that the 
Radon-Nikodym derivative ${d\nu\over d\mh}$ is $\mh$-almost everywhere constant.
Since $\nu$ and $\mh$ are equivalent (by Lemma~\ref{l1h64}) this derivative must
be almost everywhere, with respect to $\mh$ as well as $\nu$, equal to $1$. Thus
$\nu=\mh$ and we are done. \endpf

\

\fr In order to prove the existence of a Borel probability $\tf$-invariant measure on 
$J(\tf)$ equivalent to $\mh$, we will use Marco-Martens method originated in \cite{martens}.
This means that we shall first produce a $\sg$-finite $\tf$-invariant measure equivalent
to $\mh$ (this is the Marco-Martens method) and then we will prove this measure to be
finite. The heart of the Martens' method is the following theorem which is a generalization
of Proposition 2.6 from \cite{martens}. It is a generalization in the sense that we do not
assume our probability space $(X,\mathcal{B},m )$ below to be a $\sg$-compact metric space,
neither assume we that our map is conservative, instead, we merely assume that item (6)
in Definition~\ref{d:mmmap} holds. Also, the proof we provide below is based on the 
concept of Banach limits rather than (see \cite{martens}) on the notion of weak limits.

\

\bdfn
\label{d:mmmap}
Suppose $(X,\mathcal{B},m )$ is a probability space. 
Suppose $T:X\rightarrow X$ is a measurable mapping, such that 
$T(A)\in \mathcal{B}$ whenever $A\in \mathcal{B}$, and such that 
the measure $m$ is quasi-invariant with respect to $T$, 
meaning that $m\circ T^{-1}\prec m.$ Suppose further that 
there exists a countable family $\{ X_n\} _{n=0}^{\infty }$ 
of subsets of $X$ with the following properties.
\begin{itemize}
\item[(1)]
For all $n\geq 0$, $X_{n}\in \mathcal{B}.$ 
\item[(2)]
$m(X\setminus \bigcup _{n=0}^{\infty }X_{n})=0.$ 
\item[(3)] 
For all $m,n\geq 0$, there exists a $j\geq 0$ such that 
$m(X_{m}\cap T^{-j}(X_{n}))>0.$ 
\item[(4)]
For all $j\geq 0$ there exists a $K_{j}\geq 1$ such that 
for all $A,B\in \mathcal{B}$ with $A,B\subset X_{j}$ and for all 
$n\geq 0$, 
$$m(T^{-n}(A))m(B)\leq K_{j}m(A)m(T^{-n}(B)).$$ 
\item[(5)] 
$\sum _{n=0}^{\infty }m(T^{-n}(X_{0}))=+\infty .$
\item[(6)] 
$\lim _{l\rightarrow \infty }m(T(\bigcup _{j=l}^{\infty }Y_{j}))=0$, 
where $Y_{j}:= X_{j}\setminus \bigcup _{i<j} X_{i}.$  
\end{itemize}
Then the map $T:X\rightarrow X$ is called a Marco-Martens map and 
$\{ X_{j}\} _{j=0}^{\infty }$ is called a Marco-Martens cover. 
\edfn 

\brem
\label{r:mmmapf1}
Note that {\em (6)} is satisfied if the map $T:X\rightarrow X$ is 
finite-to-one. For, if $T$ is finite-to-one, 
then $\bigcap _{l=1}^{\infty }T(\bigcup _{j=l}^{\infty }Y_{j})=\emptyset .$ 

\erem
\

\bthm\label{t1h75} 
Let $(X,\mathcal{B},m)$ be a probability  space and 
let $T:X\rightarrow X$ be a Marco-Martens map with 
a Marco-Martens cover $\{ X_{j}\} _{j=0}^{\infty }.$ 
Then, there exists a $\sigma$-finite $T$-invariant 
measure $\mu $ on $X$ equivalent to $m.$ 
In addition, $0<\mu (X_{j})<+\infty $ for each $j\geq 0.$ 
The measure $\mu $ is constructed in the following way: 
Let $l_{B}:l_{\infty }\rightarrow \R$ 
be a Banach limit and let 
$Y_{j}:=X_{j}\setminus \bigcup _{i<j}X_{i}$ for each $j\geq 0.$  
For each $A\in \mathcal{B}$, set 
$$m_{n}(A):=\frac{\sum _{k=0}^{n}m(T^{-k}(A))}{\sum _{k=0}^{n}m(T^{-k}(X_{0}))} .$$
If $A\in \mathcal{B}$ and $A\subset Y_{j}$ with some $j\geq 0$, 
then we obtain $(m_{n}(A))_{n=1}^{\infty }\in l_{\infty }.$ 
We set 
$$\mu (A):= l_{B}((m_{n}(A))_{n=1}^{\infty }).$$ 
For a general measurable subset $A\subset X$, set 
$$\mu (A):=\sum _{j=0}^{\infty }\mu (A\cap Y_{j}).$$ 
In addition, if for a measurable subset $A\subset X$,  
the sequence $(m_{n}(A))_{n=1}^{\infty }$ is bounded, 
then we have the following formula. 
\begin{equation}
\label{eq:muequ}
\mu (A)=l_{B}((m_{n}(A))_{n=1}^{\infty })-\lim _{l\rightarrow \infty }
l_{B}((m_{n}(A\cap \bigcup _{j=l}^{\infty }Y_{j}))_{n=0}^{\infty}).
\end{equation}
Furthermore, 
if the transformation $T:X\rightarrow X$ is ergodic (equivalently with respect to 
the measure $m$ or $\mu $), 
then the $T$-invariant measure $\mu $ is unique up to a multiplicative 
constant. 
\ethm

In order to prove Theorem~\ref{t1h75}, 
we need several lemmas.

\begin{lem}
\label{l:dumeasure}
If $(Z,\mathcal{F})$ is a $\sigma $-algebra of sets, $Z=\bigcup _{j=0}^{\infty }Z_{j}$ 
is a disjoint union of measurable sets (elements of $\mathcal{F})$ ,
and for each $j\geq 0$,  
$\nu _{j}$ is a finite measure on $Z_{j}$, 
then the function 
$A\mapsto \nu (A):=\sum _{j=0}^{\infty }\nu _{j}(A\cap Z_{j})$, is a
$\sigma $-finite measure on $Z.$  
\end{lem}
\begin{proof}
Let $A\in \mathcal{F}$ and let $(A_{n})_{n=1}^{\infty }$ be a partition of $A$ into sets in 
$\mathcal{F}$. Then 
\begin{align*}
\nu (A) = & \sum _{j=0}^{\infty }\nu _{j}(A\cap Z_{j})=\sum
_{j=0}^{\infty }\nu _{j}(\bigcup _{n=1}^{\infty }(A_{n}\cap Z_{j}))\\  
& = \sum _{j=0}^{\infty }\sum _{n=1}^{\infty }\nu _{j}(A_{n}\cap Z_{j})
=\sum _{n=1}^{\infty }\sum _{j=0}^{\infty }\nu _{j}(A_{n}\cap Z_{j})=\sum _{n=1}^{\infty }\nu (A_{n}),
\end{align*}
where we could have changed the order of summation since all terms
involved were non-negative.  
Thus, we have completed the proof of our lemma.
\end{proof}
We now suppose that we have the assumption of Theorem~\ref{t1h75}. 
\begin{lem}
\label{l:muxjfin}
For every $j\geq 0$, 
the sequence $(m_{n}(X_{j}))_{n=1}^{\infty }$ is bounded and 
$\mu (Y_{j})\leq \mu (X_{j})<+\infty .$ 
\end{lem}
\begin{proof}
In virtue of (3) of Definition~\ref{d:mmmap} there exists a $q\geq 0$ such that 
$m(X_{j}\cap T^{-q}(X_{0}))>0.$ 
By (4) of Definition~\ref{d:mmmap}, 
we have for all $n\geq 0$ that 
\begin{align*}
m_{n}(Y_{j})\leq m_{n}(X_{j})\leq & 
K_{j}\frac{m(X_{j})}{m(X_{j}\cap T^{-q}(X_{0}))}m_{n}(X_{j}\cap T^{-q}(X_{0}))\\ 
\leq & K_{j}\frac{m(X_{j})}{m(X_{j}\cap T^{-q}(X_{0}))}m_{n+q}(X_{0})
\frac{\sum _{k=0}^{n+q}m(T^{-k}(X_{0}))}{\sum _{k=0}^{n}m(T^{-k}(X_{0}))}\\ 
= & K_{j}\frac{m(X_{j})}{m(X_{j}\cap T^{-q}(X_{0}))}
\left( 1+\frac{\sum _{k=n+1}^{n+q}m(T^{-k}(X_{0}))}{\sum _{k=0}^{n}m(T^{-k}(X_{0}))}\right) \\ 
\leq & K_{j}\frac{m(X_{j})}{m(X_{j}\cap T^{-q}(X_{0}))}
\left( 1+\frac{q}{\sum _{k=0}^{n}m(T^{-k}(X_{0}))}\right) .
\end{align*}
It follows from (5) of Definition~\ref{d:mmmap} that 
$(m_{n}(X_{j}))_{n=1}^{\infty }\in l_{\infty }$ and 
$$\mu (Y_{j})\leq K_{j}m(X_{j})/m(X_{j}\cap T^{-q}(X_{0}))<\infty .$$ 
Since $X_{j}=\bigcup _{i=0}^{j}Y_{i}$, we are therefore done. 
\end{proof}

Now, for every $j\geq 0$, set 
$\mu _{j}:= \mu |_{Y_{j}}.$ 
\begin{lem}
\label{l:mumcomp}
For every $j\geq 0$ such that $\mu (Y_{j})>0$, 
and for every measurable set $A\subset Y_{j}$, we have 
$$K_{j}^{-1}\frac{\mu (Y_{j})}{m(Y_{j})}m(A)\leq \mu _{j}(A)\leq K_{j}\frac{\mu (Y_{j})}{m(Y_{j})}m(A).$$
\end{lem}
\begin{proof}
This is an immediate consequence of (4) of Definition~\ref{d:mmmap} and the 
definition of the measure $\mu .$  
\end{proof}
\begin{lem}
\label{l:mujcam}
For any $j\geq 0$, $\mu _{j}$ is a (countably additive) measure on $Y_{j}.$ 
\end{lem}
\begin{proof}
Let $j\geq 0.$ 
We may assume without loss of generality that $\mu _{j}(Y_{j})>0.$ 
Let $A\subset Y_{j}$ be a measurable set and let 
$(A_{k})_{k=1}^{\infty }$ be a countable partition of $A$ into measurable sets. 
For every $n\geq 1$ and for every $l\geq 1$, we have 
\begin{multline}
\label{eq:1rds103}
\left(\sum _{k=1}^{\infty }m_{n}(A_{k}) \right) _{n=1}^{\infty }
- \sum _{k=1}^{l} \left(m_{n}(A_{k})\right) _{n=1}^{\infty }
=\left(\sum _{k=1}^{\infty }m_{n}(A_{k})\right)_{n=1}^{\infty }
-\left(\sum _{k=1}^{l}m_{n}(A_{k})\right)_{n=1}^{\infty }\\ 
=\left(\sum _{k=l+1}^{\infty }m_{n}(A_{k})\right) _{n=1}^{\infty }.
 \ \ \ \ \ \ \ \ \ \ \ \ \ \ \ \ \ \ \ \ \ \ \ \  
\end{multline}
It therefore follows from (4) of Definition~\ref{d:mmmap} 
that 
\begin{align*}
\left\| \left( \sum _{k=1}^{\infty }m_{n}(A_{k})\right) _{n=1}^{\infty }-\sum _{k=1}^{l}\left( m_{n}(A_{k})\right) _{n=1}^{\infty }\right\| _{\infty }
= & \left\| \left( \sum _{k=l+1}^{\infty }m_{n}(A_{k})\right) _{n=1}^{\infty }\right\| \\ 
\leq & \left\| \frac{K_{j}}{m(Y_{j})} \left( m_{n}(Y_{j})\sum _{k=l+1}^{\infty }m(A_{k})\right) _{n=1}^{\infty }\right\| _{\infty }\\ 
= & \frac{K_{j}}{m(Y_{j})}\left\| \left( m_{n}(Y_{j})\sum _{k=l+1}^{\infty }m(A_{k})\right) _{n=1}^{\infty }\right\| _{\infty }.
\end{align*}
Since, by Lemma~\ref{l:muxjfin}, $(m_{n}(Y_{j}))_{n=1}^{\infty }\in l_{\infty }$, and 
since $\lim _{l\rightarrow \infty }\sum _{k=l+1}^{\infty }m(A_{k})=0$, 
we conclude that 
$\lim _{l\rightarrow \infty }\| (\sum _{k=1}^{\infty }m_{n}(A_{k}))_{n=1}^{\infty }
-\sum _{k=1}^{l}(m_{n}(A_{k}))_{n=1}^{\infty }\| _{\infty }=0.$ 
This means that in the Banach space $l_{\infty }$, we have 
$(\sum _{k=1}^{\infty }m_{n}(A_{k}))_{n=1}^{\infty }=\sum
_{k=1}^{\infty }(m_{n}(A_{k}))_{n=1}^{\infty }.$  
Hence, using continuity of the Banach limit $l_{B}:l_{\infty }\rightarrow\R$, 
we get, 
\begin{align*}
\mu (A) =& l_{B}((m_{n}(A))_{n=1}^{\infty })=l_{B}((m_{n}(\bigcup
_{k=1}^{\infty }A_{k}))_{n=1}^{\infty }) 
=l_{B}((\sum _{k=1}^{\infty }m_{n}(A_{k}))_{k=1}^{\infty })\\ 
= & \sum _{k=1}^{\infty }l_{B}((m_{n}(A_{k}))_{n=1}^{\infty })=\sum _{k=1}^{\infty }\mu (A_{k}).
\end{align*}
We are done. 
\end{proof}
Combining Lemmas ~\ref{l:dumeasure}, \ref{l:muxjfin}, \ref{l:mumcomp}, and  
 \ref{l:mujcam}, and (3) of Definition~\ref{d:mmmap}, we get the following.
\begin{lem}
\label{l:musfmeas}
$\mu $ is a $\sigma $-finite measure on $X$ equivalent to $m$. 
Moreover,  $\mu (Y_{j})\leq \mu (X_{j})<\infty $ and 
$0<\mu (X_{j})$ 
for all $j\geq 0.$ 
\end{lem} 
\begin{lem}
\label{l:mueqpf}
The formula \ref{eq:muequ} holds. 
\end{lem}
\begin{proof}
Fix a measurable set $A\subset X.$ Then, for every $l\geq 1$ we have that 
\begin{align*} 
l_{B}((m_{n}(A))_{n=1}^{\infty })= 
& l_{B}\left(\sum _{j=0}^{l}(m_{n}(A\cap Y_{j}))_{n=1}^{\infty }\right)
+l_{B}\left( (m_{n}(\bigcup _{j=l+1}^{\infty }A\cap Y_{j}))_{n=1}^{\infty }\right) \\ 
& = \sum _{j=0}^{l}l_{B}((m_{n}(A\cap Y_{j}))_{n=1}^{\infty })
+l_{B}\left((m_{n}(A\cap \bigcup _{j=l+1}^{\infty }Y_{j}))_{n=1}^{\infty } \right).
\end{align*}
Hence, letting $l\rightarrow \infty $, we get 
\begin{align*}
l_{B}((m_{n}(A))_{n=1}^{\infty })= 
& \sum _{j=0}^{\infty }l_{B}((m_{n}(A\cap Y_{j}))_{n=1}^{\infty })
+\lim _{l\rightarrow \infty }l_{B}\left( (m_{n}(A\cap \bigcup _{j=l+1}^{\infty }Y_{j}))_{n=1}^{\infty }\right) \\ 
& = \mu (A)+\lim _{l\rightarrow \infty }l_{B}\left( (m_{n}(A\cap \bigcup _{j=l}^{\infty }Y_{j}))_{n=1}^{\infty }\right) . 
\end{align*}
We are done. 
\end{proof}
\begin{lem}
\label{l:mutinv}
The $\sigma$-finite measure $\mu $ is $T$-invariant. 
\end{lem} 
\begin{proof}
Let $i\geq 0$ be such that $m(Y_{i})>0$. Fix a measurable set $A\subset Y_{i}.$ 
Fix $l\geq 1.$ We then have 
\begin{align*}
m_{n}(T^{-1}(A)\cap \bigcup _{j=l}^{\infty }Y_{j})
= & \frac{\sum _{k=0}^{n}m(T^{-k}(T^{-1}(A)\cap \bigcup _{j=l}^{\infty }Y_{j}))}
{\sum _{k=0}^{n}m(T^{-k}(Y_{0}))}\\ 
& \leq  \frac{\sum _{k=0}^{n}m(T^{-(k+1)}(A\cap T(\bigcup _{j=l}^{\infty }Y_{j})))}
{\sum _{k=0}^{n}m(T^{-k}(Y_{0}))}\\ 
& \leq m_{n+1}(A\cap T(\bigcup _{j=l}^{\infty }Y_{j}))\cdot 
\frac{\sum _{k=0}^{n+1}m(T^{-k}(Y_{0}))}{\sum _{k=0}^{n}m(T^{-k}(Y_{0}))}\\ 
& \leq K_{i}\frac{m_{n+1}(Y_{i})}{m(Y_{i})}\cdot 
m(A\cap T(\bigcup _{j=l}^{\infty }Y_{j}))\cdot 
\frac{\sum _{k=0}^{n+1}m(T^{-k}(Y_{0}))}{\sum _{k=0}^{n}m(T^{-k}(Y_{0}))}, 
\end{align*}
where the last inequality sign was written because of (4) of Definition~\ref{d:mmmap} and
 since $A\subset Y_{i}.$ Since, the limit when $n\rightarrow \infty $ at last quotient is $1$, 
 we get that 
$$ l_{B}\left((m_{n}(T^{-1}(A)\cap \bigcup _{j=l}^{\infty }Y_{j}))_{n=1}^{\infty }\right) 
\leq \frac{K_{i}\mu (Y_{i})}{m(Y_{i})} m(T(\bigcup _{j=l}^{\infty }Y_{j})).$$
Hence, in virtue of (6) of Definition~\ref{d:mmmap}, 
$$\lim _{l\rightarrow \infty }
l_{B}\left((m_{n}(T^{-1}(A)\cap \bigcup _{j=l}^{\infty }Y_{j}))_{n=1}^{\infty } \right) 
\leq \frac{K_{i}\mu (Y_{i})}{m(Y_{i})}\lim _{l\rightarrow \infty }m(T(\bigcup _{j=l}^{\infty }Y_{j}))=0.$$
Thus, it follows from Lemma~\ref{l:mueqpf}, and as $A\subset Y_{i}$, that 
$$\mu (T^{-1}(A))=l_{B}((m_{n}(T^{-1}(A)))_{n=1}^{\infty })=l_{B}((m_{n}(A))_{n=1}^{\infty })
=\mu (A).$$ 
For an arbitrary $A\subset X$, write $A=\bigcup _{j=0}^{\infty }A\cap Y_{j}$ and 
observe that 
$$\mu (T^{-1}(A))=\mu (\bigcup _{j=0}^{\infty }T^{-1}(A\cap Y_{j}))
=\sum _{j=0}^{\infty }\mu (T^{-1}(A\cap Y_{j}))=
\sum _{j=0}^{\infty }\mu (A\cap Y_{j})=\mu (A).$$
We are done. 
\end{proof}

We now give the proof of Theorem~\ref{t1h75}.

\ 

\noindent {\bf Proof of Theorem~\ref{t1h75}:} 
Combining Lemmas~ \ref{l:muxjfin}, \ref{l:musfmeas}, 
\ref{l:mueqpf}, and \ref{l:mutinv}, 
we obtain the statement of Theorem~\ref{t1h75}. 
We are done. 
\qed 

\

\fr Applying Theorem~\ref{t1h75} we shall prove Theorem~\ref{t4h65}.

\

\fr{\bf Proof of Theorem~\ref{t4h65}.} Since the topological support of $\mh$ is equal to
the Julia set $J(\tf)$ and since, by Lemma~\ref{l1h63}, $\PCV(\tf)$ is a nowhere dense 
subset of $J(\tf)$, we have $\mh(\PCV(\tf))<1$. Since the set $\PCV(\tf)$ is forward invariant
under $\tf$, it follows from ergodicity and conservativity of $\mh$ (see Lemma~\ref{l2h73})
that $\mh(\PCV(\tf))=0$. Therefore, in virtue of Lemma~\ref{l1h68}
\beq\label{1h81}
\mh(p_2^{-1}(p_2(\PCV(\tf))))=0.
\eeq
Now, for every $z\in J(G)\sms p_2(\PCV(\tf))$ take $r_z>0$ such that
$J(G)\cap B(z,2r_z)\sbt \oc 
\sms p_2(\PCV(\tf))$. Since $J(G)\sms p_2(\PCV(\tf))$ is a separable
metric space, Lindel\"of's Theorem 
yields the existence of a countable set $\{z_j\}_{j=0}^\infty\sbt
J(G)\sms p_2(\PCV(\tf))$ such that 
$$
\bu_{j=0}^\infty B(z_j,r_{z_{j}})\spt J(G)\sms p_2(\PCV(\tf)).
$$
Set 
$
A_j:=p_2^{-1}(B(z_j,r_j)).
$
Verifying the conditions of 
Definition~\ref{d:mmmap} 
(with $X=J(\tf), T=\tf, m=\mh$, $X_{j}=A_{j}$), 
$\tf $ is nonsingular because of Corollary~\ref{c3h61} and $h$-conformality of $\mh.$   
We immediately see that condition (1) is
satisfied, that (2) holds because of (\ref{1h81}), and that (3) holds
because of $h$-conformality 
of $\mh$ and topological exactness of the map $\tf:J(\tf)\to
J(\tf)$. Condition (5) follows  
directly from ergodicity and conservativity of the measure
$\mh$. Condition (6) follows since  
$\tilde{f}: J(\tilde{f})\rightarrow J(\tilde{f})$ is finite-to-one
(see Remark~\ref{r:mmmapf1}).   
Let us prove condition (4).
Fix $j\ge 1$ and two arbitrary Borel sets $A,B\sbt A_j$ with
$\mh(A),\mh(B)>0$. Since $B(z_j,2r_{z_j}) 
\cap p_2(\PCV(\tf))=\es$, for all $n\ge 0$ all continuous inverse branches
$$
\{\tf_*^{-n}:p_2^{-1}\(B(z_j,2r_{z_j})\)\to\Sg_u\times\oc\}_{*\in I_n}
$$
of $\tilde{f}^{n}$ 
are well-defined, where $I_{n}=\{ 1,\ldots ,u\} ^{n}$, 
and because of Koebe's Distortion Theorem and $h$-conformality of the measure 
$\mh$, we have
$$
\aligned
\mh\circ \tf^{-n}(A)
&=   \mh\lt(\bu_{*\in I_n}\tf_*^{-n}(A)\rt)
 =   \sum_{*\in I_n}\mh\(\tf_*^{-n}(A)\) \\
&\le \sum_{*\in I_n}K^h|(\tilde{f}_*^{-n})'(\tau,z_j)|^h\mh(A) \\
&=   K^{2h}{\mh(A)\over \mh(B)}\sum_{*\in I_n}K^{-h}|(\tilde{f}_*^{-n})'(\tau,z_j)|^h\mh(B) \\
&\le K^{2h}{\mh(A)\over \mh(B)}\sum_{*\in I_n}\mh\(\tf_*^{-n}(B)\) \\
&=   K^{2h}{\mh(A)\over \mh(B)}\mh\lt(\bu_{*\in I_n}\tf_*^{-n}(B)\rt) \\
&=   K^{2h}\mh\circ \tf^{-n}(B){\mh(A)\over \mh(B)},
\endaligned
$$
where $\tau$ is an arbitrary element of $\Sg_u$. Hence,
$$
{\mh\circ \tf^{-n}(A)\over\mh\circ \tf^{-n}(B)}
\le K^{2h}{\mh(A)\over \mh(B)},
$$
and consequently, condition (4) of Definition~\ref{d:mmmap} 
is satisfied. Therefore, Theorem~\ref{t1h75} produces 
a Borel $\sg$-finite $\tf$-invariant measure $\mu$ on $J(\tf)$, equivalent to $\mh$. 

Now, let us show that the measure $\mu$ is finite. Indeed, 
by Theorem~\ref{t1h4}, there exists a $\delta >0$ such that 
for all $g\in G^{\ast }$ and for all $x\in J(G)$, 
every connected component $W$ of $g^{-1}(B(x,\delta ))$ satisfies 
that $\mbox{diam}(W)<\gamma $ and that $W$ is simply connected.  
Cover $p_2(\PCV(\tf))$ with finitely many open balls $\{B(z,\d):z\in F\}$, 
where $F$ is some finite subset of $p_2(\PCV(\tf))$. 
for all $j\ge 1$. Since $J(G)\sms\bu_{z\in F}B(z,\d)$ is covered by finitely many balls $B(z_j,r_{z_j})$,
$j\ge 1$, it therefore suffices to show that $\mu(p_{2}^{-1}(B(z,\d)))<+\infty$ for all $z\in F$. So,
fix $z\in F$. Since $z\in p_2(\PCV(\tf))$, there thus exists $k\ge 1$ such that 
$B(z_k,r_{z_k})\sbt B(z,\d)$. 
By Lemma~\ref{l:muxjfin} and the formula (\ref{eq:muequ}) of Theorem~\ref{t1h75}, 
it therefore suffices to show that
\beq\label{2h85}
\limsup_{n\to\infty}{\mh\(\tf^{-n}(p_2^{-1}(B(z,\d)))\)\over \mh(\tf^{-n}(A_k))}
<+\infty.
\eeq
In order to do this let for every $\tau\in\{1,2,\ld,s\}^n$, the symbol $\Ga_\tau$ denote the 
collection of all connected components of $f_\tau^{-1}(B(z,\d))$. It follows from
Theorem~\ref{t1h61}, Lemma~\ref{lku2.19h9} and \cite[Corollary 1.9]{hiroki2} 
that for every $V\in \Ga_\tau$, we have
\beq\label{1h85}
\aligned
\mh([\tau]\times V)
&\le m_h(V)
\le C\diam^h(V)
\le C\Ga^{-h}\lt({\diam(B(z,\d))\over \diam(B(z_k,r_{z_k}))}\rt)^h\diam^h(V_k) \\
&=  C(\d r_{z_k}^{-1}\Ga^{-1})^h\diam^h(V_k),
\endaligned
\eeq
where $C>0$ is a constant independent of $n$ and $\tau $, 
$V_k$ is a connected component of $f_\tau^{-1}(B(z_k,r_{z_k}))$ contained in $V$, 
and $\Gamma $ is the constant in Lemma~\ref{lku2.19h9}. But,
from conformality of the measure $\mh$ and from the fact that $V_k=f_{\tau*}^{-1}(B(z_k,r_{z_k}))$,
where $f_{\tau*}^{-1}:B(z_k,2r_{z_k})\to\oc$ is an analytic inverse branch of $f_\tau$, we 
see that
$$
\aligned
\mh([\tau]\times  V_k)
&\ge K^{-h}|(f_{\tau*}^{-1})'(z_k)|^h\mh(A_k)
 \ge K^{-h}\lt(K^{-1}{\diam(V_k)\over 2r_{z_k}}\rt)^h\mh(A_k) \\
&=   (2K^2r_{z_k})^{-h}\diam^h(V_k)\mh(A_k).
\endaligned
$$
Combining this with (\ref{1h85}) we get that
$$
\mh([\tau]\times V)
\le C(2K^2\d\Ga^{-1})^h(\mh(A_k))^{-1}\mh([\tau]\times V_k)
$$
Therefore,
$$
\aligned
\mh\(\tf^{-n}(p_2^{-1}(B(z,\d)))\)
&=   \sum_{|\tau|=n}\sum_{V\in\Ga_\tau}\mh([\tau]\times V) \\
&\le C(2K^2\d\Ga^{-1})^h(\mh(A_k))^{-1}\sum_{|\tau|=n}\sum_{V\in\Ga_\tau}\mh([\tau]\times V_k) \\
&\le C(2K^2\d\Ga^{-1})^h(\mh(A_k))^{-1}\mh\(\tf^{-n}(A_k)\).
\endaligned
$$
Thus, the upper limit in (\ref{2h85}) is bounded above by $C(2K^2\d\Ga^{-1})^h(\mh(A_k))^{-1}
<+\infty$, and finiteness of the measure $\mu$ is proved. 

Dividing $\mu$ by $\mu(J(\tf))$, we may assume without loss of generality that $\mu$ is a 
probability measure. Since for every Borel set $F\sbt J(\tf)$ the sequence $(\mu(\tf^n(F)))_{n=1}^\infty$
is (weakly) increasing, the metric exactness of $\mu$ follows from weak metrical exactness
of $\mh$ (Lemma~\ref{l2h73}) and the fact that $\mu$ and $\mh$ are equivalent. Since, by
metrical exactness, $\mu$ is ergodic, it is a unique Borel probability measure absolutely
continuous with respect to $\tilde{m}_{h}$. The proof is complete. \endpf

\section{Examples}
\label{s:Ex}

In this section, we give some examples of semi-hyperbolic rational semigroups 
with nice open set condition. 

\begin{ex}[\cite{hiroki2, hiroki4}]
\label{oscshex1}
Let $f_{1}(z)=z^{2}+2$, $f_{2}(z)=z^{2}-2$, and $f=(f_{1},f_{2}).$  
Let $G=\langle f_{1},f_{2}\rangle.$  Moreover, 
let $U:=\{ z\in \C \mid |z|<2.\} .$ 
Then, $G$ is semi-hyperbolic but not hyperbolic (\cite[Example 5.8]{hiroki2}). 
Moreover, $G$ satisfies the nice open set condition with $U.$ Since $J(G)\subset f_{1}^{-1}(\ov{U})\cup f_{2}^{-1}(\ov{U})
\subsetneqq \ov{U}$, \cite[Theorem 1.25]{hiroki4} implies that 
$J(G)$ is porous and $\HD(J(G))<2.$ 
Moreover, by Theorem~\ref{Theorem A}, 
we have $h(f)=\HD(J(G))=\PD(J(G))=\BD(J(G)).$ 
Furthermore, 
$f_{1}^{-1}(\ov{U})\cap f_{2}^{-1}(\ov{U})\neq \emptyset .$ 
See figure~\ref{fig:z2+2z2-2-1}. 
\begin{figure}[htbp]
\caption{The Julia set of 
$\langle f_{1},f_{2}\rangle$, 
where $f_{1}(z)=z^{2}+2, f_{2}(z)=z^{2}-2$.}
\includegraphics[width=3cm,width=3cm]{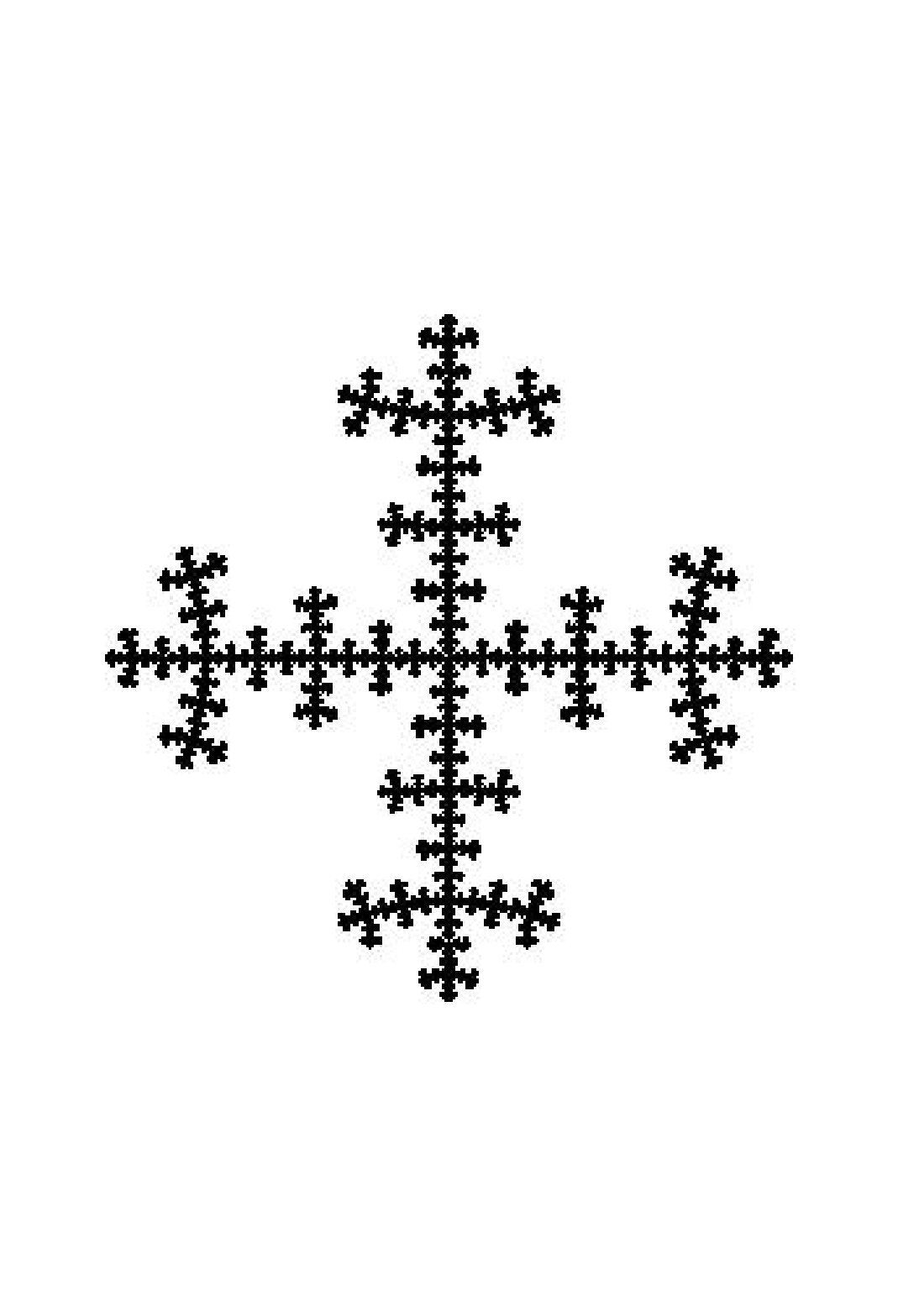}
\label{fig:z2+2z2-2-1}
\end{figure}

\end{ex}
\begin{prop} 
\label{semihyposcexprop}
(See \cite{SdpbpI, hiroki5}) 
Let $f_{1}$ be a semi-hyperbolic polynomial 
with $\deg (f_{1})\geq 2$ such that 
$J(f_{1})$ is connected. 
Let $K(f_{1})$ be the filled-in Julia set of $f_{1}$ and 
suppose that {\em int}$K(f_{1})$ is not empty. 
Let $b\in \mbox{{\em int}}K(f_{1})$ be a point. 
Let $d$ be a positive integer such that 
$d\geq 2.$ Suppose that $(\deg (f_{1}),d)\neq (2,2).$ 
Then, there exists a number $c>0$ such that 
for each $\l \in \{ \l\in \Bbb{C}: 0<|\l |<c\} $, 
setting $f_{\l }=(f_{\l ,1},f_{\l ,2})=
(f_{1},\l (z-b)^{d}+b )$ and $G_{\l }:= \langle f_{1},f_{\l, 2}\rangle $, 
 we have that 
$G_{\l }$ is semi-hyperbolic and   
$f_{\l }$ satisfies the nice open set condition with 
an open set $U_{\l }$, 
$J(G_{\l } )$ is porous,  
$\HD(J(G_{\l }))=h(f_{\l })<2$, and $P(G_{\l })
\setminus \{ \infty \} $ is bounded in $\Bbb{C}.$  
\end{prop}
{\sl Proof.} 
We will follow the argument in \cite{SdpbpI, hiroki5}. 
Conjugating $f_{1}$ by a M\"{o}bius transformation, 
we may assume that $b=0$ and the coefficient 
of the highest degree term of $f_{1}$ is equal to $1.$ 

Let $r>0$ be a 
 number such that $\overline{B(0,r)}\subset \mbox{int}K(f_{1}).$ 
We set $d_{1}:=\deg (f_{1}).$ 
 Let $\alpha >0$ be a number. 
Since $d\geq 2$ and $(d,d_{1})\neq (2,2)$, 
it is easy to see that 
$(\frac{r}{\alpha })^{\frac{1}{d}}>
2\left(2(\frac{1}{\alpha })
^{\frac{1}{d-1}}\right)^{\frac{1}{d_{1}}}
$ if and only if 
\begin{equation}
\label{Contproppfeq1}
\log \alpha <
\frac{d(d-1)d_{1}}{d+d_{1}-d_{1}d}
( \log 2-\frac{1}{d_{1}}\log \frac{1}{2}-\frac{1}{d}\log r) .
\end{equation} 
We set 
\begin{equation}
\label{Contproppfeq2}
c_{0}:=\exp \left(\frac{d(d-1)d_{1}}{d+d_{1}-d_{1}d}
( \log 2-\frac{1}{d_{1}}\log \frac{1}{2}-\frac{1}{d}\log r) \right)
\in (0,\infty ).
\end{equation}
 
Let $0<c<c_{0}$ be a small number and let $\l \in \Bbb{C} $ 
be a number with $0<|\l |<c.$ 
Put $f_{\l ,2}(z)=\l z^{d}.$ 
Then, we obtain $K(f_{\l ,2})=\{ z\in \Bbb{C} \mid 
|z|\leq (\frac{1}{|\l |})^{\frac{1}{d-1}}\} $ and 
$$f_{\l,2}^{-1}(\{ z\in \Bbb{C} \mid  |z|=r\} )=
\{ z\in \Bbb{C} \mid |z|=(\frac{r}{|\l |})^{\frac{1}{d}}\} .$$
Let 
$D_{\l }:=\overline{B(0,2(\frac{1}{|\l |})^{\frac{1}{d-1}})}.$ 
Since $f_{1}(z)=z^{d_{1}}(1+o(1))\ (z\rightarrow \infty )$,  
it follows that if $c$ is small enough, then 
for any $\l \in \Bbb{C} $ with $0<|\l |<c$, 
$$f_{1}^{-1}(D_{\l })\subset 
\left\{ z\in \Bbb{C} \mid 
|z|\leq 2\left( 2(\frac{1}{|\l |})^{\frac{1}{d-1}}\right) 
^{\frac{1}{d_{1}}}\right\} .$$  
This implies that  
\begin{equation}
\label{Contproppfeq3}
f_{1}^{-1}(D_{\l })\subset f_{\l ,2}^{-1}(\{ z\in \Bbb{C} \mid |z|<r\} ).
\end{equation} 
Hence, 
setting $U_{\l }:=\mbox{int}K(f_{\l ,2})\setminus K(f_{1})$, 
$f_{1}^{-1}(U_{\l })\cup f_{\l ,2}^{-1}(U_{\l })\subset U_{\l }$ and 
$f_{1}^{-1}(U_{\l })\cap f_{\l, 2}^{-1}(U_{\l })=\emptyset $. 
Furthermore, since $f_{1}$ is semi-hyperbolic, 
$\oc \setminus K(f_{1})$ is a John domain (see \cite{CJY}). 
Hence, $U_{\l }$ satisfies (osc3).  
Therefore, $G_{\l }$ satisfies the nice open set condition 
with $U_{\l }.$ 
We have $J(G_{\l })\subset \overline{U_{\l }}
\subset K(f_{\l ,2})\setminus \mbox{int}K(f_{1}). $ 
In particular, int$K(f_{1})\cup (\oc \setminus K(f_{\l ,2}))\subset 
F(G_{\l }).$ 
Furthermore, (\ref{Contproppfeq3}) implies that 
$f_{\l ,2}(K(f_{1}))\subset \mbox{int}K(f_{1}).$ 
Thus, 
we have $P(G_{\l })\setminus \{ \infty \} 
= \cup _{g\in G_{\l }^{\ast }}
g(CV^{\ast }(f_{1})\cup CV^{\ast }(f_{\l, 2}))
\subset 
K(f_{1})$, 
where $CV^{\ast }(\cdot )$ denotes the set of 
all critical values in $\Bbb{C}.$ Hence, 
$P(G_{\l })\setminus \{ \infty \} $ is bounded in $\Bbb{C}.$ 
Since $f_{1}$ is semi-hyperbolic, 
there exist an $N\in \N $ and a $\delta _{1}>0$ such that 
for each $x\in J(f_{1})$ and for each $n\in \N $, 
$\deg (f_{1}^{n}:V\rightarrow B(x,\delta _{1}))\leq N$ for each 
connected component $V$ of $f_{1}^{-n}(B(x,\delta _{1}).$ 
Moreover, $f_{\l, 2}^{-1}(J(f_{1}))\cap K(h_{1})=\emptyset $ and 
so $f_{\l ,2}^{-1}(J(f_{1}))\subset \oc \setminus P(G_{\l }). $ 
From these arguments, 
it follows that there exists a $0<\delta _{2} (<\delta _{1})$ such that 
for each $x\in J(f_{1})$ and each $g\in G_{\l }$, 
$\deg (g:V\rightarrow B(x,\delta _{2}))\leq N$ for each connected component $V$ of 
$g^{-1}(B(x,\delta _{2})).$ 
Since $P(G_{\l })\setminus \{ \infty \} \subset K(f_{1})$ again, 
we obtain that there exists a $0<\delta _{3} (<\delta _{2} )$ such that 
for each $x\in J(G_{\l })$ and each $g\in G_{\l }$, 
$\deg (g:V\rightarrow B(x,\delta _{3}))\leq N$ for each connected component 
$V$ of $g^{-1}(B(x,\delta _{3})).$ Thus, $G_{\l }$ is semi-hyperbolic. 
Since $J(G_{\l })\subset f_{1}^{-1}(\ov{U_{\l }})\cup f_{\l ,2}^{-1}(\ov{U_{\l }})
\subsetneqq \ov{U_{\l }}$, \cite{hiroki4} implies that $J(G_{\l })$ is porous and 
$\HD (J(G_{\l }))<2.$ Moreover, by Theorem~\ref{Theorem A}, we have 
$h(f_{\l })=\HD(J(G_{\l })).$ 
We are done. 
\qed 
\begin{figure}
\caption{The Julia set of 
$\langle f_{1}^{2},f_{2}^{2}\rangle$, 
where $f_{1}(z)=z^{2}-1, f_{2}(z)=z^{2}/4$.}
\includegraphics[width=3cm,width=3cm]{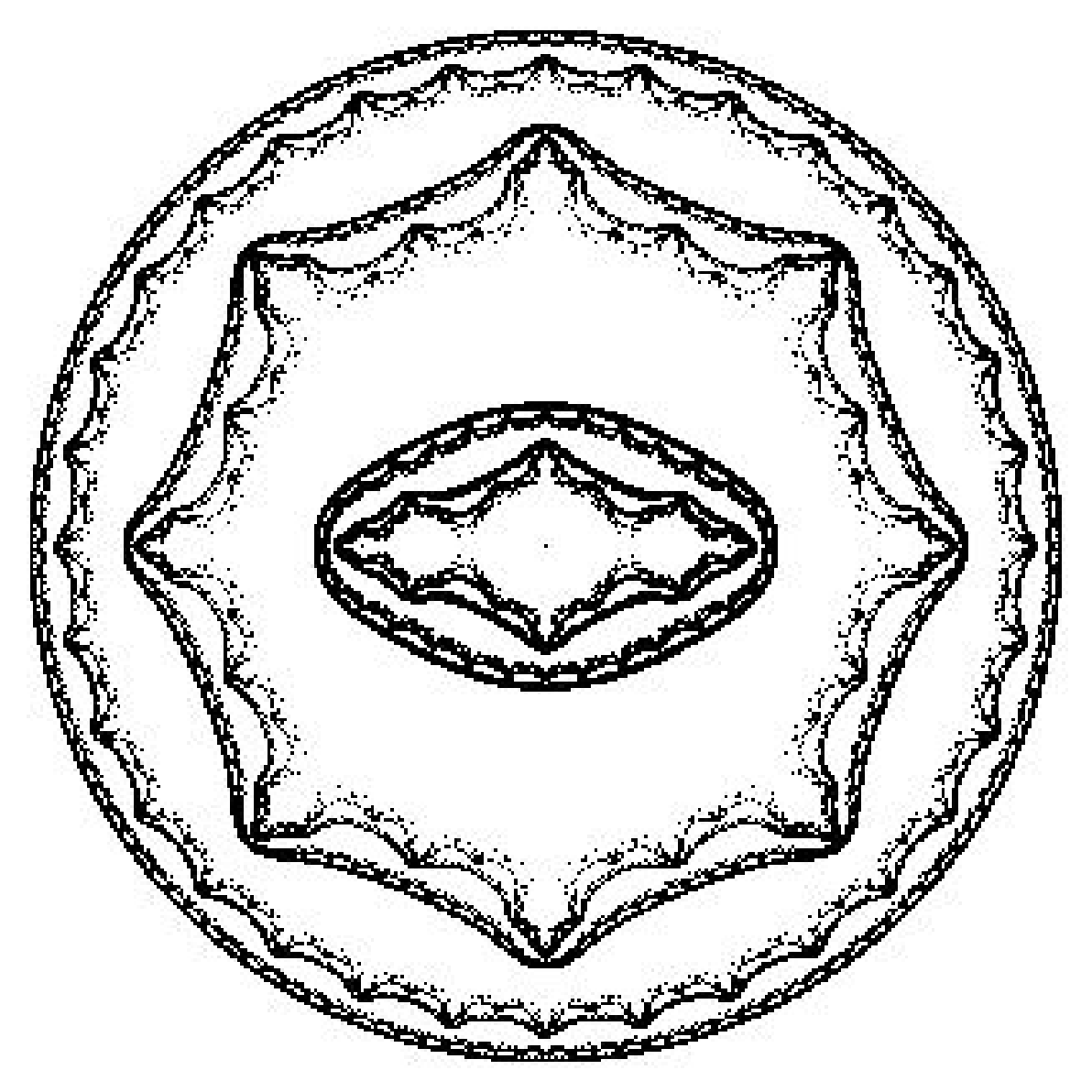}
\label{fig:z2-1z2/4}
\end{figure}

\end{document}